\documentclass[11pt]{article}
\usepackage{amsmath,amssymb,amsthm,eucal, mathrsfs}
\usepackage{color}
\usepackage[all]{xy}

\title{Nonunital spectral triples and metric completeness in unbounded $KK$-theory}

\author{Bram Mesland\S$^*$, Adam Rennie\ddag
\thanks{email: 
\texttt{b.mesland@warwick.ac.uk}, \texttt{renniea@uow.edu.au}
}
\\[3pt]
\S Mathematics Institute,
Zeeman Building,
University of Warwick\\
Coventry CV4 7AL, UK
\\[3pt]
\ddag School of Mathematics and Applied Statistics, University of Wollongong\\
Wollongong, Australia\\
}

\advance\topmargin by -\headheight
\advance\topmargin by -\headsep
\evensidemargin=\oddsidemargin
\makeatletter
\def\section{\@startsection{section}{1}{\z@}{-3.5ex plus -1ex minus
  -.2ex}{2.3ex plus .2ex}{\large\bf}}
\def\subsection{\@startsection{subsection}{2}{\z@}{-3.25ex plus -1ex
  minus -.2ex}{1.5ex plus .2ex}{\normalsize\bf}}
\makeatother

\numberwithin{equation}{section} 

\theoremstyle{plain} 
\newtheorem{thm}{Theorem}[section]
\newtheorem{lemma}[thm]{Lemma}
\newtheorem{prop}[thm]{Proposition}
\newtheorem{corl}[thm]{Corollary}

\theoremstyle{definition} 
\newtheorem{defn}[thm]{Definition}

\theoremstyle{remark} 
\newtheorem{rmk}[thm]{Remark}

\DeclareMathOperator{\Dom}{Dom}   
\DeclareMathOperator{\End}{End}   
\DeclareMathOperator{\Id}{Id}     

\newcommand{\flip}{U}  
\newcommand{\Lip}{\textnormal{Lip}}
\newcommand{\A}{\mathcal{A}}  
\newcommand{\B}{\mathcal{B}}  
\newcommand{\C}{\mathbb{C}}   
\newcommand{\D}{\mathcal{D}}  
\renewcommand{\d}{\mathrm{d}} 
\newcommand{\E}{\mathcal{E}}  
\renewcommand{\H}{\mathcal{H}}  
\newcommand{\J}{\mathcal{J}}  
\newcommand{\K}{\mathcal{K}}  
\newcommand{\N}{\mathbb{N}}   
\newcommand{\ox}{\otimes}     
\newcommand{\R}{\mathbb{R}}   
\newcommand{\Z}{\hat{\mathbb{Z}}}   
\newcommand{\im}{\textnormal{im }}
\newcommand{\Fin}{\textnormal{Fin}}
\newcommand{\stroke}{\mathbin|}     

\def\pairL_#1(#2|#3){{}_{#1}(#2\stroke#3)} 
\def\pairR(#1|#2)_#3{(#1\stroke#2)_{#3}} 
\def\scal<#1|#2>{\langle#1\stroke#2\rangle} 

\newbox\ncintdbox \newbox\ncinttbox 
	\setbox0=\hbox{$-$}
	\setbox2=\hbox{$\displaystyle\int$}
	\setbox\ncintdbox=\hbox{\rlap{\hbox
		to \wd2{\hskip-.125em \box2\relax\hfil}}\box0\kern.1em}
	\setbox0=\hbox{$\vcenter{\hrule width 4pt}$}
	\setbox2=\hbox{$\textstyle\int$}
	\setbox\ncinttbox=\hbox{\rlap{\hbox
		to \wd2{\hskip-.175em \box2\relax\hfil}}\box0\kern.1em}

\newcommand{\Nil}{\textnormal{Nil }}
\newcommand{\bB}{\mathbb{B}}
\newcommand{\hotimes}{\tilde{\otimes}}
\newcommand{\kK}{\mathbb{K}}
\newcommand{\M}{\mathbb{M}}
\renewcommand{\epsilon}{\varepsilon}

\newcommand{\diag}{\textnormal{diag}}

\newcommand{\G}{\textnormal{G}}
\newcommand{\cb}{\textnormal{cb}}
\newcommand{\e}{\varepsilon}

\begin{document}

\maketitle

\vspace{-2pc}

\begin{abstract}
By considering the general properties of approximate units in differentiable algebras,
we are able to present a unified approach to characterising completeness of spectral metric spaces, 
existence of connections on modules, and the lifting of Kasparov products to the
unbounded category. In particular, by strengthening Kasparov's technical theorem, 
we show that given any two composable $KK$-classes,
we can find unbounded representatives whose product can be constructed to
yield an unbounded representative of the Kasparov product.
\end{abstract}

\tableofcontents

\parskip=5pt
\parindent=0pt

\addtocontents{toc}{\vspace{-1pc}}

\section*{Introduction}
\label{sec:intro}
In this paper we analyse the completeness of metric spaces associated to (nonunital) spectral
triples, the existence of 
differentiable structures and connections on modules over algebras associated to spectral triples,
and we prove that Kasparov products can be lifted to the unbounded setting in a very strong sense. 
The precise conditions under which such liftings exist have become important due to recent applications
of the unbounded Kasparov product, \cite{BCR,BMS,KaLe,LRV}. 

These seemingly disparate topics are in fact related by the systematic use of approximate units
for differentiable algebras, introduced below.
Technical advances in approximate units have often heralded  conceptual advances in 
operator algebras and noncommutative geometry. Pertinent examples include
quasicentrality \cite{AP},
Higson's proof of Kasparov's technical theorem, \cite{H}, and the early approaches to 
summability for nonunital spectral triples, \cite{GGISV,RenSum}.

In this paper we refine the notion of approximate unit further by looking at 
differentiable algebras of  spectral triples (or unbounded Kasparov modules). Given a separable
$C^*$-algebra $A$ with spectral triple $(A,\H,\D)$, we define a \emph{differentiable algebra} to be a 
separable $*$-subalgebra $\A$ with
$$
\A\subset\A_{\D}=\left\{a\in A:\,[\D,a]\ \mbox{is defined on }\Dom\D,\quad \Vert a\Vert_{\D}:=\left\Vert \begin{pmatrix}a & 0\\ [\D,a] & a\end{pmatrix}\right\Vert_\infty<\infty\right\},
$$
which is closed in the norm $\Vert\cdot\Vert_\D$ and dense in the $C^{*}$-algebra $A$. Here
$\Vert\cdot\Vert_\infty$ is the usual norm of operators on $\H\oplus\H$.
While we can always choose an approximate identity $(u_n)$ 
for $A$ consisting of elements of the prescribed dense subalgebra $\A$, the requirement
that $(u_n)$ be an approximate unit for $\A$ is much stronger, and
yields finer information.

For spectral metric spaces associated to spectral triples we obtain a characterisation of 
metric completeness in terms of the existence of an 
approximate unit $(u_{n})$ for $A$ whose Lipschitz norm is uniformly bounded in the sense that
$\sup_{n}\|[\D,u_{n}]\|_{\infty}<\infty.$ This extends
previous results of Latr\'emoli\`ere \cite{Lat} to unbounded metrics. By addressing 
completeness in
a way compatible  with \cite{Lat2}, we complement Latr\'emoli\`ere's
more refined
picture of unbounded spectral metric spaces.

In addition, we obtain stronger forms of metric completeness, characterised by 
the requirement  
that the `derivatives' $[\D,u_n]$ of the approximate unit converge to zero in norm. This property
corresponds to `topological infinity' being at infinite distance, and
reflects the behaviour of geodesically complete manifolds. 
We present examples illustrating this analogy in
Section 2. 

Beyond metric properties, we use completeness and approximate units  to 
describe a refinement of unbounded Kasparov modules and correspondences
for which the Kasparov product
can be explicitly constructed. Our main results then show 
that any pair of composable $KK$-classes 
have representatives which can be lifted to such modules. These results rely in an essential way
on the two notions of completeness we introduce in Sections 1 and 2, and 
generalise all previous such constructions. We now describe these results in more detail.

We begin Section 1 by establishing the basic concepts and notation that we will use 
throughout the paper, in particular (non-self-adjoint) operator algebras and their
approximate units. Then we prove a range of results
about bounded approximate units in operator algebras, which greatly extend known results about
contractive approximate units. A key result is Theorem~\ref{thm: idempotent} which says

{\em Let $\A$ be an operator algebra with bounded approximate unit $(u_{\lambda})$, and 
$\pi:\A\rightarrow \bB(\H)$ a cb-representation. Then $\pi(u_{\lambda})$ 
converges strongly, and hence weakly, to an idempotent $q\in \bB(\H)$ with the following properties:\\
1) for all $a\in\A$, $q\pi(a)=\pi(a)q=\pi(a)$;\\
2) $q\H=[\pi(\A)\H]$;\\
3) $(1-q)\H=\Nil\pi(\A)$;\\
4) $\|q\|\leq \|\pi\|\sup_{\lambda}\|u_{\lambda}\|$.}

The close relationships 
between the notions of approximate unit, unbounded multiplier and strictly positive element for 
differentiable algebras which one
would expect from the corresponding $C^*$-theory only hold when we have the strong form of
completeness, namely $[\D,u_n]\to 0$ in norm. Such approximate units may be normalised by 
replacing them by $\tilde{u}_{n}:=\frac{u_{n}}{\|u_{n}\|_{\D}}$ to obtain an approximate 
unit that is contractive, i.e. $ \|\tilde{u}_{n}\|_{\D}\leq 1$. The result
on which the rest of the paper relies is Theorem \ref{thm:equivs}
which says

\emph{Let $\D:\Dom \D\subset E_{B}\rightarrow E_{B}$ be self-adjoint and regular and $\A\subset\Lip(\D)$ 
a differentiable algebra such that $[AE_{B}]= E_{B}$. Then the following are equivalent:\\
1) there exists an increasing commutative approximate unit $(u_n)\subset\A$ 
with $\Vert[\D,u_n]\Vert_\infty\to 0$;\\
2) there exists a positive self-adjoint complete multiplier $c$ for $\A$;\\
3) there is a strictly positive element $h\in\A$ with 
$\im (\D\pm i)^{-1}h=\im h(\D\pm i)^{-1}$, and constant $C>0$ with $i[\D,h]\leq Ch^{2}$.}

Our characterisations of completeness are also essential ingredients in constructing useful
modules over differentiable algebras. By considering the behaviour of approximate units for
the finite rank operators on such modules, we are led to two classes of modules: projective
modules, and complete projective modules. These modules are characterised by the existence
of certain types of approximate units for their compact endomorphisms, and it is in this setting that we can systematically relate frames, splittings
of the Cuntz-Quillen sequence, and existence of connections. This is discussed in Section 3. In 
particular we gain tools for studying self-adjointness of operators
arising when we take Kasparov products, and Theorem \ref{selfreg} proves

\emph{Let $\E_\B$ be a complete projective module for the unbounded Kasparov module
$(\B, F_{C},\D)$. 
Then the densely defined symmetric 
operator $1\otimes_{\nabla}\D$ on $\E\hotimes_\B F_C$ is self-adjoint and regular.}

The
proof of the self-adjointness of the operator $1\otimes_\nabla\D$
relies on the local-global principle of Kaad-Lesch, \cite{KaLe2}, and
on the completeness of the module $\E_\B$. The resulting argument 
is in the spirit of self-adjointness proofs
for Dirac operators on complete manifolds.

To prove our results about representing Kasparov classes 
as composable unbounded Kasparov modules, we
extend the notion of quasi-central approximate units to 
differentiable algebras. 

We obtain a 
novel, strong form of quasicentrality in the general context of 
non-self-adjoint operator algebras in Theorem \ref{Ped}. Our results concerning existence of 
such approximate units are new even for $C^*$-algebras. 
This study culminates in a refinement of 
Kasparov's technical theorem for differentiable algebras in Theorem \ref{thm: technical}. 
Both the statement and the proof are in the same spirit as 
Higson's version, \cite{H}.

With this tool in hand, we show that given an {\em arbitrary} pair of composable Kasparov classes,
we can find unbounded Kasparov modules which represent these classes and whose 
product can be {\em constructed} in the unbounded setting. This is done by 
associating to a  bounded Kasparov module a correspondence in a slightly broader
sense than was used in \cite{BMS, KaLe, Mes}. Earlier forms of this lifting 
construction were first considered in \cite{BJ} to handle external products, 
and later in \cite{Kucerovsky2}, in the context of internal products. 
We prove successively stronger lifting results 
in Theorem \ref{thm:firstlift}, Proposition \ref{prop: Kasparovlift}, Theorem \ref{Liplift} and
Proposition \ref{singleliplift}, culminating in Theorem
\ref{thm:best-lift} which says

\emph{Let $A,B,C$ be separable $C^{*}$-algebras, $x\in KK(A,B)$ and 
$y\in KK(B,C)$. There exists an unbounded Kasparov module $(\B,E_{C},T)$ 
representing $y$ and a correspondence $(\A,\E_\B,S,\nabla)$ for 
$(\B,E_{C},T)$ representing $x$. Consequently
$(\A,E_{B}\hotimes_{B}E_{C}, \mathcal{S}\otimes 1 + 1\otimes_{\nabla}T)$ 
represents the Kasparov product $x\otimes_B y$.}

This result can also be interpreted as an alternative proof of the existence of the Kasparov product.
By proving the existence of unbounded representatives of a very special form, the product can be 
constructed via an explicit algebraic formula. To lift Kasparov modules to unbounded
representatives, we prove the
existence of, equivalently, either a frame, an approximate unit or a strictly positive element
possessing certain properties.  Given such a frame, 
connections and so products become explicitly computable.

The unbounded Kasparov modules we construct are `complete' in the strong metric sense, so every
$KK$-class has such a representative. On the other hand, not every unbounded representative
of a Kasparov class is `complete'. For instance Kaad's example of the half-line, \cite{Kaadabsorption}, 
or 
Baum, Douglas and Taylor's examples, \cite{BDT}, 
from manifolds with boundary will not satisfy our completeness
requirements. If we take the associated $KK$-class of the Dirac operator on a
manifold with boundary, and then lift a bounded representative  
to a complete unbounded module, we will have seriously
altered the geometry.

{\bf Acknowledgements} This work has profited from 
discussions with S. Brain, I. Forsyth, V. Gayral, M. Goffeng, N. Higson, J. Kaad, 
M. Lesch, A. Sims and W. D. van Suijlekom. 
AR was supported by the Australian Research Council. 
BM was supported by EPSRC grant EP/J006580/2. 
The authors thank the Hausdorff Institute for Mathematics 
for its hospitality and support during the (northern) fall of 2014, as 
well as the University of Wollongong for hosting BM on several occasions in 2012-2014.

\section{Approximate units and unbounded multipliers for operator algebras}

This section begins by recalling some of the basic elements of operator algebras we require. Then
we address the definitions of, and relationships between, approximate units, strictly positive
elements and unbounded multipliers. For $C^*$-algebras these notions are closely related
due to the connection between the norm and the spectrum. Here we must work somewhat
harder, but the outcome is a systematic way of capturing the notion of metric
completeness, and this plays a significant r\^ole throughout the rest of the paper.

\subsection{Operator algebras and differentiable algebras}
By an \emph{operator algebra} we will mean a concrete operator algebra, 
that is, a closed subalgebra $\mathcal{A}\subset B$ of some $C^{*}$-algebra $B$. 
By representing $B$ isometrically on a Hilbert space $\H$, we can always assume that
$B=\mathbb{B}(\H)$. An \emph{operator $*$-algebra} \cite{KaLe, Mes} is an operator algebra 
$\A\subset \mathbb{B}(\H)$ with a completely bounded involution $*:\A\rightarrow \A$. 
This involution will in general not coincide with the involution of the ambient 
$C^{*}$-algebra $\mathbb{B}(\H)$, unless $\A$ itself is actually a $C^{*}$-algebra. 

There are two $C^{*}$-algebras canonically associated to a 
concrete operator $*$-algebra $\A\subset B$. The first is the 
\emph{enveloping $C^{*}$-algebra} $C^{*}(\A)$, defined to be 
the smallest $C^{*}$-subalgebra of $B$ containing $\A$. 
In fact this $C^{*}$-algebra depends only on the inclusion 
$\A\subset B$, i.e. on the structure of $\A$ as a concrete operator algebra. 
The second $C^{*}$-algebra is the $C^{*}$-\emph{closure}  $A$, 
constructed from viewing $\A$ as a Banach $*$-algebra, and completing in 
the $C^{*}$-norm coming from the square root of the spectral radius of 
$a^{*}a$. The two $C^{*}$-algebras are almost always different, 
as $\A$ is always dense in $A$, but is usually not dense in $C^{*}(\A)$.

The main examples of operator algebras that we consider arise in the following setting.
Given an unbounded (even) $(A,B)$ Kasparov module 
$(A,E_{B},\D)$, we denote the algebra of adjointable 
operators on the $C^{*}$-module $E_{B}$ by $\End^{*}_{B}(E)$. The algebras $A$ and $B$
are (possibly trivially) $\mathbb{Z}_2$-graded, as is the module $E$, and all
commutators are $\mathbb{Z}_2$-graded. The grading operator 
on $E$ will
be denoted by $\gamma$ or $\gamma_E$ when needed, 
and we observe that if $B$ is non-trivially graded, 
$\gamma_E$ is {\em not} an adjointable operator, and 
$\gamma_E(eb)=\gamma_E(e)\gamma_B(b)$, where $\gamma_B$ is the grading on $B$. 
In addition we have the identities
\[
[\D,\pi(a)]=\D\pi(a)-\pi(\gamma_A(a))\D,\quad [\D,\pi(a)]^{*}
=-[\D,\pi(\gamma_A(a^{*}))], \quad \pi(\gamma_A(a))=\gamma_{E}\pi(a)\gamma_{E}.
\] 
See \cite{Blackadar} for more information. 
When no confusion can occur, we will just write $\gamma$ in all cases.

We realise the full Lipschitz algebra 
$$
\A_{\D}=\{a\in A:\,a\Dom\D\subset\Dom\D,\ [\D,a]\in\End^{*}_{B}(E)\}
$$ 
as an operator $*$-algebra via
\begin{equation}\label{eq:lip-norm}
\pi_\D:\A_{\D}\ni a\mapsto \pi_\D(a)
:=\begin{pmatrix} \pi(a) & 0\\ [\D,\pi(a)] & \pi(\gamma(a))\end{pmatrix}\in \End^{*}_{B}(E).
\end{equation}
Here $\pi:A\rightarrow \End^{*}_{B}(E)$ is the 
representation implicit in the Kasparov module 
$(A,E_{B},\D)$ and it will often be suppressed in 
the notation (as in the Introduction).  Throughout the paper we will 
denote $\flip=\big(\begin{smallmatrix} 0& -1 \\ 1 & 0\end{smallmatrix}\big)$. 
The involution is completely isometric in this case because
\[
\pi_{\D}(a^{*})=\flip^{*}
\begin{pmatrix} \pi(\gamma(a)) & 0\\ [\D,\pi(\gamma(a))] & \pi(a)\end{pmatrix}^{*}
\flip=\flip^{*} \pi_{\D}(\gamma(a))^{*}\flip, 
\]
cf. \cite[cf. Proposition 4.1.3]{Mes}.

We will call 
$\pi_{\D}$ the standard Lipschitz representation 
of $\A_{\D}$, and always consider $\A_{\D}$
to be topologised by the operator norm 
$\Vert \pi_\D(a)\Vert_\infty$ in this representation. 
If $A$ is not represented faithfully on $E_{B}$, 
the standard Lipschitz representation should be modified to
\[
a\mapsto a\oplus\pi_{\D}(a)\in A\oplus \End^{*}_{B}(E\oplus E).
\] 
We will always suppress this modification as it is inconsequential 
for our computations (in fact it only obscures them). 

A more general setting where this operator algebra structure can be 
discussed is symmetric spectral
triples or symmetric Kasparov modules, \cite[Section 3]{Hilsum}. 
Here we have a triple $(A,E_{B},\D)$ with $\D$ a (closed) symmetric regular operator such that the 
subalgebra $\A$ of $a\in A$
for which $a\cdot\Dom\D^*\subset \Dom\D$ and $[\D^*,a]$ is bounded is dense in $A$ and also
$a(1+\D^*\D)^{-1/2}$ is compact for all $a\in A$. 
\begin{rmk} 
\label{rmk:comm-adj}
In this case $[\D^*,a]^*=-[\D^*,a^*]$, initially defined
on $\Dom\D^*$, \cite[cf. Proposition 2.1]{FMR}. 
The associated quadratic forms coincide on $\Dom\D^*$ and the boundedness of $[\D^*,a]$
ensures then that the continuous extensions to $E_B$ coincide.
\end{rmk}
We can again use the representation
\begin{equation}
\pi_{\D}:\A_{\D}\ni a\mapsto \pi_{\D}(a)
:=\begin{pmatrix} a & 0\\ [\D^*,a] & \gamma(a)\end{pmatrix}\in \End_{B}^{*}(E\oplus E)
\label{eq:sym-lip-norm}
\end{equation}
to give $\A_{\D}$ the structure of an operator space.
More generally still, associated to a closed symmetric regular operator 
$\D$ on a $C^{*}$-module $E_{B}$ is the operator algebra
\begin{equation}
\Lip(\D):=\{T\in \End^{*}_{B}(E): T\Dom \D^*\subset \Dom \D,\  [\D^*,T]\in \End^{*}_{B}(E)\},
\label{eq:sob}
\end{equation}
the \emph{Lipschitz algebra} of $\D$, which is an operator algebra 
in the representation $\pi_{\D}$ (cf. \cite[Def. 4.1.1]{Mes}). By \cite[Sec. 4.2]{Mes}, 
$\Lip(\D)$ is spectral invariant in its $C^{*}$-closure, and hence stable 
under the holomorphic functional calculus. The same holds for any closed subalgebra of $\Lip(\D)$.
\begin{defn} 
\label{def:diff-alg}
Let $\D:\Dom \D\rightarrow E_{B}$ be a closed symmetric operator. 
A \emph{differentiable algebra} is a separable operator $*$-subalgebra 
$\A\subset \Lip(\D)$ which is closed in the operator space topology given by $\pi_{\D}$. 
By projecting onto the first component of $\pi_\D(\A)$,
the $C^{*}$-closure of a differentiable algebra $\A$ coincides with the closure of $\A$ 
viewed as a subalgebra of $\End_{B}^{*}(E)$, and is thus a $C^{*}$-algebra.
\end{defn}

\begin{rmk}
We will present examples at the end of Section \ref{sec:met} showing that 
for an unbounded Kasparov module $(A,E_B,\D)$, in general one
needs to choose algebras smaller than $\A_{\D}$ in order to apply the methods
that we develop in the remainder of the paper. Therefore we employ the 
notation $(\A,E_{B},\D)$ for unbounded Kasparov modules, where 
$\A\subset\A_{\D}$ is a closed subalgebra in the Lipschitz topology whose $C^*$-closure is $A$.
\end{rmk}
Operator spaces and modules play a central r\^{o}le in this paper 
and we now introduce some concepts of operator space theory 
that we will need. All operator spaces we encounter are 
\emph{concrete} operator spaces, that is, given explicitly 
as closed subspaces of $\mathbb{B}(\H)$ for some Hilbert 
space $\H$. For a comprehensive treatment of operator 
algebras and modules, see \cite{Blecherbook}.

The main feature of an operator space $X$ is the 
existence of canonical \emph{matrix norms}, that is, 
a norm $\|\cdot \|_{n}$ on $M_{n}(X)$ for each $n\in \N$. 
For a concrete operator space $X\subset \mathbb{B}(\H)$ 
these norms are obtained from the natural representation 
of $M_{n}(X)$ on $\mathbb{B}(\H^{n})$.

A linear map $\phi:X\rightarrow Y$ between operator spaces 
$X$ and $Y$ is \emph{completely bounded} if
\[
\|\phi\|_{\cb}:=\sup_{n}\{\sup \|\phi(x_{ij})\|_{M_{n}(Y)}:\|(x_{ij})\|_{M_{n}(X)}\leq 1\}<\infty,
\]
and we refer to $\|\phi\|_{\cb}$ as the \emph{cb-norm} of $\phi$. The map 
$\phi$ is \emph{completely contractive} if $\|\phi\|_{\cb}\leq 1$.

If we are given an operator algebra $\A\subset \mathbb{B}(\H)$ 
and an operator space $X\subset \mathbb{B}(\H)$, we say that $X$ is a 
\emph{concrete left  operator $\A$-module} if $\A\cdot X\subset X$. 
Here $\cdot$ denotes the usual operator multiplication in $\bB(\H)$. 
Right modules are defined similarly.

The \emph{Haagerup tensor product} of operator spaces 
$X$ and $Y$ is defined to be the completion of the 
algebraic tensor product $X\otimes Y$ over the complex numbers, in the \emph{Haagerup norm}
\begin{equation}
\|z\|^{2}_{h}:=\inf\{\|\sum x_{i}x_{i}^{*}\|\|\sum y_{i}^{*}y_{i}\|: z=\sum x_{i}\otimes y_{i}\},
\label{eq:Haag-tens}
\end{equation}
and the completion is denoted $X\hotimes Y$. In case $X$ is a left and 
$Y$ is a right  operator $\A$-module, the \emph{Haagerup module tensor product} 
$X\hotimes_{\A} Y$ is the quotient of $X\hotimes Y$ by the closed linear 
span of the expressions $x\otimes ay-xa\otimes y$. The norm on 
$X\hotimes_{\A} Y$ is the quotient norm.

The main feature of the Haagerup tensor product is that it 
makes operator multiplication $X\hotimes \A\rightarrow X$ 
continuous for operator modules. An equivalent way of 
defining operator modules is by requiring the multiplication 
to be contractive on the Haagerup tensor product, cf. \cite[Cor. 3.3]{CES}. See also \cite[Thm. 3.3.1]{Blecherbook} and \cite{BMP}.

By an \emph{inner product operator module} \cite{KaLe2, Mes} we 
mean a right operator module $\E$ over an operator 
$*$-algebra $\B$, that comes equipped with a sesquilinear pairing
\[
\E\times \E \to \B, \quad (e_{1},e_{2})\mapsto \langle e_{1},e_{2}\rangle,
\]
satsisfying the usual inner product axioms,
\[
\langle e_{1,}e_{2}b\rangle=\langle e_{1},e_{2}\rangle b,\quad 
\langle e_{1},e_{2}\rangle^{*}=\langle e_{2},e_{1}\rangle, \quad 
\langle e,e\rangle \geq 0\ \mbox{in}\ B,\quad \langle e,e\rangle =0\Leftrightarrow e=0,
\]
and the weak Cauchy-Schwarz inequality 
$\|\langle e_{1},e_{2}\rangle\|_\B\leq C\|e_{1}\|_\E\|e_{2}\|_\E$ 
on the level of matrix norms, for some $C>0$. 
Notice that we do not require $\E$ to be complete 
in the norm $\|e\|^{2}_\E:=\|\langle e,e\rangle\|_\B$. 
Completion in this norm will give a $C^{*}$-module 
over the $C^{*}$-closure of $\B$. The norm on 
$\E$ is additional data that is in general not constructible from the 
inner product alone, as demonstrated by the following example. 
Consider the operator $*$-algebra $\A_{\D}$ in the representation 
\eqref{eq:lip-norm}, viewed as an inner product module over itself via 
$\langle a,b\rangle:=a^{*}b$. The norm of $a\in \A$ is given by 
$\|a\|_{\D}^{2}=\|\pi_{\D}(a)^{*}\pi_{\D}(a)\|\neq \|\pi_{\D}(a^{*}a)\|=\|\pi_{\D}(\langle a,a\rangle)\|$.

\subsection{Bounded approximate units}
In this section 
we gather some useful facts about bounded approximate units in 
the general setting of operator algebras. In everything that follows, $\Vert\cdot\Vert$ denotes
the norm on an operator algebra, while $\Vert\cdot\Vert_\infty$ will denote the usual norm
on operators on a Hilbert space. For the example of a (symmetric) spectral triple $(\A,\H,\D)$ we have
$\Vert a\Vert=\Vert\pi_\D(a)\Vert_\infty$.
\begin{defn} 
Let $\A$ be an operator algebra. A \emph{bounded approximate unit}
is a net $(u_{\lambda})_{\lambda\in\Lambda}\in\A$ such that:\\
1) $\sup_{\lambda}\|u_{\lambda}\|<\infty$;\\
2) for all $a\in\A$, $\lim_{\lambda\rightarrow\infty}\|u_{\lambda}a-a\|
=\lim_{\lambda\rightarrow\infty}\|au_{\lambda}-a\|=0$.

The bounded approximate unit is \emph{commutative} if $u_{\lambda}u_{\mu}=u_{\lambda}u_{\mu}$ 
for all $\lambda,\mu\in\Lambda$ and \emph{sequential} in case $\Lambda=\N$.
\end{defn}
By a \emph{cb-representation} of an operator algebra $\A$ we 
mean a completely bounded algebra homomorphism 
$\pi:\A\rightarrow \mathbb{B}(\H)$, where $\H$ is a Hilbert space. 
A representation $\pi$ is \emph{essential} (also called \emph{nondegenerate} 
in the literature) if $\pi(\A)\H$ is dense in $\H$. Our first aim is 
to show that in any cb-representation of an operator algebra $\A$, a bounded approximate unit 
converges strongly to an idempotent.

Let $\pi:\A\rightarrow\mathbb{B}(\H)$ be a cb-representation. 
For $W\subset\H$, denote by $W^{\perp}$ its orthogonal complement, 
by $[W]$ the closed linear span of $W$, and set
\[
\pi(\A)\H:=\{\pi(a)h: a\in\A,\, h\in\H\}.
\]
 The Hilbert space $\mathcal{H}$ splits as a direct sum of closed 
 orthogonal subspaces in two ways:
\[ 
\H=[\pi(\A)\H]\oplus [\pi(\A)\H]^{\perp}= [\pi(\A)^{*}\H]\oplus [\pi(\A)^{*}\H]^{\perp}.
\]
Another important subspace of $\H$ associated to $\pi$ is 
\[
\textnormal{Nil }\pi(\A):=\{h\in\H:\pi(a)h=0 \textnormal{  for all   } a\in\A\}.
\] 
\begin{lemma} 
Let $\A$ be an operator algebra and 
$\pi:\A\rightarrow \mathbb{B}(\H)$ a cb-representation. Then $\Nil\pi(\A)=[\pi(\A)^{*}\H]^{\perp}$.
\end{lemma}
\begin{proof} 
For $h\in\Nil\pi(\A)$ and $v\in\H, a\in\A$ we have
$\langle h,\pi(a)^{*}v\rangle=\langle \pi(a)h,v\rangle=0$,
so $\Nil\pi(\A)\subset[\pi(\A)^{*}\H]^{\perp}$. Now let $h\in[\pi(\A)^{*}\H]^{\perp}$, 
$v\in\H$ and $a\in\A$. Then
$
\langle\pi(a)h,v\rangle=\langle h,\pi(a)^{*}v\rangle=0$,
so $\pi(a)h=0$ and $h\in\Nil \pi(\A)$.
\end{proof}
\begin{lemma}
\label{Nil} 
Let $\A$ be an operator algebra with 
bounded approximate unit $(u_{\lambda})$ and also let $\pi:\A\rightarrow\bB(\H)$ be a cb-representation. 
Then: \\
1) $\pi(u_{\lambda})h\rightarrow h$ for all $h\in [\pi(\A)\H]$;\\
2) $\textnormal{Nil }\pi(\A)\cap [\pi(\A)\H]=\{0\}=\textnormal{Nil }\pi(\A)^{*}\cap [\pi(\A)^{*}\H].$
\end{lemma}
\begin{proof} 
To prove 1), let $h=\pi(a)v$ and observe that $\pi(u_{\lambda})h\rightarrow h$ 
since $(u_{\lambda})$ is an approximate unit. So the convergence property 
is satisfied by all $h$ in a dense subset of $ [\pi(\A)\H]$. Uniform boundedness 
of $(u_{\lambda})$ now gives the result for all $h\in [\pi(\A)\H]$.
\newline
For 2), let $h$ be a vector in the intersection $\textnormal{Nil }\pi(\A)\cap [\pi(\A)\H]$ 
so that $\pi(u_{\lambda})h\rightarrow h$ as above, since $h\in [\pi(\A)\H]$. 
On the other hand, since $h\in\textnormal{Nil }\pi(\A)$, we have $\pi(a)h=0$ 
for all $a\in\A$, so in particular $\pi(u_{\lambda})h=0$ for all $\lambda$, 
so $h=0$. Similarly $\textnormal{Nil }\pi(\A)^{*}\cap [\pi(\A)^{*}\H]=\{0\}.$
\end{proof}
The following theorem generalises the observations in the 
appendix to \cite{Mes}. The result has been known for 
contractive approximate units for a long time: 
see for example \cite[Lemma 2.1.9]{Blecherbook} and its proof.
\begin{thm}
\label{thm: idempotent} 
Let $\A$ be an operator algebra with bounded approximate unit $(u_{\lambda})$, and 
$\pi:\A\rightarrow \bB(\H)$ a cb-representation. Then the net $(\pi(u_{\lambda}))$ 
converges strongly, and hence weakly, to an idempotent $q\in \bB(\H)$ with the following properties:\\
1) for all $a\in\A$, $q\pi(a)=\pi(a)q=\pi(a)$;\\
2) $q\H=[\pi(\A)\H]$;\\
3) $(1-q)\H=\Nil\pi(\A)$;\\
4) $\|q\|\leq \|\pi\|\sup_{\lambda}\|u_{\lambda}\|$.
\end{thm}
\begin{proof} 
Denote by $p$ the projection onto $[\pi(\A)\H]$ and by 
$p_{*}$ the projection onto $[\pi(\A)^{*}\H]$.
The bounded and self-adjoint operator
\[
x\mapsto (p+(1-p_{*}))x,
\]
is injective, for if $(p+(1-p_{*}))x=0$ then $px=-(1-p_{*})x$ so 
\[
px\in [\pi(\A)\H]\cap [\pi(\A)^{*}\H]^{\perp}=[\pi(\A)\H]\cap \textnormal{Nil }\pi(\A)=\{0\}.
\]
Therefore $px=(1-p_{*})x=0$ and $x=(1-p)x=p_{*}x$, so
\[
x\in [\pi(A)\H]^{\perp}\cap  [\pi(\A)^{*}\H]=\textnormal{Nil }\pi(\A)^{*}\cap [\pi(\A)^{*}\H]=\{0\}.
\]
Since $p+(1-p_{*})$ is self-adjoint,
\[
[{\rm Im}(p+(1-p_{*}))]=\ker (p+(1-p_{*}))^{\perp}=\H,
\]
and therefore ${\rm Im}(p+(1-p_{*}))$ is dense in $\H$. In particular, the subspace
$[\pi(A)\H]+\textnormal{Nil }\pi(\A)$
is dense in $\H$. Now let $\xi\in\H$ and $\varepsilon >0$. 
Choose $x\in [\pi(\A)\H]$ and $y\in \textnormal{Nil }\pi(\A)$ 
such that $\|\xi-x-y\|<\frac{\varepsilon}{4C}$, with $C:=\sup\|\pi(u_{\lambda})\|$. 
Now choose $\lambda<\mu$ large enough such that 
$\|\pi(u_{\lambda}-u_{\mu})x\| <\frac{\varepsilon}{2}$. Then
\[
\begin{split}\|\pi(u_{\lambda}-u_{\mu})\xi\| 
&\leq\|\pi(u_{\lambda}-u_{\mu})(x+y)\|+\| \pi(u_{\lambda}-u_{\mu})(\xi-x-y)\| \\ 
&\leq \|\pi(u_{\lambda}-u_{\mu})x\|+\| \pi(u_{\lambda}-u_{\mu})\|\|(\xi-x-y)\| 
\leq \frac{\varepsilon}{2}+\frac{\varepsilon}{2}=\varepsilon,
\end{split}
\]
which shows that $(\pi(u_{\lambda}))$ is a strong Cauchy net. Since the 
strong operator topology is complete on bounded sets, the sequence 
has a limit $q$. By definition of $q$ 
\begin{equation}\label{comm}
q\pi(a)=\pi(a)=\pi(a)q,
\end{equation}
which proves $1)$. From this it follows that
\[
q^{2}\xi=\lim_{\lambda}\pi(u_{\lambda})q\xi=\lim_{\lambda}\pi(u_{\lambda})\xi=q\xi,
\]
so $q$ is idempotent and in particular has closed range. 
It is immediate from the definiton of $q$ and Equation \eqref{comm} that
\[
\pi(\A)\H\subset{\rm Im}\, q\subset [\pi(\A)\H],
\]
and so ${\rm Im }\,q=[\pi(A)\H]$, proving $2)$. For $3)$, 
observe that 
\[
\pi(a)(1-q)=(1-q)\pi(a)=0,
\]
so  we have ${\rm Im }(1-q)\subset\Nil\pi(\A)$. 
If $h\in\Nil \pi(\A)$, then $qh=0$, so $h=(1-q)h\in{\rm Im }(1-q)$. Finally, 4) follows from
\[
\|q\|=\sup_{\|h\|\leq 1}\|qh\|\leq\sup_{\|h\|\leq 1}\|\lim_{\lambda}\pi(u_{\lambda})h\|
\leq\|\pi\|\sup_{\lambda}\|u_{\lambda}\|\qedhere
\]
\end{proof}
\begin{corl} \label{sumdecomp}
The Hilbert space $\H$ splits as a non-orthogonal direct sum 
$\H\cong[\pi(A)\H]\oplus [\pi(A)^{*}\H]^{\perp}$.
\end{corl}
Such splittings for  $C^{*}$-modules need to be handled with more care, 
and we only treat the case of symmetric Kasparov modules with 
some additional convergence hypotheses. Recall that the 
\emph{strict topology} on $\End^{*}_{B}(E)$ is defined by the 
seminorms $\|T\|_{e}:=\max\{\|Te\|,\|T^{*}e\|\}$, 
and thus models pointwise convergence on $E_{B}$.
\begin{prop}
\label{cor:bdd-strong}
Let $(\A, E_{B},\D)$ be a symmetric Kasparov module 
for which $[\pi(A)E_{B}]$ is a complemented submodule 
of $E_{B}$ and $p\in \End^{*}_{B}(E_{B})$ the corresponding 
projection. Let $(u_n)$ be a sequential bounded 
approximate unit for the differentiable algebra $\A$. Then: \\
1) $p$ is the strict limit of  $(u_n)$;\\
2) $p[\D^*,u_n]p\to 0$ strictly;\\
3) if $(\D u_{n}e)$ converges for all $e\in \Dom \D^{*}$ then 
$p\in\Lip(\D^{*})$ and $[\D^{*},p]$ is the strict limit of the sequence $([\D,u_{n}])$.
\end{prop}
\begin{proof} 
Let $p$ be the projection onto $[\pi(A)E_{B}]$, 
which exists because this submodule is complemented. 
For $e\in [\pi(A)E_{B}]$ we have $u_{n}e\rightarrow e$, 
since $u_{n}a\rightarrow a$ in the $C^{*}$-norm. 
Moreover $pa=ap=a$ for all $a\in A$, and thus $(1-p)a=a(1-p)=0$. Therefore
\[
\lim_{n}u_{n}e=\lim_{n}u_{n}pe+u_{n}(1-p)e=\lim_{n} u_{n}pe=pe,
\]
and $u_{n}\rightarrow p$ strictly, proving 1).
Since $(u_n)$ is a bounded approximate unit for $\A$,
the sequence of operators $[\D^*,u_n]$ is uniformly bounded. 
For $a\in\A$ and $e\in \Dom \D^{*}$ we have
$$
[\D^*,u_n]ae=[\D^*,u_na]e-u_n[\D^*,a]e.
$$
Since $ae=pae=ape$, multiplying on the left by $p$ yields
$$
p[\D^*,u_n]pae=p[\D^*,u_na]pe-pu_n[\D^*,a]pe.
$$
Both terms on the right hand side converge to $p[\D^*,a]pe$, and so the right hand side converges
to zero. Hence the left hand side also converges to zero. As vectors of the form $ae$ are dense 
in $pE_{B}$, we see that $p[\D^*,u_n]p$ converges strictly to zero, which proves 2).

To prove 3) we first show that $p$ maps $\Dom \D^{*}$ into $\Dom \D$.
As $(\A,E_B,\D)$ is a symmetric Kasparov module, each 
$u_n$ maps the domain of $\D^*$ into the domain of $\D$. Since, by assumption, 
\[
\pi_{\D}(u_{n})\begin{pmatrix} e \\ \D^{*} e\end{pmatrix}
=\begin{pmatrix}u_{n} e \\ \D u_{n} e\end{pmatrix} ,
\]
is convergent and by 1) the projection $p$ is the strict limit of the 
$u_{n}$, we find that
\[
\lim_{n\rightarrow \infty} u_{n}\begin{pmatrix} e \\ \D e\end{pmatrix}
=\begin{pmatrix} pe \\ x \end{pmatrix}.
\]
Now the graph of $\D$ is closed so it follows that $pe\in \Dom \D$ and
\begin{equation}
\label{conv}
x=\lim_{n \rightarrow \infty} \D u_{n} e=\D pe.
\end{equation} 
Now observe that, since $p \Dom \D^{*}\subset \Dom \D$ we can write for $e\in \Dom \D^{*}$ 
\[
\begin{split}  
[\D,u_{n}]e & =[\D,u_{n}]pe+[\D,u_{n}](1-p)e\\ 
&=\D u_{n}pe-u_{n}\D pe+\D u_{n}(1-p)e - u_{n}\D^*(1-p)e \\
&=\D u_{n}e-u_{n}\D pe - u_{n}\D^*(1-p)e \\
&\to \D pe -p\D pe -p\D^*(1-p)e\qquad\qquad\qquad\qquad\textnormal{by Equation \eqref{conv}}\\
&=[\D^*,p] e,
\end{split}
\]
which tells us that $[\D,u_{n}]$ converges to $[\D^*,p]$ on $\Dom \D^{*}$. Since 
the sequence $[\D,u_{n}]$ is bounded, it converges strictly on all of 
$E_{B}$, and the operator $[\D^*, p]$ is thus bounded on $\Dom \D^{*}$. This proves 3).\end{proof}

\begin{rmk} 
It would be desirable to remove the convergence 
hypothesis in 3). At present it seems unlikely to be possible without further  assumptions.
\end{rmk}
In fact our seeming flexibility in allowing symmetric operators is redundant in the presence
of a bounded approximate unit.
\begin{corl}
\label{cor:symmetric-spec}
Let $(\A,E_{B},\D)$ be a symmetric Kasparov module with $A\cdot E_{B}$ dense in $E_{B}$. 
If $\A$ has a sequential bounded approximate unit then $\D$ is self-adjoint.
\end{corl}
\begin{proof}
Suppose we have an approximate identity $(u_n)$ with $[\D^*,u_n]$ uniformly bounded
in $n$. Then 
by Lemma \ref{cor:bdd-strong}, $[\D^*,u_n]\to 0$ strictly,  
since $p=1$. Thus for $e\in\Dom\D^*$
we find
$$
\D^*u_ne=[\D^*,u_n]e+u_n\D^*e\to \D^*e.
$$
As we also have $u_ne\to e$, and $u_ne\in \Dom\D$,
we see that $e$ is in the graph norm completion of $\Dom\D$. As
$e\in\Dom\D^*$ was arbitrary,  $\Dom\D^*\subset\Dom\D$ and so $\D$ is self-adjoint.
\end{proof}
\subsection{Bounded and unbounded multipliers of differentiable algebras}
It is a well-known fact that for a Banach algebra $\A$ with a 
bounded approximate unit, the multiplier algebra $\mathbb{M}(\A)$ 
is isomorphic to the strict closure of $\A$, and contains $\A$ as an 
essential ideal \cite[Ch 5]{Palmer}. 
Similarly, a representation 
$A\rightarrow \mathbb{B}(\H)$ of a $C^{*}$-algebra 
$A$ on a Hilbert space $\H$ extends to a representation 
of the multiplier algebra $\mathbb{M}(A)$ on that same Hilbert space. We discuss these notions here for operator algebras with bounded approximate unit.

For an operator algebra $\A$, $\mathbb{M}(\A)$ 
inherits matrix norms by viewing elements of $\mathbb{M}(\A)$ 
as operators on $\A$. The next lemma shows that the presence of a bounded approximate 
unit ensures that this norm, when restricted to $\A$, is cb-equivalent 
to the original norm on $\A$, ensuring that the inclusion is a cb-equivalence. 
\begin{lemma}[cf. Chapter 5 of \cite{Palmer}]
\label{opnormlemma}
Let $\A$ be an operator algebra with bounded approximate unit. 
Then the norm on $M_{n}(\A)$ is equivalent to the norm
\begin{equation}
\label{opnorm} 
\|a\|_{\textnormal{op},n}:=\sup_{\|b\|_{n}\leq 1}\|ab\|_{n},\quad \|a\|_{\textnormal{op},n}\leq \|a\|_{n}\leq C\|a\|_{\textnormal{op},n},
\end{equation}
with $C$ a constant independent of $n$. 
\end{lemma}
\begin{proof} 
Obviously, it holds that $\|a\|_{\textnormal{op}}\leq \|a\|$. 
If $u_{\lambda}$ is a bounded approximate unit, then 
$\frac{1}{C}\|u_{\lambda}\|\leq 1$ for some fixed constant $C$ and all $\lambda$. 
For any $\epsilon>0$ there exists $\lambda$ such that $\|b-bu_{\lambda}\|<\epsilon$ and thus
\[
\frac{1}{C}(\|b\|-\epsilon)<\frac{1}{C}(\|b\|-\|b-bu_{\lambda}\|)
\leq\frac{1}{C}\|bu_{\lambda}\|\leq \|b\|_{\textnormal{op}},
\]
which proves the assertion. The argument for the matrix norms $\|\cdot \|_{n}$ is verbatim the same using the bounded approximate unit $ (u_{\lambda}\cdot{\rm Id}_n)$. 
\end{proof}
\begin{defn}\label{strictmult} Let $\A$ be an operator algebra with bounded approximate unit. We define the \emph{multiplier algebra} $\M(\A)$ to be the \emph{strict closure} of $\A$. That is
\[
\M(\A)\!:=\{T:\A\to \A: \exists \ {\rm a\ net}\,(b_{\lambda})\subset\A\ {\rm such\ that}\ 
\forall a\in\A\  \lim \|b_{\lambda}a-Ta\|=\lim \|ab_{\lambda}-aT\|=0\},\]
with norm $\|T\|:=\|T\|_{\textnormal{op}}$ cf. Lemma \ref{opnormlemma}.
\end{defn}
It is worth noting that the strict topology on $\End_{B}^{*}(E)$ as defined before Proposition \ref{cor:bdd-strong} coincides with the strict topology in the sense of Definition \ref{strictmult} defined by the ideal $\kK(E_{B})$.
\begin{lemma}
\label{lem: essext} Let 
$\A$ be an operator algebra with bounded approximate unit $(u_{\lambda})$ and 
$\pi:\A\rightarrow \bB(\H)$ be an essential cb-representation. 
Then $\pi$ extends uniquely to a cb-representation 
$\pi:\mathbb{M}(\A)\rightarrow \bB(\H)$ such that $\pi(1)=1$.
\end{lemma}
\begin{proof} 
By assumption $\H=[\pi(\A)\H]$, so for all 
$h\in\H$ we have $\pi(u_{\lambda})h\rightarrow h$ by 
Lemma \ref{Nil}. Since $\mathbb{M}(\A)$ is the strict closure of $\A$, 
for all $b\in\mathbb{M}(\A)$ it holds that $\sup_{\lambda}\|bu_{\lambda}\|<\infty$ and for all $a\in \A$, $(bu_{\lambda}a)$ is a Cauchy net in $\A$. Therefore
\[
\pi(b)\pi(a)h:=\lim_{\lambda}\pi(bu_{\lambda}a)h,
\]
is a Cauchy net for all $a\in\A$ and $h\in \H$. Thus 
$\pi(bu_{\lambda})h$ converges for $h\in \pi(\A)\H$. 
Since this subspace is dense in $\H$ and the net 
$(\pi(bu_{\lambda}))$ is uniformly bounded, the net is 
strongly Cauchy on $\H$. This proves that the assigment 
$h\mapsto \lim_{\lambda}\pi(bu_{\lambda})h$ defines a 
bounded operator on $\H$. For $a\in\A$ and $b\in \M(\A)$, 
it is immediate from the definition that $\pi(ab)=\pi(a)\pi(b)$. Then for $a,b\in\M(\A)$ we have
\[
\pi(a)\pi(b)h=\pi(a)(\lim_{\lambda} \pi(bu_{\lambda})h)=\lim_{\lambda}\pi(abu_{\lambda})h=\pi(ab)h,
\]
proving that the extension of $\pi$ is a homomorphism. 
Since for all $a\in \A$ and $b\in\M(\A)$  we have  
$bu_{\lambda}a\rightarrow ba$ in $\A$, it is immediate that any other cb-extension of $\pi$ 
must coincide with the one given, proving uniqueness.
\end{proof}
\begin{lemma}
\label{multiplierproperties} Let $\A$ be an operator algebra with bounded approximate unit 
$(u_\lambda)\subset\A$ and 
$\pi:\A\rightarrow \bB(\H)$ an essential  cb-isomorphic representation. Then:\\
1) the strict closure $\M(\A)$ of $\pi(\A)$ is cb-isomorphic to the idealiser  of 
$\pi(\A)\in \bB(\H)$;\\
2) every element $T\in \M(\A)$ is the strict limit of a bounded net in $\A$;\\
3) if $\J\subset\A$ is a closed ideal, then $\J$ is a closed ideal in $\M(\A)$.
\end{lemma}
\begin{proof} By Lemma \ref{lem: essext}, 
$\pi$ extends to a representation of $\M(\A)$.
Let $T$ be an element of $\pi(\M(\mathcal{A}))$, so that there is a net
$(b_{\lambda})\subset \A$ with the property that for all $a\in \A$
\[
\|b_{\lambda}a-Ta\|,\ \|ab_{\lambda}-aT\|\rightarrow 0.
\]
Since $\A$ is norm closed and $\pi$ is cb-isomorphic, 
it follows that $\pi(Ta)$, $\pi(aT)\in\pi(\A)$ for all 
$a\in \A$, so $\pi(T)$ idealises $\pi(\A)$. 
Now let $T\in \bB(\H)$ be such that $T\pi(a)$, $\pi(a)T\in \pi(\A)$ for all $a\in \A$. Consider 
the net $T\pi(u_{\lambda})\in \pi(\A)$. For $a\in \A$ we have
\[
\|T\pi(u_{\lambda}a)-T\pi(a)\|\leq \|T\|\|\pi(u_{\lambda}a-a)\|\rightarrow 0,\quad 
\|\pi(a)T\pi(u_{\lambda})-\pi(a)T\|\rightarrow 0,
\]
so since $\pi$ is cb-isomorphic and essential, $T$ is the image of an element in $\M(\A)$. 
For the second statement, 
observe that $T$ is the strict limit of the bounded net 
$(Tu_{\lambda})$, as in Lemma \ref{opnormlemma}. For the 
third assertion, let $T\in \M(\pi(\A))$ and $j\in\J$. 
Since $Tj\in\A$, the net $(u_{\lambda}Tj)$ converges to 
$Tj$ in norm. But $u_{\lambda}T\in \A$ so this net actually 
lies in $\J$. Since $\J$ is closed, $Tj\in \J$, and similarly for $jT$.
\end{proof} 

\begin{thm} 
\label{thm: ext} 
Any cb-representation $\pi:\A\rightarrow \bB(\H)$ 
of an operator algebra with bounded approximate unit extends uniquely to a representation 
$\pi:\mathbb{M}(\A)\rightarrow \bB(\H)$ of the multiplier algebra 
$\mathbb{M}(\A)$, such that $\pi(1)$ is an idempotent satisfying 
$\pi(1)\H=[\pi(\A)\H]$ and $(1-\pi(1))\H=\Nil\pi(\A)$. 
\end{thm}
\begin{proof} 
The Hilbert space $\H$ is cb-isomorphic to the nonorthogonal direct 
sum $q\H\oplus(1-q)\H$, with $q$ as in Proposition \ref{thm: idempotent}. The 
representation $\pi$ is essential on $q\H$ and $0$ on $(1-q)\H$. Thus, 
Lemma \ref{lem: essext} gives a representation $\mathbb{M}(\A)\rightarrow \bB(q\H)$, 
which extends to $0$ on $(1-q)\H$, thus giving the desired representation 
$\pi:\mathbb{M}(\A)\rightarrow \bB(\H)$. By construction $\pi(1)=q$.
\end{proof}
We now consider multiplier algebras for closed subalgebras of $\Lip(\D)$, and in 
particular for differentiable algebras of spectral triples.
\begin{prop} 
\label{lipmult}
Let $\D:\Dom \D\rightarrow E_{B}$ be a self-adjoint regular operator and  
$\A\subset\Lip(\D)$ a closed subalgebra with bounded 
approximate unit and assume $[AE_{B}]=E_{B}$.
The multiplier algebra $\M(\A)$ is cb-isomorphic to the algebra
\begin{equation}
\label{mult}
\big\{T\in\mathbb{M}(A):T\Dom \D\subset\Dom\D,\ \  T\A,\,\A T\subset\A,\ \  [\D,T]\in\End^{*}_{B}(E_{B})\big\},
\end{equation}
topologised by the representation given in 
Equation \eqref{eq:sym-lip-norm}. The inclusion $\M(\A)\to \M(A)$ is spectral invariant.
\end{prop}
\begin{proof}
The algebra defined in \eqref{mult} is clearly a subalgebra of  $\mathbb{M}(A)$ contained
in the idealiser of $\pi_{\D}(\A)$  inside $\End_{B}(E\oplus E)$. 
The other inclusion can be seen by writing
\[
\begin{pmatrix} T_{11} & T_{12}\\ T_{21} & T_{22} \end{pmatrix}
\begin{pmatrix} a & 0 \\ [\D,a] & a\end{pmatrix}
=\begin{pmatrix} T_{11}a+T_{12}[\D,a] & T_{12}a \\ T_{21}a+T_{22}[\D,a] & T_{22}a\end{pmatrix},
\]
and observing that for this to be an element of $\A$ for all $a\in \A$, 
$T_{12}a=0$ for all $a\in \A$ and hence $T_{12}=0$ since $A$ is essential. 
It then follows that $T_{11}a=T_{22}a$ for all $a\in \A$ which implies 
$T_{11}=T_{22}$, again because $A$ is essential. Writing $T_{11}=T$, 
one again derives from essentiality of $\A$ that $T$ must preserve the 
domain of $\D$. Finally we get the equation $[\D,Ta]=T_{21}a+T[\D,a]$, 
which implies that $T_{21}=[\D,T]$ which is therefore bounded. Thus the algebra
\eqref{mult} contains the idealiser of $\pi_{\D}(\A)$, and is therefore equal to it.

Since $[AE_{B}]=E_{B}$, we have $\pi_{\D}(1)=1$, using Lemma \ref{lem: essext}, and 
the representation $\pi_{\D}$ is essential. An argument similar to that 
given in the proof of Lemma \ref {lem: essext} shows that the strict 
closure $\M(\A)$ maps into $\End^{*}_{B}(E\oplus E)$, whereas the argument given in Lemma \ref{multiplierproperties} shows that the idealiser of $\pi_{\D}(\A)$ in $\End_{B}^{*}(E\oplus E)$  
coincides with the image of this strict closure. 
The norm on the strict 
closure is given by Equation \eqref{opnorm}, and 
the equivalence of norms given there proves that $\M(\A)$ is cb-isomorphic to 
the idealiser \eqref{mult}. 
Spectral invariance of the inclusion $\M(\A)\subset \M(A)$ 
now follows from spectral invariance of 
$\Lip(\D)\subset \End^{*}_{B}(E_{B})$, cf. \cite[Thm B.3]{Mes}.
\end{proof}
In \cite{Ralf} it was shown that any operator algebra 
$\A$ admits a canonical unitisation. In this paper, 
our main examples are closed subalgebras 
$\A\subset\Lip(\D)$, where $\D$ is a self-adjoint 
regular operator on a $C^{*}$-module $E_{B}$ with essential $A$ representation. 
In this setting we can construct unitisations concretely.
\begin{defn}
\label{unitise}
Let $\D:\Dom\D\rightarrow E_{B}$ be a self-adjoint regular operator 
and $\A\subset \Lip(\D)$ a differentiable algebra with 
bounded approximate unit and $C^{*}$-closure $A$. 
If $[AE_{B}]=E_{B}$, the \emph{unitisation} 
$\A^{+}\subset\M(\A)\subset \Lip(\D)$ is the algebra generated by $\A$ and $\pi_{\D}(1)=1$.
\end{defn}
\begin{rmk}\label{essrmk}
The requirement that $[AE_{B}]=E_{B}$ is not a severe restriction.
By \cite[Lemma 2.8]{Kas2} every class in $KK(A,B)$ can be represented 
by a bounded Kasparov module $(A,E_{B}, F)$ with $[AE_{B}]=E_{B}$. 
Combining this with Kucerovsky's lifting results 
\cite[Lemma 1.4, Lemma 2.2]{Kucerovsky2}, every class in 
$KK(A,B)$ can be represented by an unbounded module 
$(A, E_{B}, \D)$ with $[AE_{B}]=E_{B}$. Thus, the only 
serious hypothesis in Definition \ref{unitise} is that $\A$ 
have a bounded approximate unit. Unless otherwise stated, 

{\bf from now on we assume that all
unbounded Kasparov modules are essential}.
\end{rmk}
Unbounded multipliers on $C^{*}$-algebras were introduced by Baaj (\cite{Baajthesis}) 
and Woronowicz (\cite{Wor}). In the differentiable setting,  the definition of unbounded 
multiplier requires a bit more care, because of the absence of the 
strong relation between norm and spectrum.
\begin{defn} 
\label{def:siamese}
Let $\A$ be a differentiable algebra.
A linear map $c:\Dom c\subset \A\rightarrow \A$, 
defined on the dense right ideal $\Dom c\subset\A$ is 
a \emph{multiplier} if $c(ab)=(ca)b$ for all $a\in\Dom c$ and $b\in \A$. The
operator $c$ is a \emph{symmetric unbounded multiplier} if:\\
1) $c$ is closed; \\
2) for all $a,b\in\Dom c$ we have $(ca)^{*}b=a^{*}(cb)$; \\
and $c$ is \emph{self-adjoint} if \\
3) $c\pm i$ are surjective and $(c\pm i)^{-1}\in \mathbb{M}(\A)$.\\
The multiplier $c$ is \emph{positive} if for all $a\in \Dom c$ we have $(ca)^{*}a\geq 0$ in $C^*(\A)$.
\end{defn}
The spectral invariance of the inclusion $\M(\A)\rightarrow \M(A)$ (cf. Proposition \ref{lipmult})
ensures that some of the usual properties of 
positive and self-adjoint multipliers remain valid in the differentiable context. 
As a first consequence of the inclusion $\M(\A)\to \M(A)$, 
the resolvents $(c\pm i)^{-1}\in \M(\A)$ define elements in 
the $C^{*}$-multiplier algebra $\M(A)$. Hence $c$ 
defines a self-adjoint multiplier on the $C^{*}$-algebra 
$A$ in the usual sense, with $\Dom c=\textnormal{im} (c\pm i)^{-1}\subset A$.
\begin{lemma} 
\label{spectrum}
Let $c:\Dom c\rightarrow \A$ be a self-adjoint multiplier.
For all $\lambda\in\C\setminus\R$, the operators 
$(c\pm \lambda): \Dom c\rightarrow \A$ are bijective and $(c\pm \lambda)^{-1}\in \M(\A)$. Moreover
if $c$ is positive then for all $\lambda\in \C\setminus[0,\infty)$, the operators 
$c-\lambda:\Dom c\rightarrow\A$ are bijective and $(c-\lambda)^{-1}\in\M(\A)$.
\end{lemma}
\begin{proof} 
The operators $c\pm \lambda $ are bijective in the $C^{*}$-closure $A$,
and thus $(c\pm \lambda )(c\pm i)^{-1}\in \M(A)$ are invertible. 
Spectral invariance then tells us that $g=(c\pm \lambda )(c\pm i)^{-1}$ is 
invertible in $\M(\A)$, whence $c\pm\lambda :\Dom c\rightarrow \A$ is bijective. 
The inverse satisfies the equation
\[
g^{-1}=(c\pm i)(c\pm \lambda )^{-1}=1\mp(\lambda -i)(c\pm \lambda )^{-1},
\]
in $\M(A)$, and since both $1,g^{-1}\in\M(\A)$, it follows that 
$(c\pm \lambda )^{-1}\in \M(\A)$. The positive case is proved similarly.
\end{proof}
If there is an orthogonal decomposition 
$E_{B}=[\pi(A)E_{B}]\oplus [\pi (A)E_{B}]^{\perp}$ 
(which is always the case for Hilbert spaces) we can extend 
the self-adjoint multiplier to a self-adjoint operator on 
$E_{B}$ by defining $c([\pi (A)E_{B}]^{\perp})=0$. 
We denote this extension to $E_{B}$ by $c$ as well. 
It is the \emph{affiliated operator} from \cite{Baajthesis, Wor}.
In case $\A$ has a bounded approximate unit, condition 3) of Definition \ref{def:siamese}
can be weakened to the requirement that $c\pm i$ have 
dense range and $(c\pm i)^{-1}$ are norm bounded, as we now show. 

\begin{lemma} 
Let $\A$ be a differentiable algebra with bounded approximate unit. 
Then every symmetric multiplier is closable.
\end{lemma}
\begin{proof} 
Since $\A$ has a bounded approximate unit, the norm on 
$\A$ is equivalent to the norm $\|a\|_{\textnormal{op}}:=\sup_{\|a\|\leq 1}\|ab\|$ by Lemma \ref{opnormlemma} . 
Let $(a_{n})$ be a sequence in $\Dom c$ with $a_{n}\rightarrow 0$ and 
$ca_{n}\rightarrow b$. Since $\|(ca_{n})\|=\|(ca_{n})^{*}\|$ and $c$ is 
symmetric, for arbitrary $a\in\Dom c$, we get 
\[
b^{*}a=\lim_{n\rightarrow \infty} (ca_{n})^{*}a=\lim_{n\rightarrow\infty}a_{n}^{*}(ca)=0.
\]
From this it follows that $\|b^{*}\|_{\textnormal{op}}=0$, and therefore $\|b^{*}\|=0$ so $\|b\|=0$.
\end{proof}
\begin{corl} \label{symmul}
Let $\A$ be a differentiable algebra with bounded approximate unit and let $c$ be a 
symmetric multiplier such that $(c\pm i)^{-1}$ are densely defined 
and bounded. Then the closure of $c$ is a self-adjoint unbounded 
multiplier with $\Dom c=\textnormal{Im} (c\pm i)^{-1}$.
\end{corl}
In the context of separable $C^{*}$-algebras, the notion of unbounded 
multiplier, approximate unit, and strictly positive element are closely related. 
For a differentiable algebra $\A$, an element $h\in\A$ is \emph{strictly positive} 
if it has positive spectrum and $h\A$ is dense in $\A$ (for the topology coming from $\pi_\D$). 
Note that this implies that $h$ is strictly positive in the $C^{*}$-algebra $A$.

A more refined notion of unbounded multiplier for differentiable algebras which is compatible
with strict positivity is given in the next definition. The core idea is
abstracted from \cite[Definition 10.2.8]{HR},
which gives a commutator approach to properness of the metric. 
Examples illustrating the connection are presented in Section \ref{sec:met}.
\begin{defn}
\label{Lipmult} 
Let $(\A, E_{B},\D)$ be an unbounded Kasparov module,
and $c$ a self-adjoint multiplier of $\A$. 
Then $c$ is a \emph{complete multiplier} if:\\
1) $(c\pm i)^{-1}\in\A$;\\
2) $\im (\D\pm i)^{-1}(c\pm i)^{-1}=\im (c\pm i)^{-1}(\D \pm i)^{-1}\subset E_{B}$;\\
3) $[\D,c]$ is bounded on the set $\im (\D\pm i)^{-1}(c\pm i)^{-1}$.
\end{defn}
It should be noted that the condition in 2) is natural when dealing with 
commutators of unbounded operators. The sets mentioned are the 
natural domain for the operators $\D c$ and $c\D$, as $c$ maps 
$\im (c\pm i)^{-1}(\D\pm i)^{-1}$ into $\Dom \D$ and similarly for $\D$.

The following theorem  provides the relationships between unbounded complete
multipliers, approximate units, and strictly positive elements for differentiable algebras, and 
gives us our strong notion of completeness. 
This strong completeness is analogous to that of a 
geodesically complete Riemannian manifold, and is much stronger than completeness
of a general complete metric space. We exemplify these statements in Section \ref{sec:met}.
\begin{thm}
\label{thm:equivs}
Let $\D:\Dom \D\subset E_{B}\rightarrow E_{B}$ be self-adjoint and regular and $\A\subset\Lip(\D)$ 
a differentiable algebra such that $[AE_{B}]= E_{B}$. Then the following are equivalent:\\
1) there exists an increasing commutative approximate unit $(u_n)\subset\A$ 
with $\Vert[\D,u_n]\Vert_\infty\to 0$;\\
2) there exists a positive self-adjoint complete multiplier $c$ for $\A$;\\
3) there is a strictly positive element $h\in\A$ with 
$\im (\D\pm i)^{-1}h=\im h(\D\pm i)^{-1}$, and constant $C>0$ with $i[\D,h]\leq Ch^{2}$.
\end{thm}
\begin{proof} 
We show that $1)\Leftrightarrow 2)$ and $2) \Leftrightarrow 3)$.

We assume 1), so that there is an increasing commutative 
approximate unit $(u_n)\subset\A$ with $[\D,u_n]\to 0$
in norm. Suppose without loss of generality that there exists 
$0<\epsilon <1$ such that $\Vert[\D,u_n]\Vert_\infty<\epsilon^{2n}$. 
Moreover, let $\{a_{i}\}_{i\in\N}$ be a subset of $\A$ whose linear span is dense, and assume 
without loss of generality that for $1\leq i\leq n$ that
$\|(u_{n+1}-u_{n})a_{i}\| <\epsilon^{2n}$. Write  $d_{n}:=u_{n+1}-u_{n}\geq 0$ and define
$$
c=\sum_{n=1}^\infty \epsilon^{-n}d_{n},
$$
which is a sum of positive elements of $\A$. Then $c$ is densely defined, 
since for fixed $a_{i}$ and $i<k<\ell$ we have
\begin{align*}\|\sum_{n=k}^{\ell}\epsilon^{-n}d_{n}a_{i}\| & \leq
\sum_{n=k}^{\ell}\epsilon^{-n}\|(u_{n+1}-u_{n})a_{i}\|
\leq \sum_{n=k}^{\ell}\epsilon^{n},
\end{align*}
which goes to zero as $k\rightarrow\infty$ and therefore $ca_{i}\in\A$. Moreover, $c$ is 
obviously symmetric, so by Corollary \ref{symmul} it suffices to 
show that the resolvents $(c\pm i)^{-1}$ are densely defined and bounded. 
Consider the truncations $c_{k}:=\sum_{n=1}^{k}\varepsilon^{-n}d_{n}\in \A$. 
By Proposition \ref{lipmult}, $\mathbb{M}(\mathcal{A})$ is spectral 
invariant in $\mathbb{M}(A)$, so as the operators $c_{k}\pm i\in \mathbb{M}(\A)$ 
are invertible in $\mathbb{M}(A)$, the resolvents $(c_{k}\pm i)^{-1}$ are elements of $\mathbb{M}(\A)$.
Subsequently estimate 
\begin{align*}
\Vert[\D,c_{k}]\Vert_\infty =\big\Vert \sum_{n=1}^{k}\epsilon^{-n}([\D,u_{n+1}]-[\D,u_n])\big\Vert_\infty
&\leq \sum_{n=1}^{k} \epsilon^{-n}(\Vert[\D,u_{n+1}]\Vert_\infty+\Vert[\D,u_n]\Vert_\infty)\\
&\leq 2\sum_{n=1}^{k}\epsilon^{n},
\end{align*}
from which we deduce that $\sup_{k}\|[\D,c_{k}]\|_\infty<\infty$. Therefore
\[
\sup_{k}\|[\D,(c_{k}\pm i)^{-1}]\|_\infty
=\sup_{k}\|(c_{k}\pm i)^{-1}[\D,c_{k}](c_{k}\pm i)^{-1}\|_\infty\leq \sup_{k}\|[\D,c_{k}]\|_\infty<\infty,
\]
so $(c_{k}\pm i)^{-1}$ is a bounded sequence in $\mathbb{M}(\A)$. 
Moreover, for the elements $a_{i}$ we have
\[
((c_{\ell}\pm i)^{-1}-(c_{m}\pm i)^{-1})a_{i}
=(c_{\ell}\pm i)^{-1}(c_{m}\pm i)^{-1}\sum_{n=\ell}^{m}\varepsilon^{-n}d_{n}a_{i},
\]
so the sequence is strictly Cauchy, with limit $(c\pm i)^{-1}\in \mathbb{M}(\A)$, 
whence these operators are densely defined and bounded.
Hence the closure of $c$ is a positive, self-adjoint unbounded multiplier on $\mathcal{A}$. 

Now we show that properties 1)-3) of Definition \ref{Lipmult} hold true for $c$. 
For 1), we need to show that $(c\pm i)^{-1}\in \A$. We restrict
to the commutative subalgebra $\B\subset \A$ generated by the $u_n$, so that by Gelfand theory
there is a locally compact Hausdorff space $X$ with 
$\B\subset C^{*}(\{u_{n}\})\cong C_{0}(X)$ via the Gelfand transform. 
Every closed unbounded multiplier is determined by its Gelfand transform 
\cite[Thm 2.1,2.3]{Wood}. To show that  $(1+c)^{-1},(c\pm i)^{-1}\in \A$, 
since these are elements of $\M(\A)$, it suffices to a show that 
$(1+c)^{-1}$ and $(c\pm i)^{-1}\in C_{0}(X)\subset A$. 
To this end, fix $t\in (0,1)$ and consider the sets 
\[
X_{n}:=\{x\in X: u_{n}(x)\geq t\}.
\]
The $X_n$ form an increasing sequence of compact sets such that $X=\cup X_{n}$. We claim that
\[
\sum  \epsilon^{-n}d_{n}(x)\geq (1-t) \epsilon^{-k},\quad\textnormal{for } x\in X\setminus X_{k},
\]
which implies that $(c\pm i)^{-1}, (1+c)^{-1}\in C_{0}(X)$. 
For such $x\in X\setminus X_{k}$, and any $m\geq k$ it holds that
\[
\begin{split}
\sum_{n=0}^{\infty} \epsilon^{-n}d_{n}(x)\geq\sum_{n=k}^{\infty} \epsilon^{-n}d_{n}(x) 
&=\sum_{n=k}^{m} \epsilon^{-n}d_{n}(x) +\sum_{n>m} \epsilon^{-n}d_{n}(x)\\
&\geq \sum_{n=k}^{m} \epsilon^{-k}d_{n}(x)+\sum_{n>m} \epsilon^{-n}d_{n}(x)\\
&= \epsilon^{-k}(u_{m+1}-u_{k})(x) + \sum_{n>m} \epsilon^{-n}d_{n}(x)\\
&\geq  \epsilon^{-k}(u_{m+1}(x)-t) + \sum_{n>m} \epsilon^{-n}d_{n}(x),
\end{split}
\]
and since $u_{m+1}(x)\rightarrow 1$ and 
$\sum_{n>m}\epsilon^{-n}d_{n}(x)\rightarrow 0$, the estimate follows. 
To prove that point 2) of Definition \ref{Lipmult} holds, we need 
to show that the domain equality $\im (c\pm i)^{-1}(\D \pm i)^{-1}=\im (\D\pm i)^{-1}(c\pm i)^{-1}$ is true.
Observe that for each $y\in E_{B}$, the vector  $(c\pm i)^{-1}(\D \pm i)^{-1}y$ is a limit
\[
\lim_{k\rightarrow\infty}(c_{k}\pm i)^{-1}(\D\pm i)^{-1} y.
\]
Writing
\[ 
(c_{k}\pm i)^{-1}(\D\pm i)^{-1}
=(\D\pm i)^{-1}(c_{k}\pm i)^{-1}+(\D \pm i)^{-1}(c_{k}\pm i)^{-1}[\D,c_{k}](c_{k}\pm i)^{-1}(\D\pm i)^{-1},
\]
and recalling that the sequence $[\D,c_{k}](c_{k}\pm i)^{-1}(\D\pm i)^{-1}$ 
is uniformly bounded in operator norm, it follows that
\[
\lim_{k\rightarrow\infty}(c_{k}\pm i)^{-1}[\D,c_{k}](c_{k}\pm i)^{-1}(\D\pm i)^{-1}y
=(c\pm i)^{-1}[\D,c](c\pm i)^{-1}(\D\pm i)^{-1}y\in \im (c\pm i)^{-1}.
\]
Thus  $\im (c\pm i)^{1-}(\D \pm i)^{-1}\subset\im (\D\pm i)^{-1}(c\pm i)^{-1}$. 
The other inclusion is proved in the same way by writing
\[
(\D\pm i)^{-1}(c_{k}\pm i)^{-1}=(c_{k}\pm i)^{-1}(\D\pm i)^{-1}
+ (c_{k}\pm i)^{-1}(\D\pm i)^{-1}[\D,c_{k}](c_{k}\pm i)^{-1}(\D\pm i)^{-1}.
\] 
Lastly, to prove that point 3) of Definition \ref{Lipmult} holds, observe that
the commutator $[\D,c]$, defined on $\im (c\pm i)^{-1}(\D\pm i)^{-1}$, 
is bounded because it is the strong limit of the operators $[\D,c_{k}]$ 
on this subset, and $\sup_{k}\|[\D,c_{k}]\|_\infty$ is bounded.

Conversely, let $c$ be an unbounded positive complete multiplier for $\A$. Let $f_n:\R\to\R$ be 
given by $f_n(x)=e^{-x/n}$. 
For $y\in \Dom\D$
$$
[\D,f_n(c)]y=\int_0^1\frac{d}{ds}\left(e^{-c(1-s)/n}\D e^{-cs/n}y\right)ds
=-\frac{1}{n}\int_0^1e^{-c(1-s)/n}[\D,c]e^{-cs/n}y\, ds,
$$
and since both sides are bounded, this equality extends to all of $E_{B}$. Moreover
$$
\Vert[\D,f_n(c)]\Vert_\infty\leq \frac{1}{n}\Vert [\D,c]\Vert_\infty,
$$
so that $[\D,f_n(c)]\to 0$ in norm as $n\to\infty$.
Finally we need to see that the $f_n(c)$ define an approximate unit. The density
of $(c\pm i)^{-1}\A$ in $\A$ says that the inclusion of the commutative 
subalgebra $\mathcal{C}$ generated by $(c\pm i)^{-1}$
in $\A$ is essential. Since  $(f_n(c))$ is obviously an approximate unit for $\mathcal{C}$,
we are done.

To see that 2) and 3) are equivalent, let $c$ be an unbounded 
multiplier on $\A$ and set $h:=(1+c)^{-1}$, which is 
positive with dense range in $\A$. On the other hand, if $h\in\A$ 
is positive with dense range, then $c:=h^{-1}$ is densely defined on
$\Dom  h^{-1}={\rm Im} \,h$. The domain condition follows 
from the fact that $(h^{-1}\pm i)^{-1}=h(1\pm ih)^{-1}$, and 
$1\pm ih\in\M(\A)$ is invertible, so that
$(1\pm ih)^{-1}$ maps $\Dom \D=\im (\D\pm i)^{-1}$ 
bijectively onto itself. Then
\[
\im h(1\pm ih)^{-1}(\D\pm i)^{-1}=\im h(\D\pm i)^{-1}=\im (\D\pm i)^{-1}h=\im (\D\pm i)^{-1}h(1\pm ih)^{-1}.
\]
From the further assumption that $i[\D,h]\leq Ch^{2}$, it follows that for $e\in\im h(\D \pm i)^{-1}h$
\[
\langle i[\D,h^{-1}]e,e\rangle =-\langle h^{-1}i[\D,h]h^{-1}e,e\rangle
=\langle i[\D,h]h^{-1}e,h^{-1}e\rangle \leq C\langle he,h^{-1}e\rangle=C\langle e,e\rangle.
\]
Taking a sequence $hy_{n}\rightarrow y\in E_{B}$, boundedness on the whole of 
$\im h(\D \pm i)^{-1}$ follows. 
\end{proof}

\section{Metric completeness via approximate units}
\label{sec:met}
We recall that if $(\A,\H,\D)$ is a unital spectral triple, then the formula
\begin{equation}
d(\phi,\psi):=\sup\{|\phi(a)-\psi(a)|:\,\Vert[\D,a]\Vert\leq 1\},\quad \phi,\,\psi\in\mathcal{S}(A)
\label{eq:french-metric}
\end{equation}
defines a metric on the state space $\mathcal{S}(A)$ of $A$ provided that the set
\begin{equation}
B:=\{[a]\in A/\C1:\,\Vert[\D,a]\Vert\leq 1\}
\label{eq:born-to-be-bounded}
\end{equation}
is bounded. In the non-unital case we do not need to consider the quotient Banach space $A/\C1$
in Equation \eqref{eq:born-to-be-bounded},
just $A$, and again the same conditions guarantee that we obtain a bounded metric. 
It is known that in the unital case the resulting metric topology agrees with the weak$^*$
topology provided that $B$ is pre-compact, \cite{Lat,Rieffel}. We refer to the formula in 
Equation \eqref{eq:french-metric} as Connes' formula.

One would like to define unbounded metrics so that they restrict to bounded metrics on
each weak$^*$ compact subset of $\mathcal{S}(A)$, but it turns out that this is too strong.
Latr\'emoli\`ere identifies a class of tame compact subsets 
for which this is possible, \cite[Definition 2.28]{Lat2}, and shows
by example that not all compact subset of $\mathcal{S}(A)$ are tame. As well as the
difficulty in discussing the weak$^*$-topology, examples show that 
there is also the need to consider extended
metrics, so that points can be at infinite distance.

Our initial results concerning completeness of metric spaces rely on a weaker notion of 
approximate unit than we needed earlier, though we will see below how these various notions 
are related. For now, given a (symmetric) spectral triple $(\A,\H,\D)$,
we say that $(u_n)\subset \A\subset A$ is an 
\emph{adequate approximate unit} if $(u_n)$ is a sequential 
approximate unit for $A$ (in its $C^*$-norm topology) and $\sup\Vert[\D^*,u_n]\Vert_\infty<\infty$.
This is a weaker notion than a bounded approximate unit for $\A$.

\begin{prop} 
\label{prop:metric}
Let $(X,d)$ be a metric space, $\A=\Lip_0(X)$ be the algebra of Lipschitz
functions vanishing at infinity and $A=C_0(X)$. Let
$(\A,\H,\D)$ be a symmetric spectral triple such that for all $a\in \Lip_0(X)$
$$
C_1\Vert a\Vert_{\Lip,d}\leq \Vert[\D^*,a]\Vert_\infty\leq C_2 \Vert a\Vert_{\Lip,d}
$$
where $0< C_1\leq C_2<\infty$ are constants and $\Vert a\Vert_{\Lip,d}$ is the Lipschitz seminorm
of $a\in \Lip_{0}(X)$. If $\A$ has an adequate approximate unit, 
then $(X,d)$ is metrically complete.
\end{prop}

\begin{rmk} 1) The condition of the theorem implies that 
Connes' formula, Equation \eqref{eq:french-metric}, 
defines a metric $d^C$ which is bi-Lipschitz equivalent
to $d$. \\
2) The algebra $\Lip_0(X)$ is typically not 
separable in the Lipschitz norm, \cite{Weaver}, but our results also apply to closed
separable subalgebras of $\Lip_0(X)$, such as our differentiable 
algebras, cf. Definition \ref{def:diff-alg}. More examples are presented below.
\end{rmk}

\begin{proof}
We give the proof in the self-adjoint case, as the symmetric case is the same.
We will prove that if $(X,d)$ is not complete then for any 
sequential approximate unit  $(u_k)\subset \Lip_0(X)$ for $C_0(X)$,
the sequence $\Vert [\D,u_k]\Vert_\infty$ is unbounded, and so  
$(u_k)$ can not be an adequate approximate unit. 
Since the metric $d$ is bi-Lipschitz equivalent
to Connes' metric, for any $y,\,z\in X$ we have
$$
|u_k(y)-u_k(z)|\leq C\Vert [\D,u_k]\Vert_\infty\, d(y,z)
$$
for a constant $C>0$. Now let $x$ 
be in the metric completion $\overline{X}$ of $X$, and $x\not\in X$. Let $1/2>\epsilon>0$,
fix $y\in B_{1/n}(x)\cap X$ and let $k$ be large enough so that $u_k(y)>1-\epsilon$.
This is possible since $u_k$
is an approximate unit. Now let 
$z\in B_{1/n}(x)\cap X$ be such that $u_k(z)<\epsilon$, possible since
$u_k$ vanishes at infinity. Then for this choice of $k$ and $y,\,z\in B_{1/n}(x)\cap X$
$$
1-2\epsilon<|u_k(y)-u_k(z)|\leq C\Vert [\D,u_k]\Vert_\infty\, d(y,z)< \frac{2C}{n}\Vert[\D,u_k]\Vert_\infty.
$$
Hence we
see that for any $n$ there is a $k=k(n)$ such that 
$$
\frac{n(1-2\epsilon)}{2C}<\Vert[\D,u_k]\Vert_\infty.
$$
Since this is true for any approximate identity $(u_k)\subset \Lip_0(X)$, we are done.
\end{proof}

\begin{corl} 
Let $(M,g)$ be a Riemannian spin$^c$ manifold, $A=C_0(M)$, $\A=\Lip_0(M)$, and
$(A,L^2(M,S),\D)$ the Dirac  spectral triple of the spin$^c$ structure. If $\A$ has an adequate
approximate unit then
the Riemannian manifold $(M,g)$ is geodesically complete,
and $\D$ is self-adjoint.
\end{corl}
\begin{proof}
The point here is that $(M,g)$ need not, a priori, 
be complete, in particular it may be the interior of a manifold
with boundary. First we recall that by \cite{ConnesMetric}, the norm 
$\Vert [\D^*,f]\Vert_\infty$ is equal to the Lipschitz
norm of $f$ (with respect to the geodesic distance)
for all $f\in \Lip_0(M)$. Thus we can apply Proposition \ref{prop:metric} to obtain the first statement.
In particular if such an approximate unit exists, $(M,g)$ is metrically complete, and so 
geodesically complete by the Hopf-Rinow theorem.

The self-adjointness of the Dirac operator now follows as in \cite[Prop 10.2.10]{HR}.
\end{proof}

In this last result we managed to deduce self-adjointness of a (potentially) symmetric operator
using just an adequate approximate unit, whereas Corollary \ref{cor:symmetric-spec} requires
the existence of an honest bounded approximate unit for the Lipschitz topology. This is essentially
due to the special form of the geodesic metric on a Riemannian manifold. The Hopf-Rinow
theorem says that 
completeness implies
that `topological infinity' is at infinite distance.

The issues are perhaps best seen as follows. For any metric space $(X,d)$,
we obtain a new metric space of bounded diameter by taking the new
metric $\tilde{d}=d/(1+d)$. Then one can 
check that $(X,d)$ is complete if and only if $(X,\tilde{d})$
is complete. The identity map on $X$ is 
typically not a bi-Lipschitz map between these metric
spaces, and the property of having an adequate 
approximate unit whose Lipschitz constants go to zero
is not preserved by this operation.

We collect a few examples from the world of metric spaces 
about approximate units for Lipschitz algebras and differentiable algebras. 
The first result is rather negative.

\begin{lemma}
\label{lem:fin-dim-bad}
Let $(X,d)$ be a finite-diameter, noncompact, complete metric space. Then 
there is no adequate approximate unit in $\Lip_0(X)$ whose Lipschitz 
constants go to zero. 
\end{lemma}
\begin{proof}
Let $(u_n)$ be an approximate unit in $\Lip_0(X)$.
Since $(u_n)$ is a norm approximate unit, for any $x\in X$ and $1/2>\delta>0$ we can
find $N$ such that $u_N(x)>1-\delta$. Since $u_N$ vanishes at infinity we can find $y\in X$ such
that $u_N(y)<\delta$. Then $u_N(x)-u_N(y)>1-2\delta$, and as $d(x,y)\leq {\rm diam}(X)$ we find that
$$
\frac{u_N(x)-u_N(y)}{d(x,y)}>\frac{1-2\delta}{d(x,y)}\geq \frac{1-2\delta}{{\rm diam}(X)}.
$$
Hence the Lipschitz norm of the $u_N$'s is bounded below.
\end{proof}
Hence finite diameter complete spaces do not have spectral triples which both recover
the metric and satisfy the conditions of Theorem \ref{thm:equivs}. Also observe that we did not 
ask for an approximate unit $(u_n)$  for $\Lip_0(X)$ in the Lipschitz topology
with $\Vert u_n\Vert_{\Lip,d}\to 0$. These typically do not exist.
\begin{lemma}
\label{lem:R-is-bad}
Let $(u_n)\subset C_c^\infty(\R)$ be a differentiable 
approximate unit for the supremum norm topology on $C_0(\R)$ such that
the Lipschitz constants go to zero as $n\to\infty$ (these exist). Then $(u_n)$ is not an approximate
unit for the Lipschitz topology on $\Lip_0(\R)$.
\end{lemma}
\begin{proof}
Let $f(x)=\sin(x^3)/(1+x^2)\in \Lip_0(\R)$. 
The mean value theorem says that given $x,y\in\R$ there is some
$w$ between $x$ and $y$ such that
$$
\left|(f-u_nf)(x)-(f-u_nf)(y)\right|=|(f-u_nf)'(w)|d(x,y)=|(1-u_n(w))f'(w)-u'_n(w)f(w)|d(x,y).
$$
Since the Lipschitz constants of the $u_n$ converge to zero, and 
$u_n'\to 0$ uniformly, we see that $u'_nf\to 0$ uniformly.
As the derivative of $f$ is $f'(x)=3x^2\cos(x^{3})/(1+x^2) - 2x\sin(x^3)/(1+x^2)$, and $u_n$ vanishes
at infinity for each $n$, we see that $|(1-u_n(w))f'(w)|$ does not go to zero uniformly.
\end{proof}
\begin{rmk}
The function $f(x)=\sin(x^3)/(1+x^2)$ also appears in \cite[p 43]{CGRS2}, to demonstrate that 
derivatives must be controlled to handle summability in the nonunital setting.
\end{rmk}
Despite this lack of success, even with our strongest completeness condition, 
there are positive results, and these demonstrate the need to take smaller algebras than 
$\Lip_0(X)$.
Recall, \cite{DJ}, the pointwise Lipschitz constant
of a function $f$ at a non-isolated point $x\in X$ defined by
$$
\Lip(f)(x):=\limsup_{y\to x,\,y\neq x}\frac{|f(x)-f(y)|}{d(x,y)}.
$$
If $x$ is isolated we set $\Lip(f)(x)=0$. Then we set 
$$
\textnormal{L}_{00}(X)=\{f\in \Lip(X):\,f\mbox{ and }\Lip(f)\mbox{ vanish at infinity}\}.
$$
The function $\Lip(f)$ need not be continuous, but we can still ask for it to be small outside
a compact set.
The space $\textnormal{L}_{00}(X)$ is not always a Banach space in its natural norm 
$\|f\|_\infty+\|\Lip(f)\|_{\infty}$, but we can take its completion, which is a 
subspace of $C_{0}(X)$. We denote this Banach space by $\Lip_{00}(X)$.
\begin{lemma}
\label{lem:easy-peasy}
Let $(X,d)$ be a metric space and  $(u_n)\subset \Lip_{00}(X)$ an adequate approximate unit 
such that $\Vert u_n\Vert_{\Lip}\to 0$ as $n\to\infty$. Then $(u_n)$ is an
approximate identity for $\Lip_{00}(X)$.
\end{lemma}
\begin{proof}
We just need to show that for $f\in \Lip_{00}(X)$ we have
$\Vert f-u_nf\Vert_{\Lip}\to 0$. That is, we need to show that
$$
\sup_{x\neq y}\left|\frac{(f-u_nf)(x)-(f-u_nf)(y)}{d(x,y)}\right|\to 0\ \mbox{as}\ n\to\infty,
$$
which is to say, we need to show that
\begin{align*}
&\sup_{x\neq y}\left|\frac{(u_n(x)-u_n(y))f(x)-(f(y)-f(x))(u_n(y)-1)}{d(x,y)}\right|\\
&\qquad\qquad\qquad\leq \Vert u_n\Vert_{\Lip}\,\Vert f\Vert_\infty  +
\sup_{y\in X}\left|\Lip(f)(y)(u_n(y)-1)\right| \, \to 0\ \mbox{as}\ n\to\infty,
\end{align*}
the second term going to zero since $\Lip(f)$ vanishes at infinity.
\end{proof}
The last two lemmas show why we need to be 
able to restrict to closed subalgebras of $\Lip(\D)$ which
may be smaller than $\Lip(\D)\cap A$, but which are still norm dense in $A$. 
For general metric spaces it is not clear that one can
always find suitable algebras which have adequate approximate units. 
When the metric is suitably infinite and the metric space nice enough, we can find
approximate units for $\Lip_{00}(X)$. This result 
resembles the equivalences of Theorem \ref{thm:equivs}, and captures the idea that
topological infinity is at infinite distance.
\begin{prop}
\label{prop:actually-infinite}
Let $(X,d)$ be a metric space 
and $x_0\in X$ such that the function $x\mapsto d(x_0,x)$ is proper. Then we obtain an
approximate unit for $\Lip_{00}(X)$ whose Lipschitz constants go to zero. Hence $(X,d)$ is complete.
\end{prop}
\begin{proof}
Fix $x_0\in X$ as in the statement, and 
let 
$$
K_N=\{x\in X: \, d(x,x_0)\leq N\} .
$$
Then the $K_N$ form an increasing sequence of compact sets whose union is $X$.
Define functions on $X$ by
$$
u_N(x)=\left\{\begin{array}{ll} 1 & x\in K_N\\ 
\frac{N}{N-1}\left(1-\frac{d(x_0,x)}{N^2}\right) & x\in K_{N^2}\setminus K_N\\
0 & x\not\in K_{N^2}\end{array}\right..
$$
Checking the various cases shows that each $u_N\in \Lip_{00}(X)$ 
is a bounded Lipschitz function whose Lipschitz
constant is bounded by $1/N(N-1)$, 
and so $\Vert u_N\Vert_{\Lip}\to 0$ as $N\to\infty$. 
Moreover it is clear that $(u_N)$ is a sup norm approximate unit for
$C_0(X)$, and so by Proposition \ref{prop:metric} and Lemma \ref{lem:easy-peasy} we are done.
\end{proof}
For $\R^n$, and more generally geodesically complete manifolds $M$, we can always construct
an approximate unit as in Proposition \ref{prop:actually-infinite}. 
As a consequence, we can construct a bounded approximate unit for $\Lip_{00}(M)$, and
then Corollary \ref{cor:symmetric-spec} tells us directly that Dirac-type operators on $M$ are
self-adjoint.

Given a spectral metric space, we can still 
deduce the completeness of the state space $\mathcal{S}(A)$ 
from the existence of an adequate approximate unit, 
as was first shown by Latr\'emoli\`ere \cite{Lat}
for the case of bounded metrics.
As the context is somewhat different, we give the argument. 
\begin{prop} 
\label{prop:spec-met}
Let $(\A,\H,\D)$ be a symmetric spectral triple for which Connes' formula
\[
d(\sigma,\tau):=\sup\{|\sigma(a)-\tau(a)|: \|[\D^*,a]\|_\infty\leq 1\},
\] 
defines an extended 
metric on the state space $\mathcal{S}(A)$ (so $d$ may take the value $\infty$).
If $\A$ has an adequate approximate unit then $(\mathcal{S}(A),d)$ is complete.
\end{prop}
\begin{proof} 
Let $(u_{n})$ 
be an adequate approximate unit. Let $\sigma_{k}$ be a sequence 
of states that is Cauchy for the Connes metric, i.e. for $k<\ell$
\[
\sup\{|\sigma_{k}(a)-\sigma_{\ell}(a)|: \| [\D^*,a]\|_\infty\leq 1\}\rightarrow 0,
\]
as $k\rightarrow\infty$. Then $\sigma(a):=\lim_{k}\sigma_{k}(a)$, for 
$a\in \A$, is a well defined map $\A\rightarrow \mathbb{C}$. 
It is positive since for positive $a$, $\sigma(a)$ is a limit of positive 
numbers. It remains to show that $\sigma$ has norm $1$. To this end, 
let $a\in\A$ be in the unit ball for the $C^{*}$-norm. Then 
$|\sigma_{k}(a)|\leq 1$, so $|\sigma(a)|\leq 1$, showing that 
$\|\sigma\|\leq 1$, and thus $\sigma$ extends to all of $A$. 
Now since $u_{n}$ is an approximate unit, we have 
$\sigma_{k}(u_{n})\rightarrow 1$ for fixed $k$ and 
$n\rightarrow\infty$. Since $[\D^*,u_{n}]$ is bounded, 
we may assume that $\|[\D^*,u_{n}]\|_\infty\leq C$ for all $n$ and some positive constant $C$. 
This means that for $k<\ell$
\[
\sup_{n}|\sigma_{k}(u_{n})-\sigma_{\ell}(u_{n})|\rightarrow 0,
\]
as $k\rightarrow \infty$. 
Hence there exist $\varepsilon>0$ and $k$ 
sufficiently large such that for all $n$ 
\[
|\sigma(u_{n})-\sigma_{k}(u_{n})|<\varepsilon/2.
\] 
Now choose $n$ large enough such that $\|\sigma_{k}(u_{n})-1\|<\varepsilon/2$. Then
\[
|\sigma(u_{n})-1|\leq |\sigma(u_{n})-\sigma_{k}(u_{n})|+|\sigma_{k}(u_{n})-1|\leq\varepsilon.
\]
This shows that $\sigma(u_{n})\rightarrow 1$, and in particular that 
$\|\sigma\|=1$ and $\sigma\in\mathcal{S}(A)$.
\end{proof}

In particular, the presence of an adequate approximate unit ensures that the metric topology
limit of states is a state, and so such a sequence is a tight set, \cite[Definition 2.2]{Lat2}. It is likely
that our approach to completeness can further complement Latr\'emoli\`ere's approach to 
locally compact quantum metric spaces.

Finally, let us consider what can be said about closed subalgebras of $\Lip(\D)$ for a
general (symmetric) spectral triple $(\A,\H,\D)$.
\begin{prop}
\label{prop:CDA-CDA-CDA}
Let $(\A,\H,\D)$ be a symmetric spectral triple. 
Suppose that $A$ has an adequate approximate unit $(u_n)\subset\A$
such that $[\D^*,u_n]\to 0$ in operator norm. Then $(u_n)$ is an approximate unit for $\A$
if and only if $(u_n)$ is an operator norm topology  approximate unit for the $C^*$-algebra generated
by $A$ and the commutators $[\D^*,a]$, $a\in\A$.
\end{prop}
\begin{proof}
This just boils down to asking when $[\D^*,au_n-a]\to 0$ in operator norm. Using the Leibniz rule,
$$
[\D^*,au_n-a]=a[\D^*,u_n]+[\D^*,a](u_n-1),
$$
we obtain the result immediately.
\end{proof}

\section{Approximate units and connections on operator modules}
\label{sec:fucked-name}
Having demonstrated the usefulness of approximate identities in differentiable algebras, we now
refine our concepts to address the existence of connections on modules and the unbounded
Kasparov product. Using connections to identify explicit representatives of Kasparov products
has been used in several contexts, \cite{BMS,KaLe,LRV,Mes},
but doing this in a naive algebraic way leads to problems, as shown in \cite{Kaadabsorption,Senior}. 
\subsection{Projective modules}
For an operator algebra $\B$ with bounded approximate unit $v_{\lambda}$, the right 
$\B$-module $\H_{\B}:=\H\hotimes \B$ 
is called the \emph{standard rigged module}, \cite{Blecherrigged}. 
For notational convenience we write $\Z:=\mathbb{Z}\setminus\{0\}$.
The module $\H_{\B}$ can be concretely defined using an isometric 
representation $\pi:\B\rightarrow \bB(\H)$ as the space of column vectors 
\[
\Big\{(b_{i})_{i\in\Z}:\ b_{i}\in\B,\ \  \sum_{i\in\Z} \pi(b_{i})^{*}\pi(b_{i}) <\infty\Big\},
\]
where the sum converges in norm. From now on we fix  a $\mathbb{Z}_2$-graded 
$C^{*}$-algebra $B$ and an essential
unbounded $(B,C)$ Kasparov module $(\B,F_C,\D)$ with $\gamma$ the 
$\mathbb{Z}_2$-grading operator\footnote{We recall that 
if $C$ is non-trivially $\mathbb{Z}_2$-graded
then $\gamma$ is not adjointable as an operator on $F_C$}. We fix the representation
\[
\pi_{\D}(b):=\begin{pmatrix} b & 0 \\ [\D,b] & \gamma(b)\end{pmatrix}
\in\End_{C}^{*}(F\oplus F),\quad b\in \B,
\]
and we assume $\B$ to have a bounded approximate unit. The graded operator 
$\B^+$-module $\H_{\B^+}$ is the 
graded Haagerup tensor product of the graded Hilbert space $\ell^2(\Z)$ and the graded
algebra $\B^+$. Thus the module $\H_{\B^{+}}$ is naturally $\mathbb{Z}_2$-graded via 
\begin{equation}
\label{modgrad}
\Gamma(b_{i})_{i\in\Z}:=({\rm sign}(i)\gamma(b_{i}))_{i\in\Z},
\end{equation}
and defining the self-adjoint unitary
\[
\epsilon:\H_{\B^{+}}\to \H_{\B^{+}},\quad \epsilon (b_{i})_{i\in\Z}=({\rm sign}(i)b_{i})_{i\in\Z},
\] 
the grading operator \eqref{modgrad} on $\H_{\B^{+}}$ 
decomposes as $\Gamma:=\epsilon\diag(\gamma_{\B^{+}})=\diag(\gamma_{\B^{+}})\epsilon$.
This allows us to write the 
representation presenting $\H_{\B^+}$ as a concrete operator $\B^+$-module as
$$
(b_i)_{i\in\Z}\mapsto 
\begin{pmatrix} b_i & 0\\ {\rm sign}(i)[\D,b_i]_{\B^+} & {\rm sign}(i)\gamma_{\B^+}(b_i)\end{pmatrix}_{i\in\Z}
=\begin{pmatrix} 1 & 0\\ 0 & \varepsilon\end{pmatrix}
\begin{pmatrix} b_i & 0\\ [\D,b_i]_{\B^+} & \gamma_{\B^+}(b_i)\end{pmatrix}_{i\in\Z}.
$$
We will always consider $\H_{\B^{+}}$ where $\B^{+}$ 
is the unitisation of the differentiable algebra $\B$ (cf. Definition \ref{unitise}).

The  
\emph{compact operators} $\kK(\H_{\B^{+}})$ on $\H_{\B^+}$ are defined to be 
the Haagerup tensor product   $\kK\hotimes\B^{+}$, as defined in Equation \eqref{eq:Haag-tens}.
The algebra $\kK(\H_{\B^{+}})$ has a bounded approximate unit
\[
\chi_{n}=\sum_{1\leq |i| \leq n} | e_{i}\rangle\langle e_{i} |,
\]
where $e_{i}$ is the \emph{standard basis} of $\H_{\B^{+}}$. 
In \cite{KaLe, Mes} it was shown that the standard $B$-valued inner product on the
module $\H_{\B^{+}}$ actually takes values in $\B^+$. Then one  defines 
the \emph{adjointable operators} $\End^{*}_{\B^{+}}(\H_{\B^{+}})$ 
as the algebra of completely bounded maps $T:\H_{\B^{+}}\rightarrow\H_{\B^{+}}$ 
that admit an adjoint with respect to the standard inner product, so that
\begin{equation}\label{adjointable}
\langle Te, f\rangle=\langle e, T^{*}f\rangle.
\end{equation}
In \cite{KaLe} the class of submodules of $\H_{\B^{+}}$ defined by 
projection operators in $\End^{*}_{\B^{+}}(\H_{\B^{+}})$ are 
called \emph{operator $*$-modules}, and were classified by Kaad in \cite{Jens} 
for the case of commutative $\B$. 
The class of \emph{stably rigged modules} 
discussed in \cite{Mes} is essentially the same. In \cite{BMS}, this 
class is enlarged by incorporating countable direct sums of 
projections in $\End^{*}_{\B^{+}}(\H_{\B^{+}})$. The present 
paper further broadens the class of modules that can be 
used to construct the Kasparov product, refining the approximate unit techniques of \cite{Mes}. 

In \cite{BMS} the notion of unbounded projection operator was 
introduced, in order to deal with the differential structure on the $C^{*}$-module 
arising from the Hopf fibration.
In this section we develop the theory of such modules beyond the case of 
direct sums of bounded projections. This will be put to use to demonstrate existence of 
connections on projective operator modules. 
\begin{defn}[cf. \cite{BMS}]\label{def: proj} 
Let $\mathcal{B}$ be an operator $*$-algebra. 
A \emph{projective operator module} is an inner product operator module 
$\mathcal{E}$ over $\B$ that is isometrically unitarily isomorphic to $p\Dom p$ 
for some possibly 
unbounded even projection in $\mathcal{H}_{\B^{+}}$, such that the 
canonical basis vectors $\{e_{i}\}_{i\in\Z}$ are contained in $\Dom p$.
Here $\E$ is regarded as a $\B^+$-module in the usual way.
\end{defn}
Note that a projective operator module $\mathcal{E}$ over $\mathcal{B}$ 
admits a canonical $C^{*}$-completion, coming from the inner product. 
Equivalently, this completion can be obtained as the Haagerup tensor 
product $\mathcal{E}\hotimes_{\mathcal{B}}B$ over the completely 
contractive inclusion $\mathcal{B}\rightarrow B$ \cite[Corollary 2.18]{BMS}. We now characterise 
when a given $C^{*}$-module $E$ over $B$ admits a projective $\B$-submodule. 
The algebra of finite rank operators on $E$ is denoted $\Fin_{B}(E)$. 
By a \emph{(homogenous) frame for} $E$ we mean a sequence $(x_{i})_{i\in\Z}$ with 
the property that 
\begin{equation}
\label{compactau}\gamma_{E}(x_{i})
=\left\{\begin{matrix} x_{i} & \textnormal{if } i > 0 \\ -x_{i} & \textnormal{if } i<0\end{matrix}\right. ,
\qquad \textnormal{and that}\qquad 
\chi_{n}=\sum_{1\leq |i|\leq n} |x_{i}\rangle\langle x_{i}| \in\Fin_{B}(E)
\end{equation}
is an approximate unit for $\Fin_{B}(E)$ with $\|\chi_{n}\|_{\End_B^*(E)}\leq 1$ (that is $(\chi_n)$
is contractive). We refer to $\chi_{n}$ as the \emph{frame approximate unit for $(x_i)$}. 
All frames will be homogenous unless stated otherwise, so that $\gamma_E(x_i)={\rm sign}(i)x_i$.
\begin{prop} 
\label{columnfinite}
Let $\mathcal{B}$ be a differentiable algebra and $E_B$ a graded 
$C^{*}$-module over the $C^{*}$-closure $B$. Then 
$E_B$ is the completion of a projective operator $\B$-module $\mathcal{E}_\B\subset E_B$  
if and only if there is a  frame $(x_{i})_{i\in\Z}$ such that each of the column vectors 
$v_{j}=(\langle x_{i},x_{j} \rangle)_{i\in\Z}$ 
has finite norm in $\H_{\B^{+}}$. We call such a frame $(x_i)$, and the associated approximate unit
$\chi_{n}$, \emph{column finite}.
\end{prop}
\begin{proof} 
$\Rightarrow$ When $\mathcal{E}_{\B}$ is projective we may 
assume that $\mathcal{E}_\B\subset\H_{\B^{+}}$ and $\{e_{i}\}$ is
the canonical basis of $\mathcal{H}_{\mathcal{B}^{+}}$, then setting 
$x_{i}=pe_{i}$ we observe that 
\[
\lim_{k\rightarrow\infty}\Big\langle pe_{i},\sum_{1\leq |j|\leq k}pe_{k}\langle pe_{k},pe_{j}\rangle\Big\rangle
= \lim_{k\rightarrow\infty}\Big\langle pe_{i},\sum_{1\leq |j|\leq k} e_{k}\langle e_{k},pe_{j}\rangle\Big\rangle,
\]
is norm convergent since $pe_{i}$ and $pe_{j}$ are in $\H_{\B^+}$. 
So $\chi_{n}:=\sum_{1\leq |i|\leq n} | pe_{i}\rangle\langle pe_{i} |$
is a column finite approximate unit for $\Fin_{B}(E)$.
\newline
$\Leftarrow$ We show that the matrix $p=(\langle x_{i},x_{j}\rangle)_{ij}$ is 
an even projection in $\mathcal{H}_{\mathcal{B}^{+}}$ with domain
\[
\Dom p:=\Big\{(b_{i})_{i\in\Z}\in\mathcal{H}_{\mathcal{B}^+}: \forall i\in \Z\  
\lim_{k\rightarrow\infty}
\Big(\sum_{1\leq |j|\leq k}\langle x_{i},x_{j}\rangle b_{j}\Big)\in\B\Big\}.
\]  
It is clear that $p$ is densely defined, as the canonical basis vectors 
$e_{i}\in\mathcal{H}_{\mathcal{B}^{+}}$ lie in the domain of $p$ by column finiteness. 
Moreover $p$ is closed. To see this, first denote by $q_{i}$ the projection onto 
the submodule spanned by the basis vector $e_{i}$. By column finiteness, 
$q_{i}p\in\End^{*}_{\mathcal{B}^{+}}(\mathcal{H}_{\mathcal{B}^{+}})$. Now if 
$\H_{\B^+}\ni z_{n}\rightarrow z$ and $pz_{n}\rightarrow h$, then 
$q_{i}pz_{n}\rightarrow q_{i}h$ and $q_{i}pz_{n}\rightarrow q_{i}pz$. 
Thus $q_{i}pz=q_{i}h$ for all $i$ and $pz=h\in\mathcal{H}_{\mathcal{B}^+}$ 
so $p$ is closed on its domain. 

Next we show that the symmetric operator 
$p$  is self-adjoint. Let $z\in\Dom p^{*}$, i.e.   there is 
$x\in\mathcal{H}_{\mathcal{B}^+}$ such that for all 
$w\in\Dom p$ we have $\langle pw,z\rangle=\langle w,x\rangle$. 
Since the basis vectors $e_{i}$ are in the domain, we can compute
\[
\begin{split}
\lim_{n}\sum_{1\leq|i|\leq n}q_{i}pz
&=\lim_{n}\sum_{1\leq |i|\leq n}e_{i}\langle e_{i},q_{i}pz\rangle 
=\lim_{n}\sum_{1\leq |i|\leq n}e_{i}\langle pe_{i},z\rangle 
=\lim_{n}\sum_{1\leq |i|\leq n}e_{i}\langle e_{i},x\rangle 
=x.
\end{split}
\]
This means that $pz=x$, so $z\in\Dom p$ and $p$ is self-adjoint. Now define
\begin{equation}
\label{Lipsubdef} 
\E_\B:=\{e\in E_B: (\langle x_{i},e\rangle)_{i\in\Z}\in\Dom p\},
\end{equation}
and observe that $x_{i}\in\E_\B$ by definition, so $\mathcal{E}_\B$ is dense in $E_B$. 
The module $\mathcal{E}_\B$ is closed in $\H_{\B^{+}}$ because a convergent net 
$e_{\lambda}\in\mathcal{E}_\B$ in particular converges in $E_B$ and therefore 
the limit must be of the form $(\langle x_{i},e\rangle)_{i\in\Z}$.
\end{proof}
Note that a column finite approximate unit is row finite as well, because of
the relation between the internal and external adjoint:
$(\pi_{\D}(\langle x_{i},x_{j}\rangle))^{*}=\flip\pi_{\D}(\langle x_{j},x_{i}\rangle)\flip^{*}$. 
\subsection{Connections and splittings}
We refine the notion of a connection on a projective 
module defined in \cite{BMS} by employing approximate units. 
The main improvement over \cite{BMS} is a new operator space topology on a projective module 
$\E_\B$, which is precisely the (weak) topology making the natural 
Grassmann connection continuous. 
Completing in this topology yields a possibly 
larger module $\E^{\nabla}_\B$. We will prove that for bounded projections,
the modules $\E_\B$ and $\E^{\nabla}_\B$ are cb-isomorphic.

Recall from \cite{BMS} the definition of the universal 1-forms $\Omega^{1}(B,\B)$ 
associated to a unital differentiable algebra $\B$, defined to be the kernel of the 
graded multiplication map 
\[
m: B\hotimes \B  \rightarrow B,\quad 
b_{1}\otimes b_{2} \mapsto \gamma(b_{1})b_{2}.
\]
In the nonunital case we use $\Omega^{1}(B^{+},\B^{+})$, so that the universal derivation
\begin{equation}
\d: \B\mapsto \Omega^{1}(B^{+},\B^{+}),\quad 
b\mapsto 1\otimes b - \gamma(b)\otimes 1,
\label{eq:dee}
\end{equation}
is well defined.
We will look at splittings of the \emph{universal exact sequence}
\[
0\rightarrow E\hotimes_{B^{+}}\Omega^{1}(B^{+},\mathcal{B}^{+})\rightarrow 
E\hotimes_\C \mathcal{B}^{+}\xrightarrow{m} E\rightarrow  0,
\]
that are compatible with the projective submodule $\mathcal{E}_\B\subset E_B$. 
Here $m:E\hotimes_\C \mathcal{B}^{+}\rightarrow E$, $m(e\hotimes b)=\gamma(e)b$ 
is the graded multiplication map. 
We adapt the algebraic notion of splitting to our setting.
\begin{defn} 
A completely bounded, graded, $\mathcal{B}^{+}$-module map 
$s:\mathcal{E}_\B\rightarrow E\hotimes_\C \B^{+}$ is a \emph{splitting} 
if  $m\circ s$ coincides with the inclusion map $\mathcal{E}_\B\rightarrow E_B$.
\end{defn}
We can now prove the analogue of the Cuntz-Quillen 
characterisation of algebraic projectivity 
\cite[Proposition 8.1, Corollary 8.2]{CQ} in the present analytic setting.
\begin{prop} 
Let $(x_i)_{i\in\Z}$ be a column finite frame 
as in Equation \eqref{compactau}, defining a projective $\B$-submodule $\E\subset E$.
The map
\begin{align}
\label{ausplit} 
s: \E_\B  \mapsto E\hotimes_\C \B^{+}, \qquad
e  \mapsto \sum _{i\in\Z} \gamma(x_{i})\otimes \langle x_{i},e\rangle,
\end{align}
defines a contractive $\B^{+}$-linear splitting of the universal exact sequence.
\end{prop}
\begin{proof} 
First we show that for $e\in\E$, $s(e)$ actually 
defines an element of $E\hotimes_\C \B^{+}$. To this 
end let $\varepsilon > 0$ and choose $n,m$ such that
\[ 
\Big\|\sum_{n\leq |i|\leq m}\pi_{\D}(\langle x_{i},e\rangle)^{*}\pi_{\D}(\langle x_{i},e\rangle)\Big\| 
< \varepsilon,
\]
which is possible because $e\in\mathcal{E}$. Now estimate
\[
\begin{split}\Big\|\sum_{n\leq |i|\leq m} \gamma(x_{i})\otimes \langle x_{i},e\rangle\Big\|^{2}_{h} 
& \leq  \Big\|\sum_{n\leq |i|\leq m} |x_{i}\rangle\langle x_{i}| \,\Big\|_{\kK(E)} 
\Big\|\sum_{n\leq |i|\leq m} \pi_{\D}(\langle x_{i},e\rangle)^{*}\pi_{\D}(\langle x_{i},e\rangle) \Big\|_{\B^{+}} \\
&\leq \Big\|\sum_{n\leq |i|\leq m}\pi_{\D}( \langle x_{i},e\rangle)^{*}\pi_{\D}(\langle x_{i},e\rangle) \Big\|_{\B^{+}} 
< \varepsilon,
\end{split}
\]
which shows that the partial sums of the series defining $s$ 
form a Cauchy sequence in the Haagerup norm, cf. Equation \eqref{eq:Haag-tens}. 
To show continuity of $s$ we  again estimate 
with the Haagerup norm to see that
\[
\begin{split}\|s(e)\|^{2}_{h}
& \leq \lim_{k\rightarrow\infty} \Big\|\sum_{1\leq |i|\leq k} 
| x_{i}\rangle\langle x_{i} |\,\Big\|_{\kK(E)}\ 
\Big\|\sum_{1\leq |i|\leq k} \pi_{\D}( \langle x_{i},e\rangle)^{*}\pi_{\D}(\langle x_{i},e\rangle) \Big\|_{\B}
\leq \|e\|^{2}_{\mathcal{E}}
\end{split}
\]
and we are done. 
\end{proof}
Recall that a \emph{connection} on a (graded, projective) operator module $\E_\B$ is a 
completely bounded linear operator 
$\nabla:\E_\B\rightarrow E\hotimes_{B^{+}}\Omega^{1}(B^{+},\B^{+})$ satisfying the 
Leibniz rule 
\[
\nabla(eb)=\nabla(e)b+\gamma(e)\hotimes \d b,
\] 
where $\d$ is defined in Equation \eqref{eq:dee}.
Associated to a splitting $s:\E_\B\rightarrow E\hotimes_\C \B^{+}$ is a universal connection
\[
\begin{split}
\nabla_{s}:\E_\B\rightarrow E\hotimes_{B^{+}}\Omega^{1}(B^{+}, \B^{+}), \qquad
e  \mapsto s(e)-\gamma(e)\otimes 1,
\end{split}
\]
which in the case of a column finite frame as in \eqref{compactau} takes the form
\[
\begin{split}
\nabla_{s}(e)=s(e)-\gamma(e)\otimes 1 
&=\sum \gamma(x_{i})\otimes \langle x_{i},e\rangle - \gamma(e)\otimes 1
=\sum \gamma(x_{i})\otimes \langle x_{i},e\rangle-\gamma(x_{i}\langle x_{i},e\rangle)\otimes 1\\
&=\sum \gamma(x_{i})(1\otimes \langle x_{i},e\rangle - \gamma(\langle x_{i},e\rangle)\otimes 1)
=\sum \gamma(x_{i})\otimes d\langle x_{i},e\rangle.\end{split}\]
On the other hand, the frame \eqref{compactau} induces a stabilisation map
\[
v:\E_\B \rightarrow \mathcal{H}_{\B^{+}} \qquad 
e \mapsto (\langle x_{i},e\rangle)_{i\in\Z},
\]
with adjoint
\[
v^{*}:\mathcal{H}_{\B^{+}} \rightarrow \E_\B \qquad
(b_{i})_{i\in\Z} \mapsto \sum_{i\in\Z} x_{i}b_{i},
\]
and  $v^{*}v={\rm Id}_\E$. The associated projection 
$p=vv^{*}$, is given by the matrix $p=(\langle x_{i},x_{j}\rangle)$. 

Recalling that the grading operator \eqref{modgrad} on $\H_{\B^{+}}$ 
decomposes as $\Gamma:=\epsilon\diag(\gamma_{\B^{+}})$,
the module 
$\mathcal{H}_{\B^{+}}$ admits a canonical connection 
$\epsilon \d:(b_{i})\mapsto\epsilon (\d b_{i})$. The isometry $v$ induces a 
connection 
$v^{*}\epsilon \d v:\E\rightarrow 
E\hotimes_{B^{+}}\Omega^{1}(B^{+},\B^{+})$, 
which we call the \emph{Grassmann connection}. These considerations prove the following lemma. 
\begin{lemma}
\label{vdv} Let $(x_{i})_{i\in\Z}$ be a (homogenous) frame and 
$v:E\to \H_{B^{+}}$ the associated isometry. 
Then $v$ is even, that is, $v(\gamma(e))=\Gamma(v(e))$. 
If $(x_{i})_{i\in\Z}$ is column finite the connection 
$\nabla_{s}:\E_\B\rightarrow E\hotimes_{B^{+}}\Omega^{1}(B^{+},\B^{+})$ 
associated to the splitting \eqref{ausplit} equals $ v^{*}\epsilon \d v$.
\end{lemma}
In order to deal with unbounded projections we 
need to extend the techniques developed in \cite{BMS} and 
introduce a slightly different operator space structure. 
To this end we need to pass from the universal derivation 
$\d$ to the represented derivation $\delta_{\D}:b\mapsto [\D,b]$ 
coming from the defining Kasparov module $(\B,F_C,\D)$ for $\B$. Recall, 
Remark \ref{essrmk},  that we assume $[BF_{C}]=F_{C}$ and 
define the unitisation $\B^{+}$ according to Definition \ref{unitise}. The closed linear 
span of \emph{represented 1-forms} is the operator space
\[
\Omega^{1}_{\D}:=\overline{\Big\{\sum_{i} \pi(b_{i})[\D,\pi (b_{i}')] : b_{i}, b_{i}'\in \B\Big\}}
\subset \End^*_C(F),
\]
which is a $(B,\B)$-bimodule.
Using the universality of the derivation $\d$, there is a completely bounded $(B,\B)$-bimodule map
\[
j_{\D}:\Omega^{1}(B^{+},\B^{+})\rightarrow \Omega^{1}_{\D},
\]
uniquely determined by $\d b\mapsto [\D,\pi(b)]$ since $\pi(1)=1$, cf. \cite[Prop 2.22]{BMS}. 
In this way we obtain a connection 
\[
\nabla_{\D}:\E_\B\xrightarrow{\nabla_{s}} E\hotimes_{B}\Omega^{1}(B^{+},\B^{+})
\xrightarrow {1\otimes j_{\D}} E\hotimes_{B}\Omega^{1}_{\D},
\] 
sometimes referred to as a \emph{represented connection}. 
In the case of a free module, the tensor product 
$\H_{\B^{+}}\hotimes_{B^{+}}\Omega^{1}_{\D}$ can be 
identified with the space $\H_{\Omega^{1}_{\D}}:=\H\hotimes \Omega^{1}_{\D}$ 
of square summable sequences of forms $(\omega_j)$, where $\sum_j\omega^*_j\omega_j$ 
converges in $\End^*_C(F)$.  
Let $(x_i)$ be a column finite frame for $\E_\B$, and consider the space 
\begin{equation}
\label{enabla}
\E^{\nabla}_\B:=\Big\{e\in E_B: \lim_{n\rightarrow \infty}
\Big(\sum_{1\leq |k|\leq n}\langle x_{i},\gamma(x_{k})\rangle[\D,\langle x_{k},e\rangle ]\Big)_{i\in\Z} 
\in \H_{\Omega^{1}_{\D}}\Big\}.
\end{equation}
This space $\E_\B^\nabla$ may be strictly larger than $\E_\B$, and is an 
operator module in the representation
\begin{equation}
\label{Ep}
\pi_{\nabla}(e):=\begin{pmatrix} v(e) & 0 \\ vv^{*}\epsilon [\D, v(e)] & v(\gamma(e))\end{pmatrix}
\in \bigoplus_{i\in\Z} \End_{C}^{*}(F\oplus F),\quad \Vert e\Vert_{\E^\nabla}:=\Vert\pi_\nabla(e)\Vert
\end{equation}
where we have used slightly abusive notation for the entrywise graded commutator with $\D$ in the
indicated column vector. We will write $\gamma$ for $\diag(\gamma_{\B^{+}})$ and 
$\D_{\epsilon}$ for the self-adjoint regular operator $\epsilon\diag(\D)$ on 
$\H_{B^{+}}\hotimes_{B^{+}}F$. There is an equality of domains $\Dom \diag(\D)=\Dom \D_{\epsilon}$, 
and the closed graded derivations
\begin{equation}\label{differentgradings}
[\diag(\D),T]_{\gamma}:=\diag(\D)T-\gamma T \gamma\diag(\D),\quad [\D_{\epsilon},T]_{\Gamma}:=\D_{\epsilon} T-\Gamma T\Gamma\D_{\epsilon},
\end{equation}
are related via
$[\D_{\epsilon},T]_{\Gamma}=[\diag(\D),\epsilon T]_{\gamma}=\epsilon[\diag(\D),T]_{\gamma}$. 
Therefore these derivations have the same domain inside 
$\End^{*}_{C}(\H_{\B^{+}}\hotimes_{B^{+}} F)$. With regards to gradings on
the module $\E_\B^\nabla$ defined in
\eqref{enabla}, observe that there is an identity
\[
\begin{split}
\Big(\sum_{1\leq |k|\leq n}\langle x_{i},\gamma(x_{k})\rangle[\D,\langle x_{k},e\rangle ]\Big)^{*}
&=\sum_{1\leq |k|\leq n} -[\D,\gamma\langle e, x_{k}\rangle ]\langle \gamma(x_{k}),x_{i})\rangle\\
&=\sum_{1\leq |k|\leq n} \gamma\big(\,[\D,\langle e, x_{k}\rangle ]\langle x_{k},\gamma(x_{i})\rangle\,\big),
\end{split}
\]
and thus for $e\in\E^{\nabla}$, the series of row vectors 
\begin{equation}
\label{secondconv}
\left(\sum_{1\leq |k|\leq n} [\D,\langle e, x_{k}\rangle ]\langle x_{k},\gamma(x_{i})\rangle\right)_{i\in\Z}^{t}\end{equation}
is convergent.
\begin{lemma}
\label{nablaincl}
Let $p:\Dom p\rightarrow \H_{\B^{+}}$ be a  projection 
such that $\E_\B=p\Dom p$ is a projective operator module,
and consider the operator module  
$\E^{\nabla}_\B$ defined in Equation
\eqref{enabla}.\\
1) There is a completely contractive dense inclusion $\iota: \E_\B\rightarrow \E^{\nabla}_\B$.\\
2) If $p\in \End^{*}_{\B^{+}}(\H_{\B^{+}})$ then $\iota$ is a
cb-isomorphism.
\end{lemma}
\begin{proof} 
The estimate
\[
\left\|\begin{pmatrix} v(e) & 0 \\ vv^{*}\e [\D, v(e) ] & \gamma(v(e)) \end{pmatrix}\right\|
=\left\|\begin{pmatrix} p& 0 \\ 0 & p \end{pmatrix}
\begin{pmatrix} v(e) & 0 \\ \epsilon [\D, v(e) ] & v(\gamma(e)) \end{pmatrix}\right\|\leq 
\left\|\begin{pmatrix} v(e) & 0 \\ \epsilon [\D, v(e) ] & \gamma(v(e)) \end{pmatrix}\right\|,
\]
proves 1). For 2), observe first that $\Gamma p=p\Gamma$, so  
$p\epsilon=p\gamma\epsilon\gamma=p\Gamma\gamma
=\Gamma p\gamma=\epsilon \gamma p \gamma=\epsilon \gamma(p)$ and 
\[ 
[\D_{\epsilon},p]_{\Gamma}v(e)=[\diag(\D),\epsilon p]_{\gamma}v(e)
=[\D,\epsilon v(e)]-\epsilon\gamma p \gamma [\D,v(e)]= [\D,\epsilon v(e)]- p\epsilon [\D,v(e)].
\] 
Thus we can write 
\[
\begin{pmatrix} 
v(e) & 0 \\ vv^{*}\e [\D, v(e)] & \gamma(v(e)) \end{pmatrix}
=\begin{pmatrix} 1 & 0 \\ -[\D_{\epsilon}, p]_{\Gamma} & 1\end{pmatrix}\begin{pmatrix} v(e) & 0 \\ \epsilon [\D,v(e)] & v(\gamma(e))
\end{pmatrix},
\]
and as $[\D_{\epsilon},p]_{\Gamma}$ is bounded, the matrix 
$\begin{pmatrix} 1 & 0 \\ -[\D_{\epsilon}, p]_{\Gamma} & 1\end{pmatrix}$ is invertible. The assertion follows.\end{proof}
\begin{prop}
\label{innerprod} 
The module $\E^{\nabla}_\B$ has the following properties.\\
1) The inner product $\E_\B\times \E_\B\rightarrow \B$ extends to an inner product 
$\E^{\nabla}_\B\times \E^{\nabla}_\B\rightarrow \B$. \\
2) For each $e\in \E^{\nabla}_\B$ the operator $e^{*}:\E^{\nabla}_\B\rightarrow \B$ 
defined by $f\mapsto \langle e,f\rangle$ is completely bounded and adjointable, 
with adjoint $b\mapsto eb$, and satisfies the estimate 
$\|e^{*}\|_{cb}\leq 2\|e\|_{\E^{\nabla}}$.
\end{prop}
\begin{proof}
For 1) we must show that for $e,f\in \E^{\nabla}_\B$ the inner product $\langle e,f\rangle\in\B$. 
Let $(x_i)_{i\in\Z}$ be a defining column finite frame for $\E_\B$.
By definition the series of column vectors
\[
\sum_{j\in\Z} (\langle x_{i},\gamma(x_{j})\rangle[\D,\langle x_{j},e\rangle])_{i\in\Z},
\]
is norm convergent for $e\in \E^{\nabla}_\B$ by \eqref{enabla}. Consider the partial sums
\[ 
[\D,\sum_{1\leq |j|\leq n} \langle e,x_{j}\rangle\langle x_{j},f\rangle ]  
=\sum_{1\leq |j|\leq n}\gamma(\langle e, x_{j}\rangle)[\D, \langle x_{j}, f\rangle ] 
+[\D,\langle e,x_{j}\rangle]\langle x_{j}, f\rangle.
\]
The two terms on the right hand side are convergent sums, 
since (using the pairing of row and column vectors)
\[
\begin{split}
\big\|\sum_{1\leq |j|\leq n}\gamma(\langle e, x_{j}\rangle)[\D, \langle x_{j}, f\rangle ] \big\|
&=\big\|\sum_{1\leq |j|\leq n}\sum_{i\in\Z} 
\langle \gamma(e), x_{i}\rangle\langle x_{i},\gamma(x_{j})\rangle[\D, \langle x_{j},f\rangle ] \big\| \\
&=\big\|\sum_{1\leq |j|\leq n}(\langle x_{i},\gamma(e)\rangle)_{i\in\Z}^{*}\cdot(\langle x_{i},\gamma(x_{j})\rangle 
[\D, \langle x_{j},f\rangle ])_{i\in\Z}\big\| \\
&\leq \|e\|_{E}\,\,\big\| \sum_{1\leq |j|\leq n}(\langle x_{i},\gamma(x_{j})\rangle [\D, \langle x_{j},f\rangle ] )_{i\in\Z}\big\|,
\end{split}
\]
and similarly for the other term. Since 
$\sum_{i\in \Z}\langle e,x_{i}\rangle\langle x_{i},f\rangle$ converges to 
$\langle e,f\rangle$ and $[\D,\cdot ]$ is a closed derivation, it follows that 
$\langle e,f\rangle \in\B$.  Therefore we can write
\[
\begin{split} 
\begin{pmatrix} \langle e,f\rangle & 0 \\ [\D, \langle e, f \rangle] &\gamma \langle e,f\rangle\end{pmatrix} 
=\sum_{i\in\Z} &\begin{pmatrix} \langle e,x_{i}\rangle & 0 \\ 0& \langle \gamma(e),x_{i}\rangle\end{pmatrix}\begin{pmatrix} \langle x_{i},f\rangle & 0 \\ \sum_{j\in\Z} \langle x_{i},\gamma(x_{j})\rangle [\D,\langle x_{i}, f\rangle ] & \langle x_{i}, \gamma(f)\rangle\end{pmatrix} \\ 
&+\begin{pmatrix} 0 & 0 \\ \sum_{j\in\Z} [\D, \langle e,x_{j}\rangle ]\langle x_{j},\gamma(x_{i})\rangle & 0\end{pmatrix}\begin{pmatrix}  \langle \gamma(x_{i}),f\rangle & 0\\ 0 & \langle x_{i},f\rangle \end{pmatrix}. 
\end{split}
\]
These series are convergent in view of  \eqref{enabla} and \eqref{secondconv} and the equalities also hold when $f$ is a matrix of elements of $\E^{\nabla}_\B$ yielding the estimate 
\[
\|e^{*}(f)\|_{\B}\leq (\|e\|_{E}\|\|f\|_{\E^{\nabla}}
+ \|e\|_{\E^{\nabla}}\|f\|_{E})\leq 2\|e\|_{\E^{\nabla}}\|f\|_{\E^{\nabla}}, 
\]
whence $\|e^{*}\|_{cb}\leq 2\|e\|_{\E^{\nabla}}$. 
\end{proof}
\subsection{Complete projective modules}
We now define several algebras of operators on the modules 
$\E_\B$ and $\E^{\nabla}_\B$. Recall that $\E_\B\subset \E^{\nabla}_\B$ 
is a proper submodule in general by Lemma \ref{nablaincl}. As for $C^{*}$-modules, 
we denote the space of finite rank operators by $\Fin_{\mathcal{B}}(\mathcal{E})$. We give
$\Fin_{\mathcal{B}}(\mathcal{E})$
the operator algebra structure determined by regarding $\Fin_{\mathcal{B}}(\mathcal{E})$
as an algebra of operators on $\E^\nabla_\B$:
\begin{equation}
\label{compactrep}
\pi_{\nabla}(K):=\begin{pmatrix} vKv^{*} & 0 \\ p[\D_{\epsilon},vKv^{*}]p & vKv^{*}\end{pmatrix},
\end{equation}
where $\D$ comes from the defining Kasparov module $(\B,F_C,\D)$ for $\B$ and $p$ comes from a
defining column finite frame $(x_i)_{i\in\Z}$. 
For $e,f\in\E_\B$, consider the column, respectively row, vectors 
\begin{equation}
v|e\rangle =v(e)=(\langle x_{i},e\rangle)_{i\in\Z},\quad 
\langle f | v^{*}=v(f)^{*}=(\langle f,x_{i}\rangle)_{i\in\Z}^{t},
\label{eq:row-column}
\end{equation}
which are elements of $\H_{\B^{+}}$ and $\H_{\B^{+}}^t$ respectively. Thus the rank one 
operator $|e\rangle\langle f|$ is such that 
$[\D_{\epsilon},v(|e\rangle\langle f|)v^{*}]$  is a bounded matrix. 
Therefore the representation \eqref{compactrep} is well-defined on 
$\Fin_{\B}(\E)$. We emphasise that we do 
\emph{not} consider finite rank operators associated to vectors coming from $\E^{\nabla}_\B$ here.
The
ideal of \emph{compact operators} $\kK(\E^{\nabla})$ is defined to be the closure of 
$\Fin_{\B}(\E)$ in the operator space norm $\Vert\pi_\nabla(\cdot)\Vert_\infty$. We now address the 
issue of approximate units for $\kK_{\B}(\E^{\nabla})$.
\begin{lemma}
\label{lem: cliffconv} 
Let $(\B,F_C,\D)$ be the defining Kasparov module for $\B$, and $\E_\B$ a projective 
operator module. 
For a column finite frame
approximate unit $\chi_{n}$ associated to the defining frame $(x_i)_{i\in\Z}$, any $K\in\Fin_{\B}(\E)$ 
satisfies 
\[
vKv^{*}\Dom\D_{\epsilon}\subset \Dom\D_{\epsilon},
\] 
and $[\D_{\epsilon},vKv^{*}]$ extends to a bounded 
adjointable operator in $\End^{*}_{C}(\H_{B}\hotimes_{B}F)$. 
Moreover 
\begin{align}
\label{cliffconv}
&\lim_{n\rightarrow\infty} v\chi_{n}v^{*}[\D_{\epsilon},vKv^{*}]
=vv^{*}[\D_{\epsilon},vKv^{*}],\nonumber\\
&\lim_{n\rightarrow\infty} [\D_{\epsilon},vKv^{*}]v\chi_{n}v^{*}=[\D_{\epsilon},vKv^{*}]vv^{*},
\end{align} 
in operator norm.
\end{lemma}
\begin{proof}It suffices to prove this for rank one operators 
$K=|e\rangle\langle f|$. In that case, Equation \eqref{eq:row-column}
shows that  $vKv^{*}$ is given by the infinite matrix
\[
(\langle x_{i},e\rangle\langle f,x_{j}\rangle)_{ij}\in \kK\hotimes\B,
\]
and thus is in the domain of the derivation $[\D_{\epsilon}, \cdot]$. 
The norm limits \eqref{cliffconv} are given by
\[
\begin{split}\lim_{n\rightarrow\infty}v\chi_{n}v^{*}[\D_{\epsilon},&vKv^{*}]
=\lim_{n\rightarrow\infty}\left( \sum_{1\leq |k|\leq n}\langle x_{i}, \gamma(x_{k})\rangle[\D,\langle x_{k},e\rangle\langle f, x_{j}\rangle ]\right)_{ij} \\
&=\lim_{n\rightarrow\infty}\left(\sum_{1\leq |k|\leq n}\langle x_{i}, \gamma(x_{k})\rangle\gamma\langle x_{k},e\rangle[\D,\langle f, x_{j}\rangle ]+\langle x_{i}, \gamma(x_{k})\rangle[\D,\langle x_{k},e\rangle ]\langle f,x_{j}\rangle\right)_{ij}\\
&=\big(\langle x_{i},\gamma(e)\rangle[\D,\langle f, x_{j}\rangle ]\big)_{ij} 
+\lim_{n\rightarrow\infty}\left(\sum_{1\leq |k|\leq n}\langle x_{i}, \gamma(x_{k})\rangle[\D,\langle x_{k},e\rangle ]\langle f,x_{j}\rangle\right)_{ij},
\end{split} 
\]
where the first term is a well-defined infinite matrix because $f\in\E$ and the second term is a norm convergent limit because $e\in\E\subset \E^{\nabla}$. 
The other limit is handled verbatim.
\end{proof}
Given a frame approximate unit $(\chi_n)$ for $\kK(E_B)$, denote by
$\mathscr{C}(\chi_{n})$ the convex hull of $(\chi_n)$. This is the 
algebraic convex hull, and {\em not} the closed convex hull.
\begin{defn}
\label{completemodule}
Let $(\B,F_C,\D)$ be the defining Kasparov module for $\B$, and $\E_\B$ a projective operator module with column finite frame approximate unit $\chi_{n}$. 
The module $\E_{\B}$ is a \emph{complete projective operator module} if there is 
an approximate unit $(u_n)\subset\mathscr{C}(\chi_n)$ for $\kK(E_B)$ such that the 
sequence of operators 
$p[\D_{\epsilon},vu_{n}v^{*}]p: \H_{B^{+}}\hotimes_{B^{+}}F\rightarrow \H_{B^{+}}\hotimes_{B^{+}}F$ 
converges to $0$ strictly. 
\end{defn}
This definition should be viewed in the light of property 2) of 
Proposition \ref{cor:bdd-strong} as well as Corollary \ref{cor:symmetric-spec}.
\begin{prop}
\label{cptcbau} 
Let $\E_\B$ be a complete projective module over $\B$. 
Then $\kK(\E^{\nabla})$ has a bounded approximate unit consisting of elements of $\Fin_{\B}(\E)$.
\end{prop}
\begin{proof} 
Let $\chi_n=\sum_{1\leq |i|\leq n} |x_{i}\rangle\langle x_{i}|$ be the defining
column finite frame approximate unit. Consider an 
approximate unit $(u_{n})\in \mathscr{C}(\chi_n)$ as in Definition \ref{completemodule}. 
It follows from the uniform boundedness 
principle that $\sup_{n}\|\pi_{\nabla}(u_{n})\|<\infty$: this follows because for each 
$x\in \H_{B^{+}}\hotimes_{B^{+}}F$ the sequence $p[\D,vu_{n}v^{*}]px$ 
converges, so that
\[
\sup_{n}\| p[\D_{\epsilon},vu_{n}v^{*}]px\|<\infty,\quad\textnormal{and therefore}\quad
\sup_{n}\|p[\D_{\epsilon}, vu_{n}v^{*}]p\|<\infty.
\] 
For $K\in \Fin_{\B}(\E^{\nabla})$ it then follows that
\[
p[\D_{\epsilon},vu_{n}Kv^{*}]p=p[\D_{\epsilon},vu_{n}v^{*}]vKv^{*}
+vu_{n}v^{*}[\D_{\epsilon},vKv^{*}]p\rightarrow p[\D_{\epsilon},vKv^{*}    ]p,
\]
by Lemma \ref{lem: cliffconv}. Now since $\Fin_{\B}(\E^{\nabla})\subset \kK_{\B}(\E^{\nabla})$ 
is dense and $\pi_{\nabla}(u_{n})$ is uniformly bounded, it follows that 
$\pi_{\nabla}(u_{n}K)\rightarrow \pi_{\nabla}(K)$ and 
$\pi_{\nabla}(Ku_{n})\rightarrow \pi_{\nabla}(K)$ for all $K\in\kK_{\B}(\E^{\nabla})$.
\end{proof}
We now present some sufficient conditions for a projective operator module to be complete.  
\begin{prop} 
For a projective operator module $\E_\B=p\Dom p$ with defining column finite frame $(x_i)$
and corresponding approximate unit $(\chi_n)$, each of the 
following conditions imply completeness of the module $\E_\B$:\\ 
1) there is an approximate unit $(u_n)\in \mathscr{C}(\chi_n)$ for $\kK(E_B)$ such that
the operators $p[\D_{\epsilon},u_{n}]p$ converge to $0$ in norm on the $C^{*}$-module 
$\H_{B^{+}}\hotimes_{B^+}F_C$;\\
2) the projection $p$ is a countable direct sum of finite even projections $p_{k}\in M_{2m_{k}}(\B^{+})$;\\
3) the projection $p$ is an element of $\End^{*}_{\B^{+}}(\H_{\B^{+}})$;
\end{prop}
\begin{proof} 
Because norm convergence implies strict convergence, 
1) implies that the sequence $p[\D_{\epsilon},u_{n}]p$ 
converges strictly on $\H_{B^{+}}$, hence the module $\E_\B$ is 
complete in the sense of Definition \ref{completemodule}. Thus, to prove 2), 
it is enough to show that $2)\Rightarrow 1)$.

$2)\Rightarrow 1)$ For a countable family of finite projections 
$p_{i}$ with $[\D_{\epsilon}, p_{i}]$ bounded, for each $i$ we have 
$p_{i}[\D_{\epsilon},p_{i}]p_{i}=0$ and 
\[
p_{k}=\sum_{1\leq |i|\leq m_{k}}| pe_{i}^{k}\rangle \langle pe_{i}^{k}|.
\]
By identifying the direct sum $\bigoplus_{k=0}^{\infty} (\B^+)^{2m_{k}}$ with $\H_{\B^{+}}$ and setting
$p=\oplus_i^\infty p_i$, 
we can define an approximate unit $u_n=\oplus_{i=1}^np_i$. The explicit form of $p_i$ given above
shows that $u_n$ is a subsequence of the approximate unit 
associated to the frame $(pe^i_k)$, and so in the convex hull. Then 
$p[\D,u_n]p=\sum_{i=1}^np[\D,p_i]p=\sum_{i=1}^np_i[\D,p_i]p_i=0$.

To show that 3) implies completeness, observe that 
$p\in\End^{*}_{\B^{+}}(\H_{\B^{+}})$ if and only if $p\otimes \Id_F$ 
preserves the domain of $\D_{\epsilon}$ and $[\D_{\epsilon},p\ox \Id_F]$ 
is a bounded operator. Let $q_{n}$, $i\in\Z$ denote the 
projection onto the submodule generated by the basis vectors 
$e_{i}$, $1\leq |i|\leq n$, let $x_i=pe_i$, and $\chi_n=\sum_{i=1}^n|x_i\rangle\langle x_i|$.

Now {\em on the image of $p$}, a short calculation shows that $\chi_ny=\chi_npy=pq_npy$.  
Then 
$$
p[\D_{\epsilon},\chi_{n}]p=p[\D_{\epsilon},pq_{n}]p=p[\D_{\epsilon},p]q_{n}p.
$$
The projections 
$q_{n}$ converge strongly to the identity on $\H_{B^{+}}$, and $[\D_\epsilon,p]$ is bounded. 
Therefore it follows that for any $x\in\H_{B^{+}}$,
$p[\D_{\epsilon},\chi_{n}]px=p[\D_{\epsilon},p]q_{n}px\rightarrow p[\D_{\epsilon},p]px=0,$
and so Definition \ref{completemodule} is satisfied. 
\end{proof}
\subsection{Self-adjointness and regularity}
We now come to the study of self-adjointness and regularity of induced operators 
$1\otimes_{\nabla}\D$ on tensor product modules. The setting for this construction 
is as follows. Let $(\B,F_C,\D)$ be the unbounded Kasparov module defining $\B$, 
which we recall, Remark \ref{essrmk},
is essential so that $[BF_{C}]=F_{C}$.
Given a projective module $\E_\B\subset\E^{\nabla}_\B$ with grading $\gamma$ 
one obtains an odd symmetric operator 
\begin{equation}
\label{indop}
1\otimes_{\nabla}\D: \E\otimes_{\B} \Dom \D\rightarrow E\hotimes_{B}F,
\end{equation}
via the usual formula 
$1\otimes_{\nabla}\D(e\otimes f):= \gamma(e)\otimes \D f + \nabla_{\D}(e)f$. 
We extend $1\otimes_{\nabla} \D$ to its minimal closure.

In \cite{BMS} it was shown that this operator is self-adjoint and regular in the case 
where $p$ is a direct sum of bounded projection operators. In \cite{Kaadabsorption} 
it was shown that there exist unbounded projections for which the resulting operator 
is not self-adjoint. The counterexample uses the half-line, a noncomplete metric 
space. In this section we show that for complete projective modules the induced 
operator is self-adjoint and regular, by an argument similar to that for the 
Dirac operator on a complete manifold.

Write $\partial:=v(1\otimes_{\nabla}\D) v^{*}$ 
with domain and definition
\begin{equation}
\label{domain}
\Dom \partial=v\Dom (1\otimes_{\nabla}\D)\oplus (1-p)\H_{B^{+}}\hotimes_{B^+} F,\quad 
\partial (vy+(1-p)z)=v^{*}(1\otimes_{\nabla}\D)y.
\end{equation}
We have $\G(\partial)\subset (\H_{B^{+}}\hotimes_{B^+}F)^{\oplus 2}$, 
and the graph of the adjoint operator $\partial^{*}$ is 
given by $\G(\partial^{*}):=\flip\G(\partial)^{\perp}$, where we recall that  
$\flip=\begin{pmatrix} 0 & -1 \\ 1 & 0\end{pmatrix}$. 
\begin{lemma} 
The operator $1\otimes_{\nabla}\D$ is self-adjoint and regular on 
$\Dom (1\otimes_{\nabla}\D)$ if and only if the operator 
$\partial$ is self-adjoint and regular on $\Dom\partial$.
\end{lemma}
\begin{proof} 
Recall that a closed, densely-defined symmetric operator 
$T:\Dom T\to E$ is self-adjoint and regular if and only if the operators
$T\pm i:\Dom T\to E$ have dense range, 
cf. \cite[Lemma 9.7, 9.8]{Lance} and \cite[Proposition 4.1]{KaLe2}.

Suppose that $1\otimes_{\nabla} \D \pm i$ have dense range. 
Then, for $x=vy+(1-p)z$ with $y\in \Dom (1\otimes_{\nabla}\D)$ we have
\[
(\partial\pm i)x=v(1\otimes_{\nabla}\D \pm i)y+ i(1-p)z,\quad (1\otimes_{\nabla}\D\pm i)y
=v^{*}(\partial\pm i)x.
\]
Since $\im v$ and $\im (1-p)$ are orthogonal, 
it follows that $\partial\pm i$ has dense range 
in $\H_{B^{+}}\hotimes_{B^{+}} F$ if and only if 
$(1\otimes_{\nabla}\D)\pm i$ has dense range in $E\hotimes_{B}F$.
\end{proof}
We now prove that $\partial$ is self-adjoint and regular. 
Since the representation of $B$ on $F_{C}$ is assumed to be 
essential, we have the identification
\begin{equation}\label{stablesum}
\H_{B^{+}}\hotimes_{B^{+}}F_{C}\xrightarrow{\sim}\bigoplus_{i\in\Z} F_{C},
\end{equation}
and the self-adjoint regular operator  $\D_{\epsilon}$ coincides with the operator 
$1\otimes _{d} \D$ on $\H_{B^+}\hotimes_{B^+} F$ and $\epsilon\d$ the trivial connection.
\begin{lemma}
\label{connv} 
Let $(\B,F_C,\D)$ be the defining unbounded Kasparov 
module for $\B$, with $[BF_{C}]=F_{C}$,  
and let  $\E_\B\subset\E^{\nabla}_\B$ be a complete projective 
module over $\B$ with grading $\gamma$ and defining frame $(x_i)_{i\in\Z}$.
For an elementary tensor 
$e\otimes f\in \E\otimes_{\B^{+}} \Dom\D$, $(1\otimes_{\nabla}\D)(e\ox f)$ is given by the
formula
\begin{eqnarray}
 \gamma(e)\otimes \D f + \nabla(e)f
&=& \sum_{i\in\Z} x_{i}\otimes \langle x_{i},\gamma(e)\rangle \D f
+\gamma( x_{i})\otimes [\D,\langle x_{i},e\rangle ] f \nonumber\\ \label{connectionexpansion}
&=& \sum_{i\in\Z} x_{i}\otimes \langle x_{i},\gamma(e)\rangle \D f
+ \sum_{i,j\in\Z} x_{i}\otimes \langle x_{i},\gamma(x_{j})\rangle [\D,\langle x_{j},e\rangle ] f \qquad \\
&=&\sum_{i\in\Z} x_{i}\otimes \D \langle \gamma(x_{i}),e\rangle f= \sum_{i\in\Z} \gamma(x_{i})\otimes \D \langle x_{i},e\rangle f.
\label{simple}
\end{eqnarray}
More symbolically, $1\otimes_{\nabla}\D =v^*\partial v=v^{*}\D_{\epsilon} v$ and 
$\partial=p\D_{\epsilon}p$ on $v\E\otimes_{\B^{+}} \Dom \D$. The map
\begin{equation} 
\label{graphmap}g:
\E^{\nabla}\hotimes_{\B} \G(\D)  \rightarrow  \G(1\otimes _{\nabla}\D),\quad 
e\otimes \begin{pmatrix} f \\ \D f \end{pmatrix} \mapsto 
\begin{pmatrix} e\otimes f \\ (1\otimes _{\nabla}\D) (e\otimes f)\end{pmatrix}, 
\end{equation}
is a completely contractive operator with dense range.
\end{lemma}
\begin{proof} 
First we show that the sum \eqref{connectionexpansion} is convergent, so 
that the map \eqref{graphmap} is well-defined. The first term on the right hand side of 
\eqref{connectionexpansion} converges trivially. For the second term we prove
slightly more, estimating for finite sums $\sum_k e_k\ox f_k\in \E^\nabla\otimes_\B \Dom \D$
\begin{equation}
\begin{split}
\Big\|\sum_{i,j,k} x_{i}\otimes \langle x_{i},\gamma(x_{j})\rangle 
&[\D,\langle x_{j},e_{k}\rangle ] f_{k} \Big\|^{2}_{h} \\
&\leq\Big\|\sum_{i\in\Z}|x_{i}\rangle\langle x_{i}| \Big\|_{\kK(E)}\,
\Big\|\Big(\sum_{j,k\in\Z}\langle x_{i},\gamma(x_{j})\rangle [\D,\langle x_{i},e_{k}\rangle ] f_{k}\rangle\Big)_{i\in\Z}\Big\|^{2} \\
&\leq \|\sum_{k}\pi_{\nabla}(e_{k})\pi_{\nabla}(e_{k})^{*}\|\,\|\sum_{k}\langle f_{k},f_{k}\rangle \|\\
&\leq \|\sum_{k}\pi_{\nabla}(e_{k})\pi_{\nabla}(e_{k})^{*}\| \left\|\sum\left\langle \begin{pmatrix} f_{k} \\ \D f_{k}\end{pmatrix},\begin{pmatrix} f_{k} \\ \D f_{k}\end{pmatrix}\right\rangle \right\|,
\end{split}
\label{eq:est-one}
\end{equation}
proving that both \eqref{connectionexpansion} and  \eqref{simple} are well-defined. The estimate 
\eqref{eq:est-one} also provides half of the estimates needed to prove 
continuity of the map $g$. The other half is proving continuity of
$e\otimes f\mapsto \gamma(e)\ox\D f$.
So, again, consider a finite sum
$\sum_ke_k\ox f_k\in \E^\nabla\otimes_\B \Dom \D$.
We have the estimate
\[
\begin{split}\left\|\sum_{k} \gamma(e_{k})\otimes \D f_{k}\right\|^{2}
&\leq \left\|\sum|\gamma(e_{k})\rangle\langle \gamma(e_{k})| \right\|_{\kK(E)}\left\|\sum\langle \D f_{k},\D f_{k}\rangle \right\|\\ 
& \leq\|\sum_{k}\pi_{\nabla}(e_{k})\pi_{\nabla}(e_{k})^{*}\| \left\|\sum\left\langle \begin{pmatrix} f_{k} \\ \D f_{k}\end{pmatrix},\begin{pmatrix} f_{k} \\ \D f_{k}\end{pmatrix}\right\rangle \right\|,
\end{split}
\]
by using the fact that the $C^{*}$-module tensor product $E\otimes_BF$ is 
isometrically isomorphic to the Haagerup tensor product $E\hotimes_BF$ cf. \cite[Thm. 4.3]{BlecherC*}.

Combining the two norm estimates above and taking the infimum over all representations in the tensor product shows that the map $g$ satisfies
\[
\left\| g\Big(\sum_{k} e_{k}\otimes \begin{pmatrix} f_{k} \\ \D f_{k}\end{pmatrix}\Big)\right\|
\leq 2 \left\|\sum_{k} \pi_{\nabla}(e_{k})\otimes \begin{pmatrix} f_{k} \\ \D f_{k}\end{pmatrix}\right\|_{h},
\] 
and we are done.
\end{proof}
\begin{lemma}
\label{lem:domain} Let $\E_{\B}$ be a projective operator module with 
column finite frame $(x_i)$ and $R\in \mathscr{C}(\chi_n)$. Then\\
$1)$ the operator $vRv^{*}$ maps $\Dom \D_{\epsilon}$ into $\Dom\partial$;\\
$2)$ the operator $vRv^{*}$ maps $\Dom \partial^{*} $ into 
$\Dom \D_{\epsilon}$;\\
$3)$ if $\E_{\B}$ is complete, then $vRv^{*}$ maps $\Dom \partial^{*} $ into $\Dom \D_{\epsilon}\cap\Dom \partial\subset \Dom \partial$.
\end{lemma}
\begin{proof} It suffices to show that
the frame approximate 
unit $\chi_{n}$ of $(x_i)$ 
has the properties $1),2)$ and $3)$, 
for then any finite convex combination $R$ of $\chi_{n}$'s also has these properties. 
For $1)$, consider the adjointable operators
\[
\pi_{\D}^{p}(\chi_{k}):=
\begin{pmatrix} v\chi_{k}v^{*} & 0 \\ p[\D_{\epsilon},v\chi_{k}v^{*}] & v\chi_{k}v^{*}\end{pmatrix} : 
(\H_{B}\hotimes_{B} F)^{\oplus 2}  \rightarrow (vE\hotimes_{B}F)^{\oplus 2}.
\]
For $x=h\otimes f\in \H_\B\otimes_\B \Dom \D\subset \Dom\D_{\epsilon}$, we have that
\begin{equation}
\label{algtens} 
v\chi_{k}v^{*}(x)=\sum_{1\leq |i|\leq k} v x_{i}\otimes \langle x_{i},v^{*}(h)\rangle f
=\sum_{1\leq |i|\leq k} v x_{i}\otimes \langle vx_{i},h\rangle f,
\end{equation}

and since $v(x_{i}),h\in \H_{\B^{+}}$, we have $\langle vx_{i},h\rangle\in \B$ 
and thus $\langle vx_{i},h\rangle f\in\Dom \D$. Hence the finite sum \eqref{algtens} 
is an element of $v\E\otimes_{\B}\Dom \D\subset \Dom\partial$. 
By Lemma \ref{connv} we get that
\[
\pi_{\D}^{p}(\chi_{k})\begin{pmatrix} x \\ \D_{\epsilon} x \end{pmatrix} = 
\begin{pmatrix} v\chi_{k}v^{*}x \\ p(\D_{\epsilon}) v\chi_{k}v^{*} x\end{pmatrix}
=\begin{pmatrix} v\chi_{k}v^{*}x \\ \partial v\chi_{k}v^{*} x\end{pmatrix}.
\]
It follows from \eqref{stablesum} that $\H_{\B}\otimes_{\B} \Dom\D$ is a core for 
$\D_{\epsilon}$, for it contains the algebraic direct sum $\oplus_{i\in\Z}\Dom\D$.  
Thus the bounded operators $\pi_{\D}^{p}(\chi_{k})$ map a dense subspace of 
$\G(\D_{\epsilon})$ into $\G(\partial)$, and therefore they map all of 
$\G(\D_{\epsilon})$ into $\G(\partial)$. This proves $1)$.

For $2)$, consider the adjoint 
$\pi_{\D}^{p}(\chi_{k})^{*}$, which by $1)$ maps $\G(\partial)^{\perp}$ into $\G(\D_{\epsilon})^{\perp}$. The equalities 
$\G(\partial)^{\perp}=\flip\G(\partial^{*})$ 
and 
$\G(\D_{\epsilon})^{\perp}=\flip\G(\D_{\epsilon})$ allow us to compute, for $x\in\Dom \partial^{*}$
\begin{align*}
\pi_{\D}^{p}(\chi_{k})^{*}\begin{pmatrix} -\partial^{*}x \\ x \end{pmatrix} 
&= \begin{pmatrix} v\chi_{k}v^{*} & -[\D_{\epsilon},v\chi_{k}v^{*}]p \\ 0 & v\chi_{k}v^{*}\end{pmatrix}
\begin{pmatrix} -\partial^{*}x \\ x \end{pmatrix} \\
&=\begin{pmatrix} -v\chi_{k}^{*}\partial^{*}x-[\D_{\epsilon},v\chi_{k}v^{*}]x \\ v\chi_{k}v^{*}x \end{pmatrix}
\in \flip\G(\D_{\epsilon}).
\end{align*}
Hence  $v\chi_{k}v^{*}x\in \Dom \D_{\epsilon}$ whenever $x\in \Dom \partial^{*}$ which proves $2)$.

For $3)$ it suffices to show that $v\chi_{n}v^{*}$ maps $\Dom\partial^{*}$ into 
$\Dom \partial$ and then use $2)$. Let $x\in \Dom \partial^{*}$ and, since $\E_{\B}$ is complete, 
let $u_{k}\in\mathscr{C}(\chi_{n})$ be the approximate unit from 
Definition \ref{completemodule}. By $2)$ $v\chi_{n}v^{*}x\in\Dom\D_{\epsilon}$ 
and by $1)$  $vu_{k}v^{*}v\chi_{n}v^{*}x\in\Dom\partial$. We have 
$\lim_{k}vu_{n}v^{*}v\chi_{n}v^{*}x=v\chi_{n}v^{*}x$ in norm in 
$\H_{B^{+}}\hotimes_{B^{+}}F_{C}$. Now the operator $p[\D_{\epsilon}, vu_{k}v^{*}]p$ is defined on the
dense subspace $v\E\otimes_{\B^+}\Dom\D$, and bounded there. Hence it extends to a bounded operator on the 
whole module $\H_{\B^+}\hotimes_{\B^+} F_C$. The relation  
\begin{equation}
\label{eq:colly-wobble}
(\partial vu_{k}v^{*}-vu_{k}v^{*}\D_{\epsilon})p=p[\D_{\epsilon}, vu_{k}v^{*}]p,
\end{equation}
which is valid on the subspace $\Dom\partial\cap\Dom\D_\epsilon\cap p\H_{\B^+}\hotimes_{\B^+}F_C$, 
along with the boundedness of $p[\D_{\epsilon}, vu_{k}v^{*}]p$, imply that the left hand side of 
Equation \eqref{eq:colly-wobble} is bounded. Combining all these facts with the strict convergence
$p[\D_{\epsilon}, vu_{k}v^{*}]p\to 0$,
we find that
\begin{align*}
\lim_{k}\partial vu_{k}v^{*}v\chi_{n}v^{*}x
&=\lim_{k}vu_{k}v^{*}\D_{\epsilon} v\chi_{n} v^{*}x
+\partial vu_{k}v^{*}v\chi_{n}v^{*}-vu_{k}v^{*}\D_{\epsilon}v\chi_{n}v^{*}x \\
&= \lim_{k}vu_{k}v^{*}p\D_{\epsilon} p v\chi_{n}v^{*}x + p[\D_{\epsilon}, vu_{k}v^{*}]pv\chi_{n}v^{*}x
\end{align*}
which since $vu_kv^*\to p$ strictly, shows that the sequence converges. 
Since $\partial$ is closed, $v\chi_{n}v^{*}x\in\Dom\partial$.
\end{proof}
The paper \cite{KaLe2} introduces a local-global principle for 
regular operators on $C^{*}$-modules. The main technical 
tool developed is the following. Let $E_{B}$ be a $C^{*}$-module 
and $\sigma:B\rightarrow \C$ a state and $\H_{\sigma}=L^{2}(B,\sigma)$ 
the associated GNS representation. The \emph{localisation} $E^{\sigma}$ 
is the Hilbert space completion of $E_{B}$ in the inner product 
$\langle e,f\rangle_{\sigma}:=\sigma(\langle e,f\rangle)$, 
and there is a dense inclusion $\iota_{\sigma}:E_{B}\rightarrow E^{\sigma}$ 
and a $*$-representation $\pi_{\sigma}:\End_{B}^{*}(E_B)\to \bB(E^{\sigma})$. Equivalently, 
$E^\sigma=E\ox_BL^2(B,\sigma)$, where $L^2(B,\sigma)$ denotes the GNS representation
space of $B$ defined by the state $\sigma$. A closed, densely 
defined symmetric operator $T$ on $E$ induces a closed densely 
defined symmetric operator $T^{\sigma}$ in $E^{\sigma}$, by 
defining it on the dense subspace $\iota_{\sigma}(\Dom T)\subset E^{\sigma}$ and taking the closure. 
It then holds that $\iota_{\sigma}(\Dom T^{*})\subset \Dom (T^{\sigma})^{*}$, 
cf. \cite[Lemma 2.5]{KaLe2}.
\begin{thm}[Theorem 4.2, \cite{KaLe2}] 
Let $T$ be a closed densely defined symmetric operator in the $C^{*}$-module $E_B$. 
Then $T$ is self-adjoint and regular if and only if all localisations $T^{\sigma}$ are self-adjoint.
\end{thm}
For an unbounded Kasparov module $(\B,F_{C},\D)$ and a state 
$\sigma:C\rightarrow \C$ we obtain a contractive map 
$\B\rightarrow \Lip(\D^{\sigma})$. This follows because by 
definition $\iota_{\sigma}(\Dom\D)$ is a core for $\D^{\sigma}$ 
and for all $b\in\B$ and $f\in\Dom \D$ we have 
$\pi_{\sigma}(b)\iota_{\sigma}(f)=\iota_{\sigma}(bf)\in \iota_{\sigma}(\Dom \D)$. 
Thus $\pi_{\sigma}(b)$ preserves the core $\iota_{\sigma}(\D)$ for $\D^{\sigma}$. 
The commutator satisfies
\begin{equation}
\begin{split}
\| [\D^{\sigma},\pi_{\sigma}(b)]\iota_{\sigma}(f)\|^{2}=\|\iota_{\sigma}([\D,b]f)\|^{2}
&=\sigma(\langle [\D,b]f,[\D,b]f\rangle)
\leq \|[\D,b]\|^{2}\sigma(\langle f,f\rangle)\\
&=\|[\D,b]\|^{2}\|\iota_{\sigma}(f)\|^{2},
\end{split}
\label{eq:cb-piD}
\end{equation}
and is thus bounded on this core. Thus $[\D^{\sigma},\pi_{\sigma}(b)]=\pi_{\sigma}([\D,b])$ 
and we can write $\pi_{\D^{\sigma}}(b)=\pi_{\sigma}(\pi_{\D}(b))$ and 
hence the map $\pi_{\D}(b)\mapsto \pi_{\D^{\sigma}}(b)$ is 
completely bounded. We let $\B^{\sigma}$ be the completion of 
$\B$ is the norm induced by $\pi_{\D^{\sigma}}$, and define the 
localised module $\E_{\B^\sigma}$ over $\B^{\sigma}$ via the map $\H_{\B^{+}}\to\H_{\B^{\sigma +}}$.
\begin{lemma}
\label{localdomain} 
Let $\E_\B$ be a complete projective operator module for $(\B, F_{C},\D)$ 
with column finite frame $(x_{i})$ and frame approximate unit $\chi_{n}$. 
Then for all states $\sigma:C\rightarrow C$, the localised module 
$\E_{\B^{\sigma}}$ is a complete projective module for 
$(\B^{\sigma},F^{\sigma},\D^{\sigma})$. Moreover, under the 
identification $E\hotimes_{B} F^{\sigma}\cong (E\hotimes_{B} F)^{\sigma}$ 
we have $1\otimes_{\nabla}\D^{\sigma}=(1\otimes_{\nabla}\D)^{\sigma}$ 
as unbounded operators. Therefore\\
1) for each $n$, the operator 
$\pi_{\sigma}(\chi_{n})$ maps $\Dom (\partial^{\sigma})^{*}$ 
into $\Dom \partial^{\sigma}$;\\
2) there is an approximate unit $(u_n)\subset\mathscr{C}(\chi_n)$
such that $p[\D^{\sigma}_{\epsilon},vu_{n}v^{*}]p$ 
converges to $0$ $*$-strongly on $\H_{B^{+}}\hotimes_{B^{+}}F^{\sigma}$.
\end{lemma} 
\begin{proof} 
To check that $\E_{\B^{\sigma}}$ is a complete projective module, it suffices 
to show that the defining frame $(x_i)$ of $\E_\B$ is column finite for $\B^\sigma$,
and that there exists an approximate unit 
$(u_n)\in\mathscr{C}(\chi_{n})$ such that  
$p[D_{\epsilon}^{\sigma},vu_{n}v^{*}]p\rightarrow 0$ $*$-strongly on 
$\H_{B^+}\hotimes_{B^{+}}F^{\sigma}$. Column finiteness follows from 
complete boundedness of the map 
$\pi_{\D}(b)\mapsto \pi_{\D^{\sigma}}(b)$, proved after \eqref{eq:cb-piD} above.
This is because complete boundedness shows that for all $e\in \E_\B$
\[
\|\pi_{\D^{\sigma}}(\langle x_{i},e\rangle)_{i\in\Z}\|\leq \||\pi_{\D}(\langle x_{i},e\rangle)_{i\in\Z}\|,
\]
and so in particular for all the vectors $x_{j}$. Definition \ref{completemodule} 
gives an approximate unit $(u_{n})$ in the convex set 
$\mathscr{C}(\chi_{n})\subset \Fin_\B(\E)$ 
for which the sequence 
$p[\D^{\sigma}_{\epsilon},vu_{n}v^{*}]p$ 
converges to $0$ strictly 
on $\H_{B^{+}}\hotimes_{B^{+}}F$ 
and is therefore uniformly bounded in $n$. Thus the localised sequence 
$p[\D_{\epsilon}^{\sigma},\pi_{\sigma}(u_{n})]p=\pi_{\sigma}(p[\D_{\epsilon},u_{n}]p)$ is 
bounded as well and converges strongly to $0$ on the dense subspace
$\H_{B^{+}}\hotimes \iota_{\sigma}(F)$ and thus on all of 
$\H_{B^{+}}\hotimes F^{\sigma}$. Hence $\E_{\B^{\sigma}}$ is a complete 
projective module for $(\B^{\sigma},F^{\sigma},\D^{\sigma})$, which in particular proves 2).

The operator $1\otimes_{\nabla}\D^{\sigma}$ is defined 
on its core $\E\otimes_{\B^{+}}\Dom \D^{\sigma}$ 
while $(1\otimes_{\nabla}\D)^{\sigma}$ is defined 
on $\iota_{\sigma}(\Dom 1\otimes_{\nabla}D)$. We claim that the subspace 
\[
X:=\iota_{\sigma}(\E\otimes_{\B}\Dom \D)
=\E\otimes_{\B}\iota_{\sigma}(\Dom \D)\subset\iota_{\sigma}(\Dom 1\otimes_{\nabla}D) ,
\]
is a common core for $(1\otimes_{\nabla}\D)^{\sigma}$ 
and $1\otimes_{\nabla}\D^{\sigma}$. Since $\E\otimes_{\B}\Dom \D$ 
is a core for $1\otimes_{\nabla}\D$, its image under $\iota_{\sigma}$ 
is a core for $(1\otimes_{\nabla}\D)^{\sigma}$. To see that it is also a core 
for $1\otimes_{\nabla}\D^{\sigma}$, we use the definition
\[
1\otimes_{\nabla}\D^{\sigma}(e\otimes f)=\gamma(e)\otimes \D f + \nabla_\D(e)f
=\gamma(e)\otimes \D^{\sigma} f + \sum \gamma(x_{i})\otimes [\D^{\sigma},\langle x_{i},e\rangle ]f,
\]
and take a sequence $f_{k}\in\Dom\D$ converging to $f\in \Dom \D^{\sigma}$ 
in the graph norm. The term $\gamma(e)\otimes \D^{\sigma} f_{k}$ will 
then converge to $\gamma(e)\otimes \D^{\sigma} f$. The other term 
can be estimated using the Haagerup norm
\[
\begin{split} 
\big\|\sum \gamma(x_{i})\otimes [\D,\langle x_{i},e\rangle](f_{k}-f_{\ell})\big\|^{2}_{h} 
&\leq \big\|\sum |x_{i}\rangle\langle x_{i} | \big\|_{\kK(E)} 
\left\| ([\D,\langle x_{i},e\rangle](f_{k}-f_{\ell}))_{i\in\Z}\right\|^{2}\\
&\leq  \left\|([\D,\langle x_{i},e\rangle] )_{i\in\Z}\right\|^{2}\|f_{k}-f_{\ell}\|^{2},
\end{split}
\]
and the norm of the column $([\D,\langle x_{i},e\rangle] )_{i\in\Z}$ 
is finite because $e\in \E_\B$.  Therefore we can approximate any 
$y\in \E\otimes_{\B}\Dom \D^{\sigma}$ by elements of 
$X$ in the graph norm of $1\otimes_{\nabla}\D^{\sigma}$. 
Thus the closure of $1\otimes_{\nabla}\D^{\sigma}$ on $X$ 
contains $\E\otimes_\B \Dom \D^{\sigma}$ which is the defining 
core for $1\otimes_{\nabla}\D^{\sigma}$. Therefore $X$ is a 
common core and since the operators 
$(1\otimes_{\nabla}\D)^{\sigma}$ and $1\otimes_{\nabla}\D^{\sigma}$ coincide on 
$X$, it follows that $(1\otimes_{\nabla}\D)^{\sigma}=1\otimes_{\nabla}\D^{\sigma}$.

Statement 1) now follows by applying Lemma \ref{lem:domain} 
to the frame $(x_i)$ of the complete projective module $\E_{\B^{\sigma}}$.
\end{proof}
We now come to the main application of complete projective modules: self-adjointness of the connection
operator $1\otimes_\nabla\D$. A further application of the domain mapping properties
of approximate units then allows us to show that $\kK(\E^\nabla)$ is a differentiable algebra. 
\begin{thm} 
\label{selfreg}
Let $\E_\B$ be a complete projective module for $(\B, F_{C},\D)$. 
Then the operator $1\otimes_{\nabla}\D$ is self-adjoint and regular.
\end{thm}
\begin{proof} 
We must show that for all 
states $\sigma:C\rightarrow \C$ the operator $\partial^{\sigma}$ on the Hilbert space 
$(E\hotimes_{B} F)^{\sigma}\cong E\hotimes_{B} F^{\sigma}$ is self-adjoint. 
Let $(u_{n})\subset \mathscr{C}(\chi_{n})$ be an approximate unit 
as in Definition \ref{completemodule}.
By Lemma \ref{localdomain},  $\pi_{\sigma}(u_{n})$ 
maps $\Dom (\partial^{\sigma})^{*}$ into 
$\Dom \partial^{\sigma}$ and $p[\D^{\sigma}_{\epsilon},vu_{n}v^{*}]p$ 
converges to $0$ $*$-strongly on $\H_{B^{+}}\hotimes_{B^{+}}F^{\sigma}$ . For $x$ in 
the dense subspace $\H_{\B^{+}}\otimes \Dom \D^{\sigma}$, 
which is a core for $\partial^{\sigma}$, using Equation \eqref{connectionexpansion} we have, 
\[
\begin{split}
[\partial^{\sigma},vu_{k}v^{*}]x&=\partial^{\sigma} vu_{k}v^{*}x-vu_{k}v^{*}\partial^{\sigma} x 
=p\D_{\epsilon}^{\sigma}p vu_{k}v^{*}x- vu_{k}v^{*}p\D^{\sigma}_{\epsilon}px
=p[\D^{\sigma}_{\epsilon}, vu_{k}v^*]px \rightarrow 0,
\end{split}
\]
in norm, and by uniform boundedness of $p[D_{\epsilon},vu_{k}v^{*}]p$, the convergence 
holds for all $x\in \H_{\B^{+}}\hotimes F^{\sigma}$. Therefore if $y\in \Dom (\partial^{\sigma})^{*}$, then 
$u_{k}y\in \Dom \partial^{\sigma}$ and $u_{k}y\rightarrow y$, so we can compute
\[ 
(\partial^{\sigma})^{*}y=\lim u_{k}(\partial^{\sigma})^{*}y
=\lim \partial^{\sigma} u_{k}y- [(\partial^{\sigma})^{*},u_{k}] y 
=\lim u_{k}\partial^{\sigma}y+[\partial^{\sigma},u_{k}]^{*}y= \partial^{\sigma}y,
\]
which shows that $y\in \Dom \partial^{\sigma}$ and $\partial^{\sigma}$ is self-adjoint. The local-global
principle of \cite{KaLe2} now says that $\partial$ is self-adjoint and regular.
\end{proof}
We now describe the algebra of adjointable operators on a
complete projective module. Let the isometry of $C^{*}$-modules
$
v:E\rightarrow \H_{B^{+}}
$
be such that it induces a column finite frame $(x_i)_{i\in\Z}$, which 
in turn determines a 
complete projective  submodule $\E_\B\subset\E^{\nabla}_\B\subset E_B$. The defining representation
\[
\pi_{\nabla}(K)=\begin{pmatrix} vKv^{*} & 0 \\ p[\D_{\epsilon},vKv^{*}]p & vKv^{*}\end{pmatrix},
\]
preserves the submodule 
$
(p\H_{B^{+}}\hotimes_{B^{+}}F)\oplus (p\H_{B^{+}}\hotimes_{B^{+}}F),
$
and annihilates the orthogonal complement. 
\begin{defn} 
The \emph{algebra of adjointable operators} on the projective module 
$\E^{\nabla}_\B$ is the idealiser of $\pi_{\nabla}(\kK(\E^{\nabla}))$ in 
$\End^{*}_{C}((p\H_{B^{+}}\hotimes_{B^{+}}F)^{\oplus 2})$. 
It is denoted $\End^{*}_{\B}(\E^{\nabla})$.
\end{defn}
\begin{prop}\label{repend} 
If $\E_\B$ is a complete projective module then $\End^{*}_{\B}(\E^{\nabla})$ 
is an operator $*$-algebra, isometrically isomorphic to $\M(\kK(\E^{\nabla}))$ 
and coinciding with a closed subalgebra of $\Lip(1\otimes_{\nabla}\D)$.
Hence $\kK(\E^\nabla)$ is a differentiable algebra.
\end{prop}
\begin{proof} 
This essentially follows from Propositions \ref{cor:bdd-strong} and \ref{lipmult}. 
Since $1\otimes_{\nabla}\D$ is self-adjoint and 
regular the commutators $[1\otimes_{\nabla}\D,vKv^{*}]$ coincide with 
the operators $p[\D_{\epsilon},vKv^{*}]p$ when $K$ is finite rank. Suppose now that $T\in \End^{*}_{\B}(\E^{\nabla})$, so there is a sequence of operators $T_{n}$ such that for all $K\in\Fin_{\B}(\E)$ we have $T_{n}K\in\Fin_{\B}(\E)$ and both $T_{n}K$ and  $p[\D_{\epsilon},vT_{n}Kv^{*}]p$ are convergent. Then since
\[p[\D_{\epsilon},vT_{n}Kv^{*}]p=[\partial,vT_{n}Kv^{*}],\]
it follows that 
\[ 
\partial (vT_{n}Kv^{*}x)=[\partial,vT_{n}Kv^{*}]x+vT_{n}Kv^{*}\partial (x),
\]  
is convergent for all $x\in \Dom \partial$. 
Thus $TK$ preserves $\Dom\partial$ for all 
$K\in\Fin_{\B}(\E)$ and $v\Fin_{B}(\E)v^{*}\cdot \Dom \partial$ 
is dense in $p\Dom \partial$ in the graph norm by definition of 
$1\otimes_{\nabla}\D$. Thus $T$ preserves a core, and on this core the commutator
\[ 
[\partial,vTv^{*}]vKv^{*}x=[\partial, vTKv^{*}]x-v\gamma(T)v^{*}[\partial, vKv^{*}]x,
\]
is a bounded operator. Thus $vTv^{*}\in\Lip(\partial)$, which is equivalent to 
$T\in \Lip(1\otimes_{\nabla} \D)$ as desired. The argument now 
proceeds as in Proposition \ref{lipmult}.
\end{proof}
\section{Completeness and the Kasparov product}
The constructive approach to the Kasparov product has appeared in several slightly
different versions in recent years, \cite{BMS, KaLe, Mes}. The variations have come from
the assumptions imposed on the correspondences $(\A,\mathcal{E}_\B,S,\nabla)$ which refine
the notion of unbounded Kasparov module. The most recent refinement in \cite{BMS} was
the inclusion of a class of unbounded projections into the theory, 
required to deal with examples arising from
the Hopf fibration. Unbounded projections also 
appear in the construction of products for 
Cuntz-Krieger algebras \cite{GM}, the natural 
Kasparov module for $SU_{q}(2)$ \cite{KRS,Senior} 
and the differential approach to the stabilisation theorem \cite{Kaadabsorption}.
In the previous section of the present paper, the notion of complete 
projective module enlarges the class of unbounded projections we can work with.
\subsection{Constructing the unbounded Kasparov product}
In this
section we will show that the lifting constructions of 
\cite{BJ, Kucerovsky2} can be refined in such a way that we 
can lift a pair of cycles $(A,E_{B},F_{1})$ and $(B, F_{C}, F_{2})$ 
to an unbounded Kasparov module $(\B, F_{C}, T)$ and a correspondence 
$(\A, E_{\B},S,\nabla)$ for $(\B, F_{C}, T)$. This has the advantage 
that their Kasparov product as constructed through Theorem 
\ref{constructprod} is then well-defined. Since we only 
have to lift two classes, we provide a significant 
improvement over the results of \cite{Kucerovsky2}, where it was shown 
that any three $KK$-classes, with one the product of the other two, can
be lifted to unbounded classes in a way compatible with 
Kucerovsky's conditions for representing products  \cite{Kucerovsky1}.

Our first task is to assemble the results in the literature and 
blend them with the present work in order to give sufficient 
conditions under which the unbounded Kasparov product can 
be constructed. These conditions will allow us to show in Section \ref{subsec:lift} that any 
Kasparov product can be realised as the composition 
of a correspondence and an unbounded Kasparov module.
\begin{defn} 
\label{correspondence}
Given $(\B,F_C,T)$ an unbounded Kasparov module with bounded approximate unit for $\B$, 
an $\A$-$\B$ \emph{correspondence} for $(\B,F_C,T)$ is a quadruple 
$(\A,\mathcal{E}_\B,S,\nabla)$ such that:\\
1) $\mathcal{E}_\B$ is a complete projective operator module over the 
algebra $\mathcal{B}$;\\
2) the algebra $\A$ is dense in $A$ and $\mathcal{A}\subset\End^{*}_{\mathcal{B}}(\mathcal{E}^{\nabla})\cap \Lip(S)$;\\
3) a self-adjoint regular operator $S:\Dom S\rightarrow E$ such that
$(S\pm i)^{-1}\in \End^{*}_{\B}( \mathcal{E}^{\nabla})$ and $a(S\pm i)^{-1}\in \kK_{\mathcal{B}}(\mathcal{E}^{\nabla})$ for $a\in\mathcal{A}$;\\
4) a connection $\nabla:\mathcal{E}^{\nabla}\rightarrow E\hotimes_{B}\Omega^{1}_{\D}$ 
such that $\nabla((S\pm i)^{-1}\E^{\nabla})\subset \Dom S\otimes 1$ and 
$[\nabla,S](S\pm i)^{-1}:\mathcal{E}^{\nabla}\rightarrow E\hotimes_{\mathcal{B}}\Omega^{1}_{T}$ 
is completely bounded.

The correspondence is called \emph{strongly complete} if there is an 
approximate unit $(u_{n})$ for $A$ such that both $[S,u_{n}]\rightarrow 0$ 
and $[1\otimes_{\nabla}T,u_{n}\otimes {\rm Id}_F]\to 0$ in $C^{*}$-norm. 
\end{defn} 
In condition 4), we regard $\nabla$ as an odd operator so the commutator is the
graded commutator $[\nabla,S]=\nabla S - \gamma(S)\nabla$. 
One of the key points in the construction of the Kasparov product 
is the self-adjointness of the product operator, and this is deduced from the general framework 
of weakly anti-commuting operators described in the appendix.
\begin{lemma}
\label{anticomm} 
Let $(\A,\E_{\B},S,\nabla)$ be an $\A$-$\B$ 
correspondence for $(\B, F_{C}, T)$. The self-adjoint regular operators 
$s:=S\otimes 1$ and $t:=1\otimes_{\nabla}T$ weakly anticommute in $E\hotimes_{B}F$.
\end{lemma}
\begin{proof} We will show that the conditions of Definition \ref{correspondence} 
imply those of Definition  \ref{weaklyanti}.
By Lemma \ref{connv} the map
\[
g:\E^{\nabla}\hotimes_{\B}\G(T)\to \G(1\otimes_{\nabla} T),\quad 
e\otimes \begin{pmatrix} f \\ Tf \end{pmatrix}
\mapsto \begin{pmatrix} e\otimes f \\ 1\otimes_{\nabla}T (e\otimes f)\end{pmatrix},
\]
has dense range. This means that the submodule $X:=\E^{\nabla}\otimes_{\B}\Dom T$,
is a core for $1\otimes_{\nabla}T$. Since $(S\pm i)^{-1}:\E^{\nabla}\rightarrow \E^{\nabla}$, 
the resolvents $(s\pm i)^{-1}$ preserve the core $X$, so $1)$ of Definition \ref{weaklyanti} is satisfied. By condition $4)$ of Definition \ref{correspondence} it follows that $t(s\pm i)^{-1}X\subset \Dom s$, so $2)$ of Definition \ref{weaklyanti} is satisfied as well. On the core $X$ 
the graded commutator can be computed as
\[
\begin{split} 
[t,(s\pm i)^{-1}]=t(s\pm i)^{-1}+(s\mp i)^{-1}t
&= (s\mp i)^{-1}(t(s\pm i) +(s\mp i)t)(s\pm i)^{-1} \\
&= (s\mp i)^{-1}[s,t](s\pm i)^{-1} ,
\end{split}
\]
and this is a bounded operator because 
$$
[s,t](e\otimes f)=(S\ox 1)(\gamma(e)\ox Tf+\nabla(e)f)+\gamma(Se)\ox Tf+\nabla(Se)\ox f=[\nabla, S](e)f.
$$ 
Thus, $(s\pm i)^{-1}$ preserve the domain of $t$ and $[s,t](s\pm i)^{-1}$ 
are bounded there, proving condition $3)$ of Definition \ref{weaklyanti}.
\end{proof}
The following theorem encompasses and generalises the 
constructions of the unbounded Kasparov product that have appeared in \cite{BMS, KaLe, Mes}.
\begin{thm}
\label{constructprod}
Let $(\B,F_C,T)$ be an unbounded Kasparov module and let
$(\A,\mathcal{E}_\B,S,\nabla)$ an $A$-$B$ correspondence for $(\B,F_C,T)$. 
Then $(\A,(E\hotimes_{B}F)_C, S\otimes 1 + 1\otimes_{\nabla}T)$ 
is an unbounded Kasparov module representing the Kasparov product of $(\A,E_B,S)$ and 
$(\B,F_C,T)$.
\end{thm}
\begin{proof} 
The operator $t:=1\otimes_{\nabla}T$ is self-adjoint and 
regular in $E\hotimes_{B}F$ by Theorem \ref{selfreg}, as is the operator $s:=S\otimes 1$. 
Now $s$ and $t$ weakly anticommute  by Lemma \ref{anticomm}, and hence their sum 
is self-adjoint and regular in $E\hotimes_{B}F$ by Theorem \ref{thm:bob}, 
with locally compact resolvent by an argument similar to \cite[Theorem 6.7]{KaLe}. 
Thus $(\A,(E\hotimes_{B}F)_C, S\otimes 1 + 1\otimes_{\nabla}T)$ is an unbounded Kasparov module.
One then shows, exactly as in \cite[Theorem 7.2]{KaLe} 
and \cite[Theorem 6.3.4]{Mes}, that the hypotheses of
Kucerovsky's theorem, \cite[Theorem 13]{Kucerovsky1}, are satisfied. 
Hence this cycle 
represents the Kasparov product.
\end{proof}
Now we embark on a series of lifting results of increasing sophistication, whose ultimate aim is
to show that any pair of composable $KK$-classes can be represented by
unbounded Kasparov modules satisfying the hypotheses of Theorem \ref{constructprod}.
Recall that for a bounded $(A,B)$-Kasparov module 
$(A,E_{B}, F)$ the associated ideal of $A$ -\emph{locally compact operators} is 
$$
J_A(E_B):=\{T\in \End^{*}_{B}(E_{B}): aT,\,Ta\in \kK(E_{B}) \ \textnormal{for all } a\in A\}.
$$
The operator $F$ is in the idealiser of $J_A(E_B)$, for if $T\in J_A(E_B)$ then 
$FTa\in\kK(E_{B})$ since $Ta\in \kK(E_{B})$ and $aFT=FaT-[F,a]T\in \kK(E_{B})$ as well. 
Moreover, $1-F^{2}$ and hence $(1-F^{2})^{\frac{1}{2}}$ are both 
elements of $J$. The $C^{*}$-algebra $J_A(E_B)$ is not $\sigma$-unital in general. 
The following counterexample to $\sigma$-unitality arose 
from discussions of the first author with J. Kaad: 
Let $I$ be an ideal in a unital $C^{*}$-algebra $B$. 
Take $E:=C_{0}(\N, I)$ viewed as a $C^{*}$-module 
over $C_{0}(\N, I)$ and let $A:=C_{0}(\N, B)$. 
Then $J_{A}(E)=C_{b}(\N, I)$ which is not $\sigma$-unital.

In \cite{BJ} it was shown that any bounded Kasparov module 
$(A,E_{B},F)$ can be represented by an unbounded Kasparov 
module $(A,E_{B},\D)$. The operator $\D$ is obtained from 
$F$ by constructing a suitable strictly positive $\ell$ element 
in the ideal $J_A(E_B)$ and then setting $\D:=F\ell^{-1}$. We denote 
by $J_{F}$ the $C^{*}$-algebra generated by $\kK(E_B)$ and 
$1-F^{2}$, which is a separable subalgebra of $J_A(E_B)$. The 
element $\ell$ is constructed from an approximate unit for 
$J_{F}$ with certain quasicentrality properties. 
\begin{defn}
\label{admlift} 
Let $(A,E_{B}, F)$ be a Kasparov 
module with $[AE_{B}]=E_{B}$, $F=F^{*}$ and $1-F^{2}\geq 0$. 
A strictly positive element $\ell\in J_{F}$ is \emph{admissible} if:\\
1) there exists $C>0$ with $i[F,\ell]\leq C\ell^{2}$;\\
2) $(1-F^{2})^{\frac{1}{2}}\ell^{-1}$ is bounded on the range of $\ell$ and has norm $c<1$;\\
3) there is a total subset $\{a_{i}\}\subset A$ for which $a_i:\ell E\to \ell E$, the commutators 
$[\ell^{-1},a_{i}]$ 
and $[F,a_{i}]\ell^{-1}$ are bounded on the range of $\ell$, and 
so extend to operators in $\End^{*}_{B}(E)$.
\end{defn}
\begin{thm} 
\label{thm:firstlift}
If $\ell\in J_{F}$ is admissible then $\D:=\frac{1}{2}(F\ell^{-1}+\ell^{-1}F)$ 
is a self-adjoint regular operator with the property that $(\A, E_{B},\D)$ 
is an unbounded Kasparov module defining the same class as $(A,E_{B}, F)$.
\end{thm}
\begin{proof} 
It is shown in \cite[Lemmas 1.4, 2.2]{Kucerovsky2} that $F\ell^{-1}$ defines 
an almost self-adjoint (so $[F,\ell^{-1}]$ is bounded) 
regular operator on $E_{B}$ with resolvent in $J_F$ 
and that it has bounded commutators with all the $a_{i}$. Thus 
$(\A,E_B,\D)$ is an unbounded Kasparov module and it suffices to show
the equivalence of the Kasparov modules defined by $F$ and  
$\tilde{\D}(1+\tilde{\D}^*\tilde{\D})^{-1/2}$, where $\tilde{\D}=F\ell^{-1}$.
By \cite[Proposition 17.2.7]{Blackadar}, it suffices to show that 
$$
a\big(F\tilde{\D}(1+\tilde{\D}^*\tilde{\D})^{-1/2}+(\tilde{\D}(1+\tilde{\D}^*\tilde{\D})^{-1/2}F\big)a^*
$$
is positive modulo compacts for all $a\in A$. Simplifying yields
\begin{align*}
&F\tilde{\D}(1+\tilde{\D}^*\tilde{\D})^{-1/2}+\tilde{\D}(1+\tilde{\D}^*\tilde{\D})^{-1/2}F\\&=
F^2\ell^{-1}(1+(F\ell^{-1})^*(F\ell^{-1}))^{-1/2}+F\ell^{-1}(1+(F\ell^{-1})^*(F\ell^{-1}))^{-1/2}F\\
&=F[F,\ell^{-1}](1+(F\ell^{-1})^*(F\ell^{-1}))^{-1/2}+F\ell^{-1}[F,(1+(F\ell^{-1})^*(F\ell^{-1}))^{-1/2}]\\
&\qquad\quad+2F\ell^{-1/2}[\ell^{-1/2},(1+(F\ell^{-1})^*F\ell^{-1})^{-1/2}]F+2F\ell^{-1/2}(1+(F\ell^{-1})^*F\ell^{-1})^{-1/2}\ell^{-1/2}F
\end{align*}
The first term on the right hand side of the last equality is compact when multiplied by 
any $a\in A$ on the right and the last term is positive, so we are left with the second and third terms.
Now we compute the commutator in the second term 
using the integral formula for fractional powers, \cite{CGRS2}. So
\begin{align*}
&\ell^{-1}[F,(1+(F\ell^{-1})^*(F\ell^{-1}))^{-1/2}]=
\frac{1}{\pi}\int_0^\infty\lambda^{-1/2}\ell^{-1}[F,(\lambda+1+(F\ell^{-1})^*(F\ell^{-1}))^{-1}]\,d\lambda\\
&=-\frac{1}{\pi}\int_0^\infty\lambda^{-1/2}\ell^{-1}(\lambda+1+\ell^{-1}F^2\ell^{-1})^{-1}
[F,(F\ell^{-1})^*(F\ell^{-1})](\lambda+1+\ell^{-1}F^2\ell^{-1})^{-1}d\lambda\\
&=-\frac{1}{\pi}\int_0^\infty\lambda^{-1/2}\ell^{-1}(\lambda+1+\ell^{-1}F^2\ell^{-1})^{-1}
[F,\ell^{-1}]F^2\ell^{-1}(\lambda+1+\ell^{-1}F^2\ell^{-1})^{-1}d\lambda\\
&-\frac{1}{\pi}\int_0^\infty\lambda^{-1/2}\ell^{-1}(\lambda+1+\ell^{-1}F^2\ell^{-1})^{-1}
\ell^{-1}F^2[F,\ell^{-1}](\lambda+1+\ell^{-1}F^2\ell^{-1})^{-1}d\lambda
\end{align*}
We observe that $\ell^{-1}(1-F^2)^{1/2}$ is bounded 
and of norm $c<1$ by Definition \ref{admlift}. 
It follows that $\ell^{-1}F^2\ell^{-1}=\ell^{-2}-\ell^{-1}(1-F^2)\ell^{-1}\geq \ell^{-2}-c^21$. The 
functional calculus then yields the estimates \cite[Appendix A]{CP1}
(for the norm of endomorphisms on $E_B$)
$$
\Vert\ell^{-1}(1+\lambda+\ell^{-2}-c^2)^{-1}\Vert\leq \frac{1}{2\sqrt{1+\lambda-c^2}},\quad
\Vert(1+\lambda+\ell^{-2}-c^2)^{-1}\Vert\leq \frac{1}{1+\lambda-c^2}.
$$
Thus the integral converges in norm. Since multiplying the integrand on right and left by
an element of $A$ yields a compact endomorphism, the same is true of the integral. For the third term
the integral formula yields
\begin{align*}
&\ell^{-1/2}[\ell^{-1/2},(1+\ell^{-1}F^2\ell^{-1})^{-1/2}]\\
&=-\ell^{-1/2}\frac{1}{\pi}\int_0^\infty\lambda^{-1/2}(1+\lambda+\ell^{-1}F^2\ell^{-1})^{-1}
\ell^{-1}[\ell^{-1/2},F^2]\ell^{-1}(1+\ell^{-1}F^2\ell^{-1}+\lambda)^{-1}\,d\lambda.
\end{align*}
In order to obtain the norm convergence of this integral, we write 
$$
\ell^{-1}[\ell^{-1/2},F^2]\ell^{-1}=\ell^{-1/2}(\ell^{-1}(F^2-1)\ell^{-1})-(\ell^{-1}(F^2-1)\ell^{-1})\ell^{-1/2}
$$
and 
$$
\ell^{-1/2}(1+\lambda+\ell^{-1}F^2\ell^{-1})^{-1}
=\ell^{-1/2}(1+\lambda+\ell^{-1}F^2\ell^{-1})^{-1}\ell^{-1/2}\ell^{1/2}.
$$
Since $\ell^{-1/2}(1+\lambda+\ell^{-1}F^2\ell^{-1})^{-1}\ell^{-1/2}\leq \ell^{-1}(1+\lambda+\ell^{-2}-c^2)^{-1}$,
the same norm estimates we used for the second term give us
\begin{align*}
&\Vert\ell^{-1/2}[\ell^{-1/2},(1+\ell^{-1}F^2\ell^{-})^{-1/2}]\Vert\\
&\leq \frac{c^2}{2\pi}\int_0^\infty\lambda^{-1/2}(1+\lambda-c^2)^{-3/2}d\lambda
+ \frac{c^2 \Vert \ell^{1/2}\Vert}{4\pi}\int_0^\infty\lambda^{-1/2}(1+\lambda-c^2)^{-1}d\lambda<\infty
\end{align*}
and so the integral converges in norm. As 
multiplying the integrand on both sides by an element of
$A$ yields a compact endomorphism, the same is 
true of the operator defined by the integral. This completes the proof.
\end{proof}
\subsection{Quasicentral approximate units}
The construction of an admissible multiplier 
$\ell$ needed in Theorem \ref{thm:firstlift} uses quasicentral approximate units as in \cite{BJ}. 
In order to lift Kasparov modules to correspondences, this notion needs to be refined.
The existence of quasicentral approximate units in $C^{*}$-algebras 
has been crucial for the development of $KK$-theory, notably in 
Higson's proof of the Kasparov technical theorem.

{\bf In this section $\B$ will always denote a unital operator algebra, and $\mathcal{J}\subset\B$
a closed ideal with bounded approximate unit.} 

For such $\mathcal{J}$ and $\B$, we wish to prove the existence of quasicentral approximate 
units.
We will do this by using the argument of 
Akemann and Pedersen \cite{AP}. This method was employed in \cite[Theorem 3.1]{Arias} to 
construct quasicentral approximate units for closed ideals in operator algebras with \emph{contractive} 
approximate units. 
By virtue of Proposition \ref{thm: idempotent},  the technique 
works for operator algebras with bounded approximate unit. 

In Theorem \ref{Ped} below, we prove a strong form of 
quasicentrality, unknown even in the case of $C^{*}$-algebras. 
Namely, we will view the ideal $\J\subset \B$ as sitting inside 
$\End^{*}_{\B}(\H_{\B})$ as `scalar' matrices $\J\cdot{\rm Id}_{\H_\B}$. 
Although $\J$ is not an ideal inside $\End^{*}_{\B}(\H_{\B})$, 
we will see that $\J$ admits approximate units that are 
quasicentral inside $\End^{*}_{\B}(\H_{\B})$. That is 
$[u_{\lambda}\cdot{\rm Id}_{\H_\B}, T]\to 0$ in norm for all $T\in \End_{\B}^{*}(\H_{\B})$. 

By an \emph{ideal} in an operator algebra we will always mean a closed, 
two sided ideal. For an 
operator algebra $\B$ its \emph{amplification} is 
\[
\B^{\infty}:=\{(b_{i})_{i\in\Z}\in\prod_{i\in\Z}\B: \sup_{i\in\Z}\|b_{i}\|<\infty\},
\] 
which is canonically an operator algebra in the indicated norm. 

For a general operator algebra $\B$, the module of infinite columns $\H_{\B}$ is paired with the module of infinite rows $\H_{\B}^{t}$ via
\[((b_{i})_{i\in\Z}^{t},(c_{i})_{i\in\Z}):=\sum b_{i}c_{i},\]
and $\End^{*}_{\B}(\H_{\B})$ is defined to be the algebra of 
completely bounded operators $T:\H_{\B}\to\H_{\B}$ for 
which there exists $\tilde{T}:\H_{\B}^{t}\to \H_{\B}^{t}$ 
such that $(x,Ty)=(\tilde{T}x,y)$ for all $x\in\H_{\B}^{t}$, 
$y\in \H_{\B}$, cf. \cite[Section 3]{Blecherrigged}. 
For operator $*$-algebras, the spaces $\H_{\B}^{t}$ and 
$\H_{\B}$ are anti-isomorphic (\cite[Lemma 4.4.1]{Mes}) 
and this definition of $\End^{*}_{\B}(\H_{\B})$
is equivalent to the one given earlier in  \eqref{adjointable}.

We wish to describe $\End^{*}_{\B}(\H_{\B})$ as an algebra of infinite 
matrices.  Since $\B$ is unital, to an element $T\in\End^{*}_{\B}(\H_{\B})$ 
we can associate matrix coefficients $(T_{ij})$ by using the canonical basis 
$e_{i}$ of $\H_{\B}$. For an infinite matrix $T:=(T_{ij})_{i,j\in\Z}$ 
the $N$-\emph{truncation} is the finite $2N\times 2N$-matrix
\[
T_{N}=(T_{ij})_{N}:=(T_{ij})_{1\leq |i|,|j|\leq N}.
\]
\begin{lemma}
\label{lem: l2rowcolumn} 
Let $T\in \End^{*}_{\B}(\H_{\B})$ and $\pi:\B\to\bB(\H)$ be  
a completely isometric representation. Then the matrix coefficients $T_{ij}\in\B$ satisfy

\[
\sum_{i}\pi(T_{ij})^{*}\pi(T_{ij})<\infty,\qquad \sum_{j}\pi(T_{ij})\pi(T_{ij})^{*}<\infty,
\]
where these series are norm convergent in $\bB(\H)$. For any $(b_{i})_{i\in\Z}\in \H_{\B}$ the series
\[\sum_{j\in\Z} \pi(T_{ij})\pi(b_{j}),\quad \sum_{j}\pi(T_{ij})^{*}\pi(b_{j}),\]
are norm convergent in $\bB(\H)$. 
\end{lemma}
\begin{proof} 
Using the basis vectors $e_i$ and considering 
\[T(e_j)=(T_{ij})_{i\in\Z}\in\H_{\B},\,\tilde{T}(e_j)^{t}=(T_{ji})_{i\in\Z}^{t}\in\H_{\B}^{t},\] we obtain the stated conditions 
on the rows and columns of $T$. Considering the series $\sum_{j}T_{ij}b_{j}$, estimate the tails by
\[\begin{split} \|\sum_{|j|\geq n} \pi(T_{ij}b_{j})\|^{2}&=\|\sum_{|j|,|k|\geq n}\pi(x_{k})^{*}\pi(T_{ik})^{*}\pi(T_{ij})\pi(b_{j})\|\\
&=\|(\pi(b_{k}))_{|k|\geq n}^{*}\cdot (\pi(T_{ik})^{*})_{|k|\geq n}\cdot(\pi(T_{ij}))^{t}_{|j|\geq n}\cdot(\pi(b_{j}))_{|j|\geq n}\|\\
&\leq \|\sum \pi(x_{j})^{*}\pi(x_{j})\|\| (\pi(T_{ik})^{*})_{|k|\geq n}\cdot(\pi(T_{ij}))^{t}_{|j|\geq n}\|\\
&=\|x\|^{2}_{\H_{\B}}\|\sum_{|j|\geq n} \pi(T_{ij})\pi(T_{ij})^{*}\|\to 0,\end{split}\]
as $n\to \infty$ because the the rows $(T_{ij})_{j\in\Z}$ are elements of $\H_{\B}$. The argument for the series $\sum_{i}\pi(T_{ij})^{*}\pi(T_{ij})$ is similar, now using the condition on the columns of $(T_{ij})$.
\end{proof}
Given a closed, two-sided ideal $\J\subset\B$ there is an 
embedding $\H_{\J}\rightarrow \H_{\B}$ and we define the subalgebra 
\begin{equation}\label{MZ}
\End^{*}_{\B}(\H_{\B},\J)=\{T\in \End_{\B}^{*}(\H_{\B}): T\H_{\B}\subset \H_{\J}\}.
\end{equation}
\begin{lemma}
Every $T\in\End^{*}_{\B}(\H_{\B})$ has the property that $T\H_{\J}\subset \H_{\J}$. Consequently the subalgebra $\End^{*}_{\B}(\H_{\B},\J)$ 
is a closed two-sided ideal in $\End^{*}_{\B}(\H_{\B})$. The algebra
$\End^{*}_{\B}(\H_{\B},\J)$ can be equivalently 
defined as the subalgebra of those $T\in\End^{*}_{\B}(\H_{\B})$ all of 
whose matrix coefficients $T_{ij}\in\J$.
\end{lemma}
\begin{proof} 
Let $x=(b_{j})_{j\in\Z}\in \H_{\J}$ and $e_{i}$ the standard basis elements of $\H_{\B}$. 
We need to show that for $T\in \End^*_\B(\H_\B)$, the coordinates $\langle e_{i},Tx\rangle$ 
are elements of $\J$. For an isomtetric representation $\pi$ we have 
\[
\pi(\langle e_{i},Tx\rangle) =\sum_{j\in\Z} \pi(T_{ij}b_{j})\in \J,
\]
which is a convergent series by Lemma \ref{lem: l2rowcolumn}. 
The elements of the series lie in $\J$ since $b_{j}\in\J$, and 
$\J$ is closed, so it follows that $\langle e_{i},Tx\rangle\in\J$.
  
To see that $\End_{\B}^*(\H_{\B},\J)$ is closed, we use the fact that 
$\H_{\J}\subset \H_{\B}$ is closed.  For then if $T_{n}$ is a sequence in 
$\End_{\B}^*(\H_{\B},\J)$ which is Cauchy for the norm on 
$\End^{*}_{\B}(\H_{\B})$ then for each $x\in\H_{\B}$ the sequence 
$T_{n}x\in \H_{\J}$ is Cauchy. Hence the limit $Tx\in \H_{\B}$ is actually 
an element of $\H_{\J}$. For $S\in\End_{\B}^{*}(\H_{\B})$, $T\in\End^{*}_{\B}(\H_{\B},\J)$ 
and $x\in \H_{\B}$ we have $STx\in \H_{\J}$ because $Tx\in \H_{\J}$ and 
$TSx\in \H_{\J}$ by definition of $\End^{*}_{\B}(\H_{\B},\J)$.

It is immediate that for $T\in \End_{\B}^{*}(\H_{\B},\J)$ all the $T_{ij}$ are in $\J$. 
Conversely, if $T_{ij}\in \J$ for all $i,j$ then by Lemma \ref{lem: l2rowcolumn} 
for any $(b_{i})_{i\in\Z}\in\H_{\B}$ we have $T(b_{i})_{i\in\Z}=(\sum T_{ij}b_{j})_{i\in\Z}$ 
which is an element of $\H_{\B}$ all of whose coordinates are in $\J$ and 
hence an element of $\H_{\J}$.
\end{proof}

For a directed set $\Lambda$, the set
\[
\Lambda^{\infty}:=\{(\lambda_{i})_{i\in\Z}:\lambda_{i}\in\Lambda\},
\]
is a directed set with the partial order 
\[
(\lambda_{i})\leq(\mu_{i})\Leftrightarrow \lambda_{i}\leq\mu_{i},\quad\textnormal{ 
for all } i.
\]
If $\{u_{\lambda}\}_{\lambda\in\Lambda}$ is a bounded approximate unit for an ideal $\J\subset\B$, then 
\[
u_{(\lambda_{i})}:=\{\textnormal{diag }(u_{\lambda_{i}})_{i\in\Z}\}_{(\lambda_{i})\in\Lambda^{\infty}},
\]
is a net in $\End^{*}_{\B}(\H_{\B},\J)$ indexed by $\Lambda^{\infty}$. 
The diagonal matrices $v_{n,\lambda}$ defined by
\[
(v_{n,\lambda})_{ii}:= \left\{\begin{matrix} v_{\lambda} & |i|\leq n \\ 0, & |i|>n\end{matrix}\right.
\]
constitute a subnet.
The algebra $\End^{*}_{\B}(\H_{\B},\J)$ admits an approximate unit whenever $\J$ does.
\begin{lemma}
\label{lem: ampau} 
If $(u_{\lambda})_{\lambda\in\Lambda}$ is a bounded approximate unit for $\J$, then:\\
1) the net 
$(u_{(\lambda_{i})})_{(\lambda_{i})\in\Lambda^{\infty}}$,
is a bounded approximate unit for $\End^{*}_{\B}(\H_{\B},\J)$;\\
2) the subnet $(v_{n,\lambda})$ is a bounded approximate unit for $\kK\hotimes\J$.
In particular $\kK\hotimes\J$ has a sequential approximate unit whenever $\J$ does.
\end{lemma}
\begin{proof}  
Given an operator $T:=(b_{ij})$ and $\varepsilon >0$, by 
Lemma \ref{lem: l2rowcolumn} we can choose $\lambda_{j}$, $j\in\Z$, 
such that the columns $(b_{ij})_{i}$ satisfy
\[
\|(b_{ij}-b_{ij}u_{\lambda_{j}})_{i}\|<\varepsilon 2^{-(|j|+1)}.
\]
Since $(\lambda_{j})_{j\in\Z}\in\Lambda^{\infty}$, the matrix $u_{(\lambda_{j})}=\textnormal{diag}(u_{\lambda_{j}})$ is an element of the directed set in 1). We can estimate
\[
\begin{split}\|T\textnormal{diag}(u_{\lambda_{j}})-T\|
&=\|(b_{ij})\textnormal{diag}(u_{\lambda_{j}})-(b_{ij})\| 
=\|\sum_{j}(b_{ij}u_{\lambda_{j}}-b_{ij})_{i}\| \\
&\leq \sum_{j}\|(b_{ij}u_{\lambda_{j}}-b_{ij})_{i}\| 
< \sum_{j}\varepsilon 2^{-(|j|+1)} 
\ \leq \varepsilon,\end{split}
\]
showing we have a right approximate unit. In a 
similar way one shows that the directed set is a left approximate unit. The proof of 2) is similar
but easier.
\end{proof}

We define two representations $\pi:\B\rightarrow\bB(\H_{\pi})$ and 
$\rho:\B\rightarrow\bB(\H_{\rho})$ to be \emph{cb-equivalent} if there 
exists a cb-isomorphism $g:\H_{\pi}\rightarrow\H_{\rho}$ such that $\pi=g^{-1}\rho g$.
\begin{lemma}
\label{lem: decompose} 
Let $\pi:\End^{*}_{\B}(\H_{\B})\rightarrow\bB(\H)$ be a cb-representation. 
There exist idempotents $q_{i}\in\pi(\End^{*}_{\B}(\H_{\B}))\subset\bB(\H)$ 
such that $\sum_{i}q_{i}=\pi(1)$ and $\sup_{i}\|q_{i}\|<\infty$. Consequently, 
$\pi$ is cb-equivalent to the representation
\[
\begin{split} 
\End^{*}_{\B}(\H_{\B})\rightarrow\bB\Big(\bigoplus_{i\in\Z}q_{i}\H\Big) \qquad
(b_{ij})\mapsto (q_{i}\pi(b_{ij}) q_{j}).
\end{split}
\]
\end{lemma}
\begin{proof} 
Using matrix coefficients and embedding $\B$ in the 
$(i,i)$-diagonal slot of $\End_{\B}^{*}(\H_{\B})$ 
we obtain from $\pi$ a family of representations 
$\pi_{i}:\B\rightarrow \bB(\H)$ satisfying $\pi_{i}\pi_{j}=0$. The elements $q_{i}:=\pi_{i}(1)$ 
are the corresponding idempotents in $\bB(\H)$. We also write $q=\pi(1)$. If we 
write $\H_{i}:=q_{i}\H$, then since $q_{i}q_{j}=\delta_{ij}$, and $\|q_{i}\|\leq\|\pi\|$, 
we have a cb-isomorphism $q\H\cong\bigoplus_{i\in\Z}\H_{i}$ and so a 
cb-isomorphism $g:\H\rightarrow (1-q)\H\oplus\bigoplus_{i\in\Z}\H_{i}$. 
Using the identifications $[\pi(\B)\H]=q\H$ and $\Nil\pi(\B)=(1-q)\H$, this clearly gives a 
cb-equivalence between $\pi$ and the matrix representation $(q_{i}\pi(a_{ij})q_{j})$.
\end{proof}
\begin{lemma}
\label{lem: qcomm} 
Let $\J\subset\B$ be an ideal in a unital operator algebra $\B$ and 
let $(u_{\lambda})$ be a bounded approximate unit 
for $\J$. Let $\pi:\End^{*}_{\B}(\H_{\B})\rightarrow \bB(\H)$ 
be a cb-representation. Consider the subalgebras 
$\J\cong\J\cdot \textnormal{Id}_{\H_{\B}}\subset\End^{*}_{\B}(\H_{\B},\J)\subset\End^{*}_{\B}(\H_{\B})$ 
and the representations $\pi:\J\rightarrow \bB(\H)$ 
and $\pi:\End^{*}_{\B}(\H_{\B},\J)\rightarrow \bB(\H)$, defined by restriction. Then: \\
1) there is an equality of essential subspaces $[\pi(\J)\H]=[\pi(\End^{*}_{\B}(\H_{\B},\J))\H]$;\\
2) the idempotent $q=w$-$\lim u_{\lambda}\in\bB(\H)$ from 
Proposition \ref{thm: idempotent} commutes with $\pi(T)$ for all $T\in \End^{*}_{\B}(\H_{\B})$.
\end{lemma}
\begin{proof} 
First we show that $[\pi(\J)\H]=[\pi(\End^{*}_{\B}(\H_{\B},\J))\H]$. It is clear that 
\[[\pi(\J)\H]\subset [\pi(\End^{*}_{\B}(\H_{\B},\J))\H],\] so we proceed to show the reverse inclusion. By 
Lemma \ref{lem: decompose} we may assume that there are idempotents 
$q_{i}:\H\rightarrow \H$ such that $\H=\bigoplus_{i\in\Z}\H_{i}$ and  
$\pi(a_{ij})=(q_{i}\pi(a_{ij})q_{j})$. We wish to show that for 
$(h_{i})\in  [\pi(\End^{*}_{\B}(\H_{\B},\J))\H]$, it holds that 
$\pi(u_{\lambda}\cdot{\rm Id}_{\H_{\B}})(h_{i})\rightarrow (h_{i})$, 
so that $(h_{i})\in [\pi(\J\cdot{\rm Id}_{\H_{\B}})\H]$. Thus we must show that for every $\varepsilon >0$ 
there exists $\mu\in\Lambda$ such that for all $\lambda\geq \mu$ it holds that 
$\|\pi(u_{\lambda}\cdot{\rm Id}_{\H_{\B}})(h_{i})-(h_{i})\|<\varepsilon$. So let $\varepsilon>0$ 
and choose $N\in\N$ such that 
\[
\Big\|\sum_{|i| > N}\langle h_{i},h_{i}\rangle\Big\|^{\frac{1}{2}}<\frac{\varepsilon}{2(C\|\pi\|+1)},
\]
where $C:=\sup_{\lambda}\|u_{\lambda}\|$. We claim we can choose $\mu$ 
such that for all $\lambda\geq \mu$ and  $1\leq |i| \leq N-1$
\begin{equation}
\label{mu}
\|\pi(u_{\lambda}\cdot{\rm Id}_{\H_{\B}})(h_{i})_{|i|<N}-(h_{i})_{|i|<N}\|<\frac{\varepsilon}{2N}.
\end{equation}
To see this, first observe that $(h_{i})_{1\leq |i|\leq N}\in [\pi(\End^{*}_{\B}(\H_{\B},\J))\H_{i}]$. 
This is the case because \[q_{[N]}:=\sum_{1\leq |i|\leq N}q_{i}\in \End^{*}_{\B}(\H_{\B}),\] 
and $(h_{i})_{1\leq |i|\leq N}= q_{[N]}(h_{i})$. Then, by Lemma \ref{Nil} and 
Lemma \ref{lem: ampau}
$$
\pi(u_{\lambda}\cdot\textnormal{Id}_{\H_{\B}})(h_{i})_{|i|<N}\rightarrow (h_{i})_{|i|<N},
$$
and since we only deal with finitely many entries (at most $2N$), 
this means we can choose $\mu$ as in Equation \eqref{mu}. 
Thus we have for $\lambda\geq\mu$ that
\[
\begin{split}
\|\pi(u_{\lambda}\cdot{\rm Id}_{\H_{\B}})(h_{i})-(h_{i})\|  
& \leq \|\pi(u_{\lambda}\cdot{\rm Id}_{\H_{\B}})(h_{i})_{1\leq |i|\leq N}-(h_{i})_{1\leq |i|\leq N}\|\\
&\qquad\quad+\|\pi(u_{\lambda}\cdot{\rm Id}_{\H_{\B}})(h_{i})_{|i|> N}-(h_{i})_{|i|> N}\| \\
&\leq \sum_{1\leq |i|\leq N}\|\pi_{i}(u_{\lambda}\cdot{\rm Id}_{\H_{\B}})h_{i}-h_{i}\|\\
&\qquad\quad+\Big\|\sum_{|i|> N}\langle(\pi(u_{\lambda}\cdot{\rm Id}_{\H_{\B}}-{\rm Id}_{\H_{\B}})h_{i}, (\pi(u_{\lambda}\cdot{\rm Id}_{\H_{\B}}-{\rm Id}_{\H_{\B}})h_{i}\rangle\Big\|^{\frac{1}{2}}\\
&< \frac{\varepsilon}{2}+(\|\pi(u_{\lambda}\cdot{\rm Id}_{\H_{\B}})\|+1)\,
\Big\|\sum_{|i|> N}\langle h_{i}, h_{i}\rangle\Big\|^{\frac{1}{2}}\\
&< \frac{\varepsilon}{2}+\frac{(\|\pi(u_{\lambda}\cdot{\rm Id}_{\H_{\B}})\|+1)\varepsilon}{2(C\|\pi\|+1)}
< \varepsilon,
\end{split}
\]
showing that $\pi(u_{\lambda})(h_{i})\rightarrow (h_{i})$. 

In the same vein $[\pi(\J)^{*}\H]=[\pi(\End^{*}_{\B}(\H_{\B},\J))^{*}\H]$.
As $\End^{*}_{\B}(\H_{\B},\J))$ is an ideal in $\End^{*}_{\B}(\H_{\B})$, the subspace 
$[\pi(\End^{*}_{\B}(\H_{\B},\J))\H]$ is $\End^{*}_{\B}(\H_{\B})$-invariant. The topological 
complement, given by $[\pi(\End^{*}_{\B}(\H_{\B},\J))^{*}\H]^{\perp}$, is $\End^{*}_{\B}(\H_{\B})$ 
invariant as well: Let $v\in [\pi(\End^{*}_{\B}(\H_{\B},\J))^{*}\H]^{\perp}$ and 
$h\in [\pi(\End^{*}_{\B}(\H_{\B},\J))^{*}\H]$. Then $\pi(T)^{*}h\in  [\pi(\End^{*}_{\B}(\H_{\B},\J))^{*}\H]$ 
because $\End^{*}_{\B}(\H_{\B},\J)$ is an ideal and thus
\[
\langle \pi(T)v,h\rangle=\langle v,\pi(T)^{*}h\rangle=0.
\]
That is $\pi(T)v\in [\pi(\End^{*}_{\B}(\H_{\B},\J))^{*}\H]^{\perp}$. From $3),\,4)$ of 
Proposition \ref{thm: idempotent}, we see that
\[
q\pi(T)q=\pi(T)q,\quad (1-q)\pi(T)(1-q)=\pi(T)(1-q),
\]
from which $q\pi(T)=\pi(T)q$ follows readily.
\end{proof}
\begin{lemma}
\label{weak} 
Let $(u_{\lambda})$ be a bounded approximate unit for an ideal $\mathcal{I}$ in an operator  
algebra $\mathcal{B}$. Then for all $T\in \End^{*}_{\B}(\H_{\B})$,  
$[ u_{\lambda}\cdot\textnormal{Id}_{\H_{\B}},T]\xrightarrow{w} 0$. That is, 
$u_{\lambda}\cdot\textnormal{Id}_{\H_{\B}}$ commutes with $\End^{*}_{\B}(\H_{\B})$ weakly asymptotically.
\end{lemma}
\begin{proof} The argument we give is modelled on the proof of \cite[Lemma 3.1]{AP}. 
We assume that $\End^{*}_{\B}(\H_{\B})$, and hence $\J$ and $\B$ are 
completely isometrically embedded in $\bB(\H)$. Let the linear functional $\phi:\End^{*}_{\B}(\H_{\B})\rightarrow \C$ 
be continuous. By the Hahn-Banach theorem we may extend $\phi$ to the enveloping 
$C^{*}$-algebra $C^{*}(\End^{*}_{\B}(\H_{\B}))$, the $C^{*}$-algebra generated 
by $\End^{*}_{\B}(\H_{\B})\subset\bB(\H)$. Since every element in the 
dual of a $C^{*}$-algebra is a linear combination of four states, it 
suffices to prove weak convergence with respect to all states of $C^{*}(\End^{*}_{\B}(\H_{\B}))$. 
If we denote by $\pi_{u}:C^{*}(\End^{*}_{\B}(\H_{\B}))\to\bB(\H_{u})$ the universal GNS-representation of 
$C^{*}(\End^{*}_{\B}(\H_{\B}))$, the state $\phi$ has the form 
$b\mapsto\langle\pi_{u}(b)v,v\rangle$, where $v$ is a  vector 
in the GNS-space $\H_{u}$. 
Since $(u_{\lambda})$ is an approximate unit for $\mathcal{J}$, $\pi(u_{\lambda})$ 
converges strongly to an idempotent $q$ onto $[\pi(\mathcal{J})\mathcal{H}_{u}]$. 
By Lemma \ref{lem: qcomm}, $q$ commutes with $\pi(a_{ij})$. Hence 
\[
\lim_{\lambda}\phi([u_{\lambda}\cdot\textnormal{Id}_{\H_{\B}},T])=\langle v,[q,\pi_{u}(T)]v\rangle=0.\qedhere
\]
\end{proof}
\begin{defn} 
Let $\mathcal{J}\subset \B$ be an ideal and $u_{\lambda}$ a 
bounded approximate unit for $\mathcal{J}$. The bounded approximate unit 
$u_{\lambda}$ is said to be 
$\End^{*}_{\B}(\H_{\B})$-\emph{quasicentral} if for all $T\in\End^{*}_{\B}(\H_{\B})$ 
we have $\lim_{\lambda}\| [T, u_{\lambda}\cdot\textnormal{Id}_{\H_{\B}}]\|=0$.
\end{defn}
We are now ready to establish the existence of 
$\End^{*}_{\B}(\H_{\B})$-quasicentral approximate units. 
It should be noted that this result is new even for $C^{*}$-algebras.
\begin{thm}
\label{Ped}
Let $(u_{\lambda})_{\lambda\in \Lambda}$ be a bounded approximate unit for a closed ideal 
$\mathcal{J}$ in a unital operator algebra $\mathcal{B}$. Then there is an 
$\End^{*}_{\B}(\H_{\B})$-quasicentral approximate unit $(v_{\mu})_{\mu\in M}$ 
for $\mathcal{J}$, contained in the convex hull of $(u_{\lambda})$.
\end{thm}
\begin{proof} The proof is formally identical to that of \cite[Theorem 3.2]{AP}. 
We assume $\End^{*}_{\B}(\H_{\B})$ is isometrically isomorphically embedded in $\bB(\H)$.
Denote by $\mathscr{C}(u_{\lambda})$ the convex hull of $(u_{\lambda}\cdot{\rm Id}_{\H_{\B}})$. 
Choose elements $b_{1},\dots,b_{n}\in \End^{*}_{\B}(\H_{\B})$, $v\in\mathscr{C}(u_{\lambda})$ 
and $\varepsilon>0$. Consider 
\[
\mathscr{C}(u_{\lambda})\subset \End^{*}_{\B}(\H_{\B})^{n}
=\bigoplus_{i=1}^{n}\End^{*}_{\B}(\H_{\B}),
\]
by diagonal embedding and set $b=\textnormal{diag}(b_{1},\dots,b_{n})$. The set 
$\mathscr{C}_{b}:=\{[u,b]:u\in \mathscr{C}(u_{\lambda})\}$ is 
convex and hence its norm and weak closures in $\bB(\mathcal{H}^{n})$ 
coincide. By Lemma \ref{weak}, $[u_{\lambda}\cdot\textnormal{Id}_{\H_{\B}},b]\xrightarrow{w} 0$, and 
hence $0$ must be a norm limit of elements of $\mathscr{C}_{b}$. That is,  
there exists $v\in\mathscr{C}_{b}$ with $\|[v,b]\|<\varepsilon$. Letting 
$\Omega$ denote the set of finite subsets of $\End^{*}_{\B}(\H_{\B})$, the 
argument shows that for each pair $(\lambda,\omega )\in \Lambda\times\Omega$ 
there is a $v_{\lambda, \omega}\in\mathscr{C}(u_{\mu}:\lambda\leq\mu)$ for which
\[
\|[v_{\lambda\omega},b]\|<\frac{1}{|\omega|},
\]
for all $b\in\omega$.
The relation 
\[
(\lambda,\omega)\leq (\lambda',\omega ')
\Leftrightarrow \lambda\leq\lambda'\ \mbox{and}\  \omega\subset\omega ',
\]
defines a partial order on $\Lambda\times\Omega$, with 
respect to which $(v_{\lambda\omega})$ is a bounded approximate unit. 
\end{proof}
The next theorem considers quasicentral approximate units for 
algebras of multipliers, relative to a second ideal. This result is 
not as general as the above theorem, but provides the statement we 
need  for our refinement of the Kasparov technical theorem in the next section.

\begin{thm}
\label{Ped2} 
Let $\B$ be a unital operator algebra and $\K$ an
ideal with bounded approximate unit $(v_{n})$. 
Assume that $\J$, $\A\subset \B$ are subalgebras 
such that $\J$ is an ideal in $\B$, $\J\A$, $\A\J\subset \K$ and $\K\A=\A\K=\K$. 
If $\A$ has a bounded approximate unit $(u_{k})$ 
then there exists an $\End^{*}_{\B}(\H_{\B},\J)$ quasicentral 
approximate unit for $\A$ contained in the convex hull of $(u_{k})$.
\end{thm}
\begin{proof} 
Assume without loss of generality that $\End_{\B}^{*}(\H_{\B})$ is completely 
isometrically embedded in $\bB(\H)$ and let 
$C^{*}(\End_{\B}^{*}(\H_{\B}))\subset \bB(\H)$ denote its enveloping 
$C^{*}$-algebra in this representation. We wish to prove that for all functionals 
$\phi:\End^{*}_{\B}(\H_{\B},\K)\rightarrow \C$, and all $T\in \End^{*}_{\B}(\H_{\B},\J)$ we have
$\phi([u_{k},T])\rightarrow 0.$ 

As in the proof of Lemma \ref{weak}, it will suffice to prove 
this for vector states $\phi=\langle v,\cdot v\rangle$ on $C^{*}(\End_{\B}^*(\H_{\B}))$ 
coming from the universal GNS representation 
$\pi_{u}:C^{*}(\End_{\B}(\H_{\B}))\rightarrow \bB(\H_{u})$. 
Since both $\K$ and $\A$ have bounded approximate 
units, Proposition \ref{thm: idempotent} gives  two
idempotents: $p$ mapping onto 
\[
[\pi_{u}(\K\cdot\textnormal{Id}_{\H_{\B}})\H_{u}]=[\pi_{u}(\End^{*}_{\B}(\H_{\B},\K))\H_{u}],
\] 
(cf. Lemma \ref{lem: qcomm}) 
and $q$ mapping onto $[\pi_{u}(\A\cdot\textnormal{Id}_{\H_{\B}})\H_{u}]$ as the strong limits of 
$(v_{n})$ and $(u_{k})$ respectively. Since $\A\K=\K$, we 
have $[\pi_{u}(\A)[\pi_{u}(\K)\H_{u}]]=[\pi_{u}(\K)\H_{u}]$ and thus $u_{k}\rightarrow 1$ 
strongly on $[\pi_{u}(\K)\H_{u}]=p\H_{u}$, again by Proposition 
\ref{thm: idempotent}. We claim that $pq=qp=p$. 
To see this, first observe that $pq=qp$, since by 
Lemma \ref{weak}, $p$ commutes with all elements of 
$\B$ and hence in particular with $\A$. Therefore, for any $h\in\H$ we have
\[
qph=\lim u_{k}ph=\lim pu_{k}h=pqh.
\]
Then since $(u_{k})$ converges strongly to $1$ on $[\pi_{u}(\K)\H_{u}]$ it follows that 
\[
qph=\lim \pi(u_{k})ph=ph,
\]
which proves our claim. Now let $T\in \End^*_{\B}(\H_{\B},\J)$ and 
consider $[u_{k},T]\in \End_{\B}^{*}(\H_{\B},\K)$. The operator 
$T$ commutes with $p$ and since $[u_{k},T]\in \End_{\B}^{*}(\H_{\B},\K)$ 
this operator equals $p[u_{k},T] p$, and
\begin{align*}
\lim \phi([u_{k},T])=\lim\langle v,  \pi_{u}([u_{k},T])v\rangle
&=\lim \langle v,p\pi_{u}([u_{k},T])pv\rangle\\
&\rightarrow\langle v,p[q,\pi_{u}(T)]pv\rangle=\langle v,[p,\pi_{u}(T)]v\rangle=0.
\end{align*}
Thus, the commutators $[u_{k},T]$ converge to $0$ 
in the weak topology of $ \End_{\B}^{*}(\H_{\B},\K)$. The same 
argument as in the proof of Theorem \ref{Ped} 
now shows that the convex hull of $(u_{k})$ contains 
an approximate unit that is quasicentral for $ \End_{\B}^{*}(\H_{\B},\J)$.
\end{proof}

\subsection{Completeness and the technical theorem}
Having established the existence of quasi-central approximate units in 
operator algebras with bounded approximate unit, we can formulate an 
extension of Kasparov's technical theorem in the spirit of Higson, \cite{H}. 
For practical purposes we state the following corollary of 
Theorem \ref{Ped} as a Lemma. When $(u_{n})$ is a sequential approximate unit, 
we say that $(v_{n})\in\mathscr{C}(u_{k})$  is a sequence of \emph{far out} 
convex combinations if $(v_{n})\in\mathscr{C}(u_{k}:k\geq n)$.
\begin{lemma}
\label{uniformseq} 
Let $\J$ be an ideal in a separable unital operator algebra 
$\B$, $(u_{n})$ a sequential bounded approximate unit for $\J$. Let  $(z_{i})_{i\in\N}\subset \B$ 
a countable norm bounded subset of $\B$ and  $1>\varepsilon>0$. There exists a 
$\B$ quasicentral, sequential bounded approximate unit $(v_{n})$ for $\J$, contained 
in the convex hull of $(u_{n})$, such that  
\[
\sup_{i\in\N}\| [v_{n+1}-v_{n},z_{i}] \|<\varepsilon^{n}.
\]
\end{lemma}
\begin{proof} 
Assuming $\mathcal{B}$ separable and 
$(u_{n})$ countable, the new approximate 
unit can be chosen in such a way as to satisfy the asserted properties. 
This is done by choosing a countable dense subset 
$\{b_{1},b_{2},\dots \}$ of $\mathcal{B}$, embedding 
$\B$ in $\B^{\infty}$ diagonally as usual, and considering 
$z:=\diag(z_{1},\dots,z_{i},\dots)\in\B^{\infty}$. From the proof of 
Theorem \ref{Ped} we obtain an $\End_{\B}^{*}(\H_{\B})$ quasicentral 
approximate unit $v_{n,\omega}$ indexed by $\N\times\Omega$, with 
$\Omega$ the set of finite subsets of $\End_{\B}^{*}(\H_{\B})$. Now 
choose $v_{0}:=v_{0\{z\}}$ and inductively assume 
$v_{k}:=v_{n_{k},\omega_{k}}\in\mathscr{C}(u_{n})$ 
was chosen from $v_{n,\omega}$ in such a way that
\[
v_{k}=\sum_{i=n_{k}}^{m_{k}}\theta_{i}u_{i},\quad \theta_{i}\in [0,1],\quad  
\sum_{i=n_{k}}^{m_{k}}\theta_{i}=1,
\]
for some $n_{k},m_{k}\in\N$. Now choose 
\[
v_{k+1}:=v_{n_{k+1},\{x, b_{1},\dots,b_{n_{k+1}}\}},
\] 
where $n_{k+1}>m_{k},$ yielding a sequential approximate unit for 
$\J$, which is quasicentral for the $(b_{i})$ and hence for $\B$, 
as well as for the element $z\in\B^{\infty}$. It should be noted that 
quasicentrality does not necessarily hold for all of $\B^{\infty}$. 
By choosing a subsequence, we can realise that
\[
\| [(v_{n+1}-v_{n})\cdot{\rm Id}_{\H_\B},x] \|
=\| [(v_{n+1}-v_{n})\cdot{\rm Id}_{\H_\B},\diag(z_{1},\dots,z_{i},\dots)] \|
=\sup_{i\in\N}\| [v_{n+1}-v_{n},z_{i}] \| <\varepsilon^{n},
\]
for the given sequence $z_{i}$
as desired.
\end{proof}

\begin{thm}
\label{thm: technical} 
Let $\B$ be a unital operator algebra, $\K\subset \B$ an ideal with 
countable bounded commutative approximate unit and 
$\A,\J\subset \B$ closed separable subalgebras with
$\A\J$, $\J\A\subset \K$, $\K\subset \J$ and $\A\K=\K\A=\K$. Suppose we are given\\ 
1) a bounded total subset $\{a_{i}\}\subset\A$ and countable bounded 
approximate units $(u'_{k})\subset\A$, $(v'_{k})\in\J$, and \\
2) $F\in\B$ such that, $[F,\A]\subset \K$ and  $\lim_{k\rightarrow\infty}\|[F,u_{k}']\|=0$.

For any $0<\e<1$, there exist countable bounded approximate 
units $(v_{n}),(u_{k})$ contained in the convex hull 
$\mathscr{C}(v'_{k})$ and $\mathscr{C}(u'_{k})$ respectively such that for  
$d_{n}:=v_{n+1}-v_{n}$ the following convergence properties hold:\\
1) $\|[d_{n},F]\|<\e^{2n}$;\\
2) $\|[d_{n},a_{i}]\|<\e^{2n}$  for all $i$;\\
3) $\|[d_{n},u_{k}]\|<\e^{2n}$ for all $k$;\\
4) $\|[d_{n},u_{k}]\|<\e^{2k}$ for all $n$;\\
5) $\|d_{n}[F,a_{i}]\|<\e^{2n}$ for  $n\geq i$;\\
6) $\|d_{n}[F,u_{k}]\|<\e^{2k}$ for all  $n;$\\
7) $\|d_{n}[F,u_{k}]\|<\e^{2n}$ for  $n\geq k;$\\
8) $\|u_{k}a_{i}-a_{i}\|<\e^{2k}$ for $k>i$;\\
9) $\|d_{n}[F,u_{k}a_{i}-a_{i}]\|<\e^{k+2n}$ for  $n>k>i$;\\
10)$\|[d_{n},u_{k}a_{i}-a_{i}]\|<\e^{k+2n}$ for  $n>k>i$.

In fact these properties continue to hold for
any subsequence $(\tilde{v}_{n}):=(v_{k_{n}})$ and the conclusion  
holds for any finite number of subalgebras $\A_{1},\dots, \A_{n}$ satisfying the hypotheses 
on $\A$.
\end{thm}
\begin{proof} The case of a finite number of algebras $\A_{1},\dots,\A_{n}$ is reduced to that of a single algebra by setting $\A:=\oplus_{i=1}^{n}\A_{i}\subset \oplus^n_{i=1}\B$ and $\oplus^n_{i=1}\J$. 
We thus prove the theorem for a single algebra $\A$.

Embed $\B$, and hence $\K$, $\J$ and $\A$ into 
$\End_{\B}^{*}(\H_{\B})$ as multiples of the identity operator $\Id_{\H_{\B}}$. The elements 
\[
a:=\textnormal{diag}(a_{1},\dots,a_{i},\dots)\quad \textnormal{ and }\quad 
u:=\textnormal{diag}(u_{1}',\dots,u_{k}',\dots),
\]
are elements of 
$\End^{*}_{\B}(\H_{\B})$ as well. Thus by Theorem \ref{Ped}, there exists a 
countable, commutative, approximate unit 
$(v_{n})\subset \mathscr{C}(v'_{n})$, which is quasicentral for $F, \, a,\,u$.

Write $d_{n}:=v_{n+1}-v_{n}$. By quasicentrality for $F, a$ and $u$, we can 
re-index the $(v_{n})$ if necessary, and we may assume that
\begin{equation}
\label{estimates} 
\|[d_{n},F]\|, \|[d_{n},a]\| ,\| [d_{n}, u] \| <\e^{2n},
\end{equation}
which in particular means that
\[\|[d_{n}, u_{k}'] \|, \| [d_{n},a_{i}] \|< \e^{2n},\quad\textnormal{for all } i,k.\]
This proves the estimates $1)$, $2)$ and $3)$ for $u_{k}'$. 

Next we construct an approximate unit $(u_{k})\subset\mathscr{C}(u_{k}')$ 
satisfying 3) and 4). The hypotheses imply that $\J$ is an ideal in the algebra $\tilde{\B}:=1+\A+\J\subset\B$. 
 The uniformly bounded sequence 
$d_{n}\in \J$ defines an element $d:=$ diag$(d_{1},\dots, d_{n},\dots )\in \End_{\tilde{\B}}^{*}(\H_{\tilde{\B}},\J)$. 
Apply Theorem \ref{Ped2} to the approximate unit $(u_{k}')$ to obtain a 
countable approximate unit $(u_{k})\subset\mathscr{C}(u_{k}')$ which is 
quasicentral for $d$. Thus we may assume that
\[
\|[d_{n},u_{k}]\|<\e^{2k},\quad\textnormal{for all } n.
\]
Property $3)$ is preserved under convex combinations, and is thus valid for $u_{k}$.

For 5), since $[F,a_{i}]\in \K\subset \J$, we may reindex $v_{n}$ and assume that
\begin{equation}
\label{estimates2}
\|d_{n}[F,a_{i}]\|<
\e^{2n},\quad\textnormal{whenever } n\geq i,
\end{equation}
as desired. 

To prove $6)$, observe that because by assumption 
$\lim_{k\to\infty} \|[F,u_{k}']\|=0$, we may assume that
\[
\|[F,u_{k}]\|<\frac{\e^{2k}}{2C},
\]
with $C:=\sup_{k}\|v_{k}'\|$. Then since $\|d_{n}\|\leq 2C$ the claim follows.

For $7)$ we use that $[F,u_{k}]\in\K\subset \J$ so $d_{n}[F,u_{k}]\xrightarrow{n} 0$ 
for each $k$, which, by passing to a subsequence of the $d_{n}$ if necessary, amounts to
\[
\|d_{n}[F,u_{k}]\|<\e^{2n},\quad \textnormal{whenever } n\geq k.
\]
For each $a_{i}$ 
we have the norm convergence $u_{k}a_{i}\xrightarrow{k} a_{i}$. Therefore,
given $0<\e<1$, 
we may re-index the approximate unit $u_{k}$ for $\A$ such that 
\[
\|u_{k}a_{i}-a_{i}\|<\e^{k}, \textnormal{ whenever } k>i.
\]
Note that such a re-indexing does not affect the norm convergence 
$[F,u_{k}]\xrightarrow{k} 0$ nor the properties $3),4)$ and $6)$. 
To preserve property $7)$ we may need to pass to a reindexing of 
$v_{n}$, which can be done without affecting properties $1)-6)$.
This means that we may assume that for all $k>i$ we have 
$\|u_{k}a_{i}-a_{i}\|<\e^{2k}$, which proves 8). 

For 9) and 10), we can, since 8) is true, assume that for all $i\in \Z$ we have that
\[
z_{i}:=\diag\left(\frac{u_{k}a_{i}-a_{i}}{\e^{k}}\right)_{k}\in \A^{\infty}\subset \End^{*}_{\B}(\H_{\B}).
\]
Apply Lemma \ref{uniformseq} to obtain an approximate unit 
$(v_{n})$ for $\J$ which is quasicentral for all the $z_{i}$, and we may achieve 
\begin{equation}
\label{estimates3}
\|[d_{n},z_{i}]\|<\e^{2n},\quad\textnormal{whenever } n\geq i,
\end{equation}
as well. 
Note that the estimates \eqref{estimates}, \eqref{estimates2}, \eqref{estimates3} 
remain valid when $(v_{n})$ is replaced by a subsequence or a 
sequence of far out convex combinations. The same is true 
when $(u_{k})$ is replaced by a subsequence $(\tilde{u}_{k}):=(u_{n_{k}})$, so 
that $u$ is replaced by $\tilde{u}:={\rm diag}(\tilde{u}_{k})_{k}$ and $z_{i}$ by 
$\tilde{z}_{i}:=(\frac{\tilde{u}_{k}a_{i}-a_{i}}{\e^{k}})_{k}$:
\[
\|[d_{n},\tilde{u}]\|
=\sup_{k}\|[d_{n},\tilde{u}_{k}]\|=\|[d_{n},\tilde{u}_{n_{k}}]\|\leq\sup_{k}\|[d_{n},u_{k}]\|<\e^{2n},
\]
and
\[
\begin{split}\|[d_{n},\tilde{z}_{i}]\| 
&=\sup_{k}\|[d_{n},\frac{\tilde{u}_{k}a_{i}-a_{i}}{\e^{k}}]\| 
\leq\sup_{k}\|[d_{n},\frac{u_{n_{k}}a_{i}-a_{i}}{\e^{n_{k}}}]\|
\leq\sup_{k}\|[d_{n},\frac{u_{k}a_{i}-a_{i}}{\e^{k}}]\| <\e^{2n}.
\end{split}
\]

Lastly, since for fixed $i,k$ we have $[F,\frac{u_{k}a_{i}-a_{i}}{\e^{k}}]\in \J$, the sequence
\[
d_{n}[F,\frac{u_{k}a_{i}-a_{i}}{\e^{k}}]\xrightarrow{n} 0,
\]
which means that another re-indexing achieves
\[
\|d_{n}[F,\frac{u_{k}a_{i}-a_{i}}{\e^{k}}]\|<\e^{2n},
\]
for $n>k>i$ and thus
\[
\|d_{n}[F,u_{k}a_{i}-a_{i}]\|<\e^{k+2n},\quad\textnormal{whenever } n>k>i.
\]
This completes the proof of $9)$ and $10)$.
\end{proof}
For our first lifting result, 
we need the following elementary result concerning the strict topology on a $C^{*}$-module.
\begin{lemma}
\label{strictconvex} 
Let $(T_{n})\subset\kK(E_{B})$ be a sequence converging strictly to 
$T\in\kK(E_{B})$. Then there exists a sequence 
$(S_{n})\subset\mathscr{C}(T_{n})$ such that $S_{n}\to T$ in norm. Hence for any essential
Kasparov module $(A,E_B,F)$, $A$ has an approximate unit $(u_n)$ such that $[F,u_n]\to 0$ in norm.
\end{lemma}
\begin{proof} We need to show that $T_{n}\to T$ in the weak topology of $\kK(E_B)$, 
for then there exists a sequence in the convex hull of the $T_{n}$ 
that converges to $S_{n}$ in norm. Since every linear functional on a 
$C^{*}$-algebra is a linear combination of four states, it suffices to 
show that for all states $\sigma:\kK(E_{B})\to \C$ we have 
$\sigma(T_{n})\to \sigma (T)$. Any such $\sigma$ can be realised as a 
vector state for a vector $v_\sigma$ in the universal GNS-representation 
$\H_{u}$ of $\kK(E)$. The representation $\pi_u:\kK(E_B)\to \bB(\H_u)$ is essential, so 
$\pi_{u}(T_{n})$ converges to $\pi_{u}(T)$ weakly in 
$\bB(\H_{u})$. Thus 
$\sigma(T_{n})=\langle v_{\sigma}, \pi_u(T_{n})v_{\sigma}\rangle_{u}
\to \langle v_{\sigma}, \pi_u(T)v_{\sigma}\rangle_{u}=\sigma(T)$ and we are done.

Let $(A, E_{B}, F)$ be a Kasparov module for which the 
$A$ representation is essential. Any approximate unit 
$(u_{n})$ for $A$ will converge strictly to the identity operator on 
$E_B$ and $[F,u_{n}]\to 0$ strictly in $E_{B}$. Therefore the previous argument
yields a contractive 
approximate unit $(u_{n}')$ for $A$ such that $[F,u_{n}']\rightarrow 0$ in norm. 
\end{proof}
\begin{prop}
\label{prop: Kasparovlift} 
Let $A,B$ be separable $C^{*}$-algebras. 
Any class in $KK(A,B)$ can be represented by an unbounded
Kasparov module $(\A,E_B,S)$ such that $\A$ admits an 
approximate unit $(u_{n})$ with $[S,u_{n}]\to 0$ in norm. 
\end{prop}
\begin{proof} 
Let $(A, E_{B}, F)$ be a Kasparov module for which the 
$A$ representation is essential. 
%
Let $(u_k')$ be an approximate unit  for $A$ such that $[F,u_{k}']\rightarrow 0$ in norm,
as in Lemma \ref{strictconvex}. 
Let $v_{n}'$ be a 
contractive approximate unit for $J_{F}$. Applying Theorem \ref{thm: technical} with 
$\B=\End^{*}_{B}(E)$, $\K=\kK(E)$, $\J=J_{F}$, $\A=A$, we obtain approximate 
units $(v_{n})$ for $J_{F}$ and $(u_{n})$ for $A$ with the properties $1)$-$10)$ 
of Theorem \ref{thm: technical}.
Recalling that $d_{n}:=v_{n+1}-v_{n}$, we define
\[
\ell^{-1}=c:=\sum_{n=0}^{\infty}\e^{-n}d_{n},
\]
which is an unbounded multiplier of the ideal $J_{F}$. By 2) of Theorem \ref{thm: technical}, 
$[\ell^{-1},a]$ is bounded on $\im \ell$ for a dense set of $a\in A$, and by 5), 
$[F,a]\ell^{-1}$ is also bounded on $\im \ell$
for the same dense set of $a\in A$. By 1) $[F,\ell^{-1}]$ is bounded on 
$\im \ell$ and  since $(1-F^{2})^{\frac{1}{2}}\in J_{F}$ we may also 
assume $(1-F^{2})^{\frac{1}{2}}\ell^{-1}$ to be bounded of norm $<1$. 
Thus $\ell$ is admissible and by Theorem $\ref{admlift}$ the operator 
$S:=\frac{1}{2}(F\ell^{-1}+\ell^{-1}F)$ lifts $(E_{B},F)$ to an unbounded Kasparov module.

Thus it only remains to check that $[S,u_{k}]\rightarrow 0$ and that 
$[S, u_{k}a_{i}-a_{i}]\xrightarrow{k} 0$ for all $i$.
The properties $3),4), 6) $ and $7)$ now imply that 
\[
\begin{split} 
\| [F\ell^{-1},u_{k}]\|&=\|F [\ell^{-1},u_{k}]+[F,u_{k}]\ell^{-1}\| 
 =\|\sum_{n=0}^{\infty}F\varepsilon^{-n}[d_{n},u_{k}] +\varepsilon^{-n}[F,u_{k}]d_{n}\| \\
&\leq \| \sum_{n\leq k}F\varepsilon^{-n}[d_{n},u_{k}] +\varepsilon^{-n}[F,u_{k}]d_{n}\| 
 +\|  \sum_{n>k}F\varepsilon^{-n}[d_{n},u_{k}] +\varepsilon^{-n}[F,u_{k}]d_{n}\| .
 \end{split}
 \]
 Now applying properties 4) and 6) to the first term and properties 3) and 7) to the second we estimate
 \[
 \begin{split}
\| [F\ell^{-1},u_{k}]\|&\leq \sum_{n\leq k}2\varepsilon^{2k-n} 
+ \sum_{k>n} 2\varepsilon^{n} 
=\varepsilon^{k}(\sum_{n\leq k}2\varepsilon^{n}+\sum_{n=1}^{\infty}2\varepsilon^{n})
\leq C\varepsilon^{k},
\end{split}
\]
and thus $\lim_{k\to\infty}[S,u_{k}]\rightarrow 0.$ 
Lastly, observe that by 8) we have $\|u_{k}a_{i}-a_{i}\|<\e^{2k}$ 
whenever $k>i$. Then for $k>i$ we can estimate
\[
\begin{split} 
\| [ F\ell^{-1},u_{k}a_{i}-a_{i} ] \|&\leq \| [F,u_{k}a_{i}-a_{i}]\ell^{-1}\|+\|F[\ell^{-1},u_{k}a_{i}-a_{i}]\| \\ 
&\leq \sum \e^{-n}\|F[d_{n},u_{k}a_{i}-a_{i}]\|+\e^{-n}\|[F,u_{k}a_{i}-a_{i}]d_{n}\| \\
&\leq \sum_{n\leq k} \e^{-n}\| F[d_{n},u_{k}a_{i}-a_{i}]\|+ \e^{-n}\|[F,u_{k}a_{i}-a_{i}]d_{n}\| \\ 
&\qquad + \sum_{n>k} \e^{-n}\|F[d_{n},u_{k}a_{i}-a_{i}]\|+\e^{-n}\|[F,u_{k}a_{i}-a_{i}]d_{n}\|\\
&\leq \sum_{n\leq k}C\e^{2k-n}+\sum_{n>k} C\e^{n+k}\quad\qquad{\textnormal{by 8), 9) and 10)}}\\
&\leq 2C\e^{k},
\end{split}
\]
and thus $[S,u_{k}a_{i}-a_{i}]\xrightarrow{k} 0$ for all $i$. 
\end{proof}
In the case of Fredholm modules, the absence of technicalities 
related to complementability of submodules allow for a better 
version of the above proposition. For an $A$-Fredholm module $(A,\H, F)$, we write $p$ for 
the projection onto $[\pi(A)\H]$, also note that $[F,p]\in J$ since
\[
[F,p]\pi(a)=[F,p\pi(a)]-p[F,\pi(a)]=[F,\pi(a)]-p[F,\pi(a)].
\]
\begin{lemma}
\label{h} 
Let $A$ be a separable $C^{*}$-algebra, $(A,\H,F)$ be a 
Fredholm module with $F=F^*$ and $F^2\leq 1$. Let $p$ be the 
projection onto $[\pi(A)\H]$. There exists $h\in J$ such that:\\
1) $h$ has dense range;\\
2) $[pFp,h]=0$.
\end{lemma}
\begin{proof} 
Since the module $(\H,pFp)$ is a compact perturbation of 
$(\H,F)$, and $1-p\in J$, it suffices to construct $h\in\bB(p\H)$, since 
$h+(1-p)\in\bB(\H)$ then has the desired properties. 
Thus, we replace $\H$ with $p\H=[\pi(A)\H]$ and $F$ with $pFp$, 
so we may assume $p=1$.  The operator 
$1-F^{2}$ is in $J$ and $\H$ splits as 
\[ 
\H=[\textnormal{Im}(1-F^{2})]\oplus\ker(1-F^{2}).
\]
The operator $F$ respects this decomposition since $F$ is self-adjoint. Choose a 
strictly positive $k\in \kK(\ker(1-F^{2}))$, and consider 
$k+FkF:\ker(1-F^{2})\rightarrow \ker(1-F^{2})$. This 
element commutes with $F|_{\ker(1-F^{2})}$ since $F^{2}=1$ on this subspace. Now define
\[
h:=(1-F^{2})+k+FkF\in J,
\]
which commutes with $F$. Moreover $h\geq 0$ since $1-F^{2}\geq 0$ and 
$h$ has dense range because it is an orthogonal sum of elements with 
dense range in their respective subspaces. 
\end{proof}
Using this particular $h$, one can lift the Fredholm module 
directly to a self-adjoint element, and one does not need to assume 
that the representation is essential.
\begin{defn}
\label{compatible}
Let $A,\,B,\,C$ be separable $C^{*}$-algebras 
$(A, E_{B},F)$ an essential $(A,B)$ Kasparov module 
and  $(\B, F_{C},T)$ an essential unbounded Kasparov module
such that $\B$ has bounded approximate unit.

Suppose we are given a complete projective $\B$-submodule 
$\mathcal{E}_{\B}\subset E_{B}$ with 
Grassmann connection 
$\nabla:\mathcal{E}^{\nabla}_\B\rightarrow E\hotimes_{B}\Omega^{1}_{T}$ 
so that  $1\otimes_{\nabla}T$ is the associated self-adjoint 
regular operator on $E\hotimes_{B}F_C$.
The pair $(\mathcal{E}_{\B},\nabla)$  is \emph{compatible} with $(A,E_B,F)$ if \\ 
1) $F\otimes 1, (1-F^{2})^{\frac{1}{2}}\otimes 1 \in \Lip(1\otimes_{\nabla}T)$; \\
2) there are bounded total subsets $\{a_{i}\}\subset A$ and $\{c_{j}\}\subset J_{F}$ 
such that for all $i,j$, the elements 
$a_{i},c_{j}\in\Lip(1\otimes_{\nabla}T)$; \\
3) there is an approximate unit $(u_{k})\subset A$ for which 
\[\lim_{k\to \infty}\|[1\otimes_{\nabla}T,u_{k}]\|\rightarrow 0,\quad \lim_{k\rightarrow\infty}\|[F,u_{k}]\|\rightarrow 0,\textnormal{  
and for all } i, \lim_{k\rightarrow\infty}\|[1\otimes_{\nabla}T, u_{k}a_{i}-a_{i}]\|=0;\]
4) there is an 
approximate unit $(w_{n})$ for $J_{F}$ such that 
$\lim_{n\rightarrow \infty}\|[1\otimes_{\nabla}T, w_{n}]\|=0$ 
and for all $j$, $\|[1\otimes_{\nabla}T, w_{n}c_{j}-c_{j}]\|\rightarrow 0$.
\end{defn}

The proof of the next result brings together our various technical innovations. 
First, the characterisation of our strongest form of completeness
from Theorem \ref{thm:equivs} is present to ensure 
that the operator $[s,t](s\pm i)^{-1}$ is bounded. Our version of Kasparov's 
technical theorem, Theorem \ref{thm: technical}, is used only for $C^{*}$-algebras, 
but we use quasicentrality for a differentiable algebra $\J$ in precisely one place,  
to ensure that $[F,\ell^{-1}]$ is not just bounded, 
but even in $\Lip(1\otimes_{\nabla}T)$. 
\begin{thm}
\label{Liplift}
Let $A,\,B$ be separable $C^{*}$-algebras, $\mathcal{B}$ a differentiable
algebra of an unbounded Kasparov module $(\B, F_{C},T)$ and $(A,E_B,F)$ a bounded 
$(A,B)$ Kasparov 
module. If $E_B$ admits a compatible 
complete projective submodule $\mathcal{E}_{\B}\subset E_{B}$, 
then $(A, E_{B},F)$ can be lifted to a correspondence $(\A,\mathcal{E}_\B,S,\nabla)$ 
for $(\B, F_{C},T)$.
\end{thm}
\begin{proof} 
Since we are given a compatible complete projective 
submodule $\mathcal{E}_{\B}\subset E_{B}$ we use the 
notation and data of Definition \ref{compatible}.
By Proposition \ref{repend} the elements $a_{i}$ and 
$c_{j}$ are in $\End^{*}_{\B}(\E^{\nabla})$.
Denote by $\A$ the closed subalgebra of $\End^{*}_{\B}(\E^{\nabla})$ 
generated by the $a_{i}$, by $\J$ the closed subalgebra of 
$\End^{*}_{\mathcal{B}}(\mathcal{E}^{\nabla})$ 
generated by the $c_{j}$, and let 
\[
\nabla:\mathcal{E}^{\nabla}_\B\rightarrow E\hotimes_{B}\Omega^{1}_{T},
\] 
be the Grassmann connection. 

We proceed in two steps. By hypothesis, 
there is a countable 
approximate unit  $(w_{n})$ 
 for $\J$ satsifying condition 
 4) of Definition \ref{compatible}. Also $F\J$, $\J F\subset \J$. By Theorem \ref{Ped} 
there is an $F$ quasi-central approximate unit $v_{n}'$ for $\J$ 
contained in the convex hull of $w_{n}$, and we set $d'_n=v_{n+1}'-v_n'$, and we can 
assume $\|[F,d'_{n}]\|_{\J}=\|[F,d'_{n}]\|_{1\otimes_{\nabla}T}<\e^{2n}$, cf. Theorem \ref{thm: technical}.
Since $(v_n')=\sum_{i=k_{n}}^{m_{n}}\theta_{i}w_{i}$ 
is built from far out convex combinations of $w_{n}$, 
this approximate unit will in particular again satisfy 
\[ 
[1\otimes_{\nabla} T, v_{n}']=\sum_{i=k_{n}}^{m_{n}}\theta_{i}[1\otimes_{\nabla}T,w_{i}]\rightarrow 0,
\] 
in norm. 
We may without loss of generality assume that 
$\|[1\otimes_{\nabla}T,v_{n}']\|<\e^{2n}$. Thus we have
\begin{equation}
\label{primaryestimates} 
\|[F,v_{n+1}'-v_{n}']\|_{1\otimes_{\nabla}T}<\e^{2n},\quad \|[1\otimes_{\nabla}T,v_{n}']\|<\e^{2n}.
\end{equation}

In particular $(v_{n}')$ is an approximate unit for the 
$C^{*}$-algebra $J_{F}$. By applying Theorem \ref{thm: technical} with 
$\B=\End^{*}_{B}(E)$, $\J=J_{F}$, $\A=A$, and $\K=\kK(E_B)$ we obtain the approximate 
unit $(v_{n})$ for $J_{F}$ with the properties $1)$-$10)$ of Theorem \ref{thm: technical}. 

It is important to notice that this can be achieved without 
losing the properties \eqref{primaryestimates}, because 
these are stable under far out convex combinations. 
Thus, by Theorem \ref{thm:equivs} 
we obtain the positive unbounded multiplier
\[
\ell^{-1}=c:=\sum \varepsilon^{-n}d_{n}, \qquad d_n=v_{n+1}-v_n
\]
for the algebra $\J$, with the property that 
$[1\otimes_{\nabla}T,\ell^{-1}]$ is bounded. Because the 
$(v_{n})$, as an approximate unit for the $C^{*}$-algebra 
$J_{F}$, have properties 1)-10) from Theorem \ref{thm: technical}, 
the argument from the proof of Proposition \ref{prop: Kasparovlift} 
can be repeated to see that the self-adjoint lift $S:=\frac{1}{2}(F\ell^{-1}+\ell^{-1}F)$ 
makes $(\A,E_{B},S)$ into an unbounded Kasparov module such that
$\A$ has an approximate unit with $[S,u_n]\to 0$ in norm. 
So 1) of Definition \ref{correspondence} is 
satisfied, and 2) is satisfied by the assumption that $\E_{\B}$ is 
complete and 3) by assumptions 2) and 3) of Definition \ref{compatible}. 
We now turn to proving 4) of Definition \ref{correspondence}.

It follows from the properties \eqref{primaryestimates} that 
$[F,\ell^{-1}]\in \Lip(1\otimes_{\nabla}T)$. To see this, observe that the finite sums 
$\sum_{n=0}^{k}\e^{-n}[F,d_{n}]$ preserve the domain of 
$1\otimes_{\nabla}T$, and using the properties \eqref{primaryestimates} we find that
\[ 
\|[1\otimes_{\nabla}T,[F,\ell^{-1}]]\|
=\|\sum_{n}\e^{-n}[1\otimes_{\nabla}T,[F,d_{n}]]\|\leq 
\sum_{n}\e^{-n}\|[1\otimes_{\nabla}T,[F,d_{n}]]\|\leq \sum \e^{n}<\infty.
\]
Now compute on $\Dom 1\otimes_{\nabla}T$
\[
\begin{split} 
[1\otimes_{\nabla}T, S](S\pm i)^{-1} &=[1\otimes_{\nabla}T,S](S\pm i)^{-1} 
=[1\otimes_{\nabla}T, F\ell^{-1}-[F,\ell^{-1}]](S\pm i)^{-1}\\
&=[1\otimes_{\nabla}T, F\ell^{-1}](S\pm i)^{-1}-[1\otimes_{\nabla}T,[F,\ell^{-1}]](S\pm i)^{-1},
\end{split}
\]
by which it suffices to show that $[1\otimes_{\nabla}T,F\ell^{-1}](S\pm i)^{-1}$ 
is bounded. Note that we may assume that $\|[F,\ell^{-1}]\|<1$ 
and thus that $F\ell^{-1}\pm i:\im\ell\to E_B$ is surjective 
and has adjointable inverse $(F\ell^{-1}\pm i)^{-1}$. Then by the resolvent equation
\[
(S\pm i)^{-1}=(F\ell^{-1}-\frac{1}{2}[F,\ell^{-1}]\pm i)^{-1}
=(F\ell^{-1} \pm i)^{-1}+\frac{1}{2}(F\ell^{-1}\pm i)^{-1}[F,\ell^{-1}](S\pm i)^{-1},
\]
and it suffices to show that $[1\otimes_{\nabla}T, F\ell^{-1}](F\ell^{-1}\pm i)^{-1}$ 
is bounded on $\Dom 1\otimes_{\nabla} T$. Then
\[
\begin{split}
[1\otimes_{\nabla}T, F\ell^{-1}](F\ell^{-1}\pm i)^{-1} 
& =[1\otimes_{\nabla}T, F]\ell^{-1}(F\ell^{-1}+i)^{-1}+F[1\otimes_{\nabla}T,\ell^{-1}](F\ell^{-1}\pm i)^{-1}\\
&=[1\otimes_{\nabla}T, F](F\pm i\ell)^{-1} + F[1\otimes_{\nabla}T,\ell^{-1}](F\ell^{-1}\pm i)^{-1},
\end{split}
\]
which is bounded on $\Dom 1\otimes_{\nabla}T$.
Therefore $(\A,\mathcal{E}_\B,S,\nabla)$ has the required properties.
\end{proof}
\subsection{Lifting Kasparov modules to correspondences}
\label{subsec:lift}
In view of Proposition \ref{Liplift}, in order to lift a pair of Kasparov modules 
$(A,E_{B},G_{1})$ and $(B,F_{C}, G_{2})$ so that we can 
construct their Kasparov product, we need to find an unbounded 
representative $(\B,F_{C}, T)$ such that $(A,E_{B},G_{1})$ admits a 
compatible pair $(\E_\B,\nabla)$ of a complete projective submodule and connection 
in the sense of Definition \ref{compatible}. We begin with a Lemma about a special subalgebra
of the multipliers of the linking algebra. Recall that the linking algebra $\mathcal{L}(E_B)$
of the $C^*$-module $E_B$ is the algebra of compact endomorphisms of $E_B\oplus B$. 
\begin{lemma}
\label{linkingA} 
Let $E_{B}$ be a $C^{*}$-module and $A\subset\End^{*}_{B}(E)$ 
an essential $C^{*}$-subalgebra with self-adjoint contractive approximate unit 
$(u_{n}^{A})$, and let $(u_n^B)$ be an approximate unit for $B$. Then the collection
of operators
\begin{equation}
\label{linking}
\mathcal{L}_{A}(E_{B}):=\left\{\begin{pmatrix} a +K & |e\rangle \\ \langle f | & b\end{pmatrix}: 
a\in A,\, K\in\kK(E_{B}),\  b\in B,\, e,\,f\in E_{B}\right\}\subset\End^{*}_{B}(E\oplus B),
\end{equation}
is a $C^{*}$-subalgebra containing the linking algebra $\mathcal{L}(E_{B})$ as an ideal and 
$u_{n}:=\big(\begin{smallmatrix} u_{n}^{A} & 0 \\ 0 & u_{n}^{B}\end{smallmatrix}\big)$ 
is an approximate unit for $\mathcal{L}_{A}(E_{B})$. 
The algebra $\mathcal{L}_{A^{+}}(E_{B^{+}})$ coincides 
with the unitisation $\mathcal{L}_{A}(E_{B})^{+}$.
\end{lemma}
\begin{proof} 
A quick computation shows that $\mathcal{L}_{A}(E_{B})$ $*$-algebra. 
Since $\kK(E_{B})$ is an ideal in $\End_{B}^{*}(E)$ it follows that 
$A+\kK(E_{B})$ is a $C^{*}$-algebra. Moreover, because 
$A$ is essential, $[AE_{B}]=E_{B}$, it follows that 
$u_{n}^{A}\langle e|=\langle u_{n}^{A}e| \to \langle e|$ and 
$|e\rangle u_{n}^{A}=|u_{n}^{A}e\rangle\to |e\rangle$ for all 
$e\in E_{B}$. Thus $u_{n}^{A}K\to K$ for all $K\in\kK(E_{B})$, 
and $u_{n}^{A}$ is approximate unit for $A+\kK(E_{B})$. 
Using the corresponding properties for $u_{n}^{B}$,  
the same argument shows that 
$\big(\begin{smallmatrix} u_{n}^{A} & 0 \\ 0 & u_{n}^{B}\end{smallmatrix}\big)$ 
is an approximate unit for $\mathcal{L}_{A}(E)$. The 
statement on the unitisation is immediate.
\end{proof}
We now recall a definition from \cite{CS}, where the notion of 
connection was introduced in the bounded picture of $KK$-theory.
\begin{defn}
\label{bddcon} 
Let $(B,F_{C}, G)$ be a $(B,C)$ 
Kasparov module and $E_{B}$ a $C^{*}$-module over $B$. 
An operator $G\in\End_{C}^{*}(E\otimes_{B}F_{C})$ is a 
$G_{2}$-\emph{connection} if for each $e\in E$ the operator 
\[
\left[\begin{pmatrix} G & 0 \\ 0 & G_{2}\end{pmatrix},
\begin{pmatrix} 0 & |e\rangle \\ \langle e|  & 0\end{pmatrix}\right] 
\in \End^{*}_{C}(E\hotimes_{B}F_{C}\oplus F_{C}),
\]
is compact.
\end{defn}
It is well known that $G_2$-connections always exist, by realising 
$E_B$ as a complemented submodule of $\mathcal{H}_{B^{+}}$ via 
$v:E_B\rightarrow \H_{B^{+}}$ and defining $G:=v^{*}\epsilon\diag(G_2)v=:v^*G_{2,\epsilon}v$. 
It is useful to observe that in $\End^{*}_{C}(E\hotimes_{B}F\oplus F)$ we have the identity
\[
\begin{pmatrix} G & 0 \\ 0& 0\end{pmatrix}
=\sum_{i\in\Z} \begin{pmatrix} 0 & | \gamma(x_{i})\rangle \\ 0 & 0\end{pmatrix}
\begin{pmatrix} 0 & 0 \\ 0& G_{2} \end{pmatrix}\begin{pmatrix} 0 & 0 \\ \langle x_{i} | & 0 \end{pmatrix},
\]
which can also be written as a matrix product
\[
\begin{pmatrix} 0 & | x_{i}\rangle \\ 0 & 0\end{pmatrix}_{i\in\Z}^{t}\cdot
\left(\epsilon\diag\begin{pmatrix} 0 & 0 \\ 0& G_{2} \end{pmatrix}\right)
\cdot\begin{pmatrix} 0 & 0 \\ \langle x_{i} | & 0 \end{pmatrix}_{i\in\Z},
\]
of a row, a diagonal matrix and a column.
\begin{lemma}
\label{lem:connectionquasi} 
Let $A\subset\End^{*}_{B}(E_{B})$ be an essential subalgebra 
with approximate unit $(\tilde{u}_{n})$ and $(B, F_{C}, G_{2})$ a 
Kasparov module. Let $(x_{i})_{i\in\Z}$ be a homogenous frame 
for $E_{B}$ with  stabilisation isometry $v:E_{B}\to\H_{B^{+}}$ 
and associated $G_{2}$-connection $G=v^{*}\epsilon\diag(G_2)v$. 
Then there is a $G$-quasicentral approximate unit $(u_{n})$ for 
$A$ contained in the convex hull $\mathscr{C}(\tilde{u}_{n})$: that is $[G,u_{n}]\to 0$ in norm.
\end{lemma}
\begin{proof} 
This will follow from a direct application of Theorem \ref{Ped}. 
Consider the algebra $\mathcal{L}_{A^{+}}(E_{B^{+}})$ and the ideal 
$\mathcal{L}_{A}(E_{B^{+}})$ as defined in \eqref{linking}, 
with its approximate unit $u_{n}'=\begin{pmatrix} \tilde{u}_{n} & 0 \\ 0 & 1\end{pmatrix}$. 
We view the row, respectively column,
\[
|x\rangle:=\begin{pmatrix} 0 & | x_{i}\rangle \\ 0 & 0\end{pmatrix}_{i\in\Z}^{t},\quad \langle x|
:=\begin{pmatrix} 0 & 0 \\ \langle x_{i} | & 0 \end{pmatrix}_{i\in\Z},
\]
as elements in $\End^{*}_{\mathcal{L}_{A^{+}}(E_{B^+})}(\H_{\mathcal{L}_{A^{+}}(E_{B^+})})$. 
By Theorem \ref{Ped}, the convex hull $\mathscr{C}(\tilde{u}_{n})$ contains an 
$\langle x|, |x\rangle$-quasicentral approximate unit 
$u_{n}''=\begin{pmatrix} u_{n} &0 \\ 0 & 1\end{pmatrix}$. 
Observe that 
\[
\left[ u_{n}''\cdot\textnormal{Id}_{\H_{\mathcal{L}_{A^{+}}}(E_{B^+})},\epsilon\cdot\textnormal{diag}
\begin{pmatrix} 0 & 0 \\0 & G_{2}\end{pmatrix}\right] =0.
\]
Writing $u_{n}''$ for $u_{n}''\cdot\textnormal{Id}_{\H_{\mathcal{L}_{A^{+}}}(E_{B^+})}$ 
when necessary, we can compute 
\begin{align*} 
\begin{pmatrix} [G,u_{n}] & 0 \\ 0 & 0 \end{pmatrix} 
&=\left[\begin{pmatrix} G & 0\\ 0 & 0\end{pmatrix} {}_, 
\begin{pmatrix} u_{n} & 0 \\ 0 & 1 \end{pmatrix}\right] \\
&=\left[\begin{pmatrix} 0 & |x\rangle \\ 0 & 0 \end{pmatrix}
\cdot \left(\epsilon\cdot\diag\begin{pmatrix} 0 & 0 \\ 0 & G_{2}\end{pmatrix}\right)
\cdot\begin{pmatrix} 0 & 0 \\ \langle x| & 0\end{pmatrix}{}_,  
\begin{pmatrix} u_{n} & 0 \\ 0 & 1 \end{pmatrix}\right] \\
&=\left[\begin{pmatrix} 0 & |x\rangle \\ 0 & 0 \end{pmatrix}
{}_,  
\begin{pmatrix} u_{n} & 0 \\ 0 & 1 \end{pmatrix}\right] 
 \left(\epsilon\cdot\diag\begin{pmatrix} 0 & 0 \\ 0 & G_{2}\end{pmatrix}\right)\cdot
 \begin{pmatrix} 0 & 0 \\ \langle x| & 0\end{pmatrix}
\\ & \quad\quad\quad\quad\quad\quad\quad +\begin{pmatrix} 0 & |x\rangle \\ 0 & 0 \end{pmatrix}
\cdot\left(\epsilon\cdot\diag\begin{pmatrix} 0 & 0 \\ 0 & G_{2}\end{pmatrix}\right)
\cdot \left[\begin{pmatrix} 0 & 0 \\ \langle x| & 0\end{pmatrix}{}_,  
\begin{pmatrix} u_{n} & 0 \\ 0 & 1 \end{pmatrix}\right]
\end{align*}
which converges to $0$ in norm by Theorem \ref{Ped}. 
Therefore $[G,u_{n}]$ converges to $0$ in norm.
\end{proof}
Given a $*$-homomorphism $B\to \End^{*}_{C}(F)$, the algebra 
$\End^{*}_{B}(E_{B}\oplus B)$ is naturally represented on the 
$C^{*}$-module $E\hotimes_{B}F\oplus F$. In particular, the 
linking-type algebras defined in Lemma \ref{linkingA} act on 
$E\hotimes_{B}F\oplus F$. We prepare the setting for the 
proof of our final lifting result, which is yet another application of Theorem \ref{thm: technical}.
Recall that given a bounded Kasparov module $(A,E_B,G_1)$, we define $J_{G_1}$ 
to be the $C^*$-algebra generated by $\kK(E_B)$ and ${\rm Id}_{E_B}-G_1^2$.
\begin{lemma}
\label{preparetoapplytechnical} 
Let $(A, E_{B}, G_{1})$ and $(B,F_{C},G_{2})$ be Kasparov modules 
and $(x_{i})_{i\in\Z}$ a homogenous frame for $E_{B}$ with  
stabilisation isometry $v:E_{B}\to\H_{B^{+}}$ and associated 
$G_{2}$-connection $G=v^{*}G_{2,\epsilon} v$. Write 
$\tilde{G}:=\big(\begin{smallmatrix}G & 0\\ 0 & G_{2}\end{smallmatrix}\big)$ 
and $\tilde{G}_{1}:=\big(\begin{smallmatrix} G_{1}\otimes 1 & 0\\ 0 &  0\end{smallmatrix}\big)$. 
As in Lemma \ref{linkingA} consider the algebras 
\[
A_{0}:=\mathcal{L}_{A}(E_{B}),\quad A_{1}:=\mathcal{L}_{J_{G_{1}}}(E_{B}).
\]
For $p=0,1$ define 
the $C^{*}$-algebras $\tilde{J}_{p}$ generated by $\kK(E_{B})\otimes 1\oplus \kK(F_{C})$,
$[\tilde{G},A_{p}], [\tilde{G}, \tilde{G}_{1}]$ and $1-\tilde{G}^{2}$ 
on $E\hotimes F_C\oplus F_C$. 
Let $\tilde{B}_{p}$ be the $C^{*}$-algebras generated by  
$\textnormal{Id}_{E\hotimes_{B}F\oplus F},$ $A_{p}$ and 
$\tilde{J}_{p}$. Let $K_{p}$ be the $C^{*}$-subalgebra of 
$\tilde{B}_{p}$ generated by $A_{p}\tilde{J}_{p}$ and $\tilde{J}_{p}A_{p}$.
Finally, define $J_p=K_p+\tilde{J}_p$. Then:\\
1) $A_{p}$ admits $\tilde{G}, G_{1}\otimes 1$-quasicentral 
approximate units $u_{n}^{p}$ of the form 
\[
(u_{n}^{0})=\begin{pmatrix}\tilde{u}_{n}^{A} & 0 \\ 0 & \tilde{u}_{n}^{B}\end{pmatrix}, \quad (u_{n}^{1})
=\begin{pmatrix}\tilde{w}_{n}^{1} & 0 \\ 0& \tilde{u}_{n}^{B}\end{pmatrix},
\] 
where $(\tilde{u}_{n}^{A})$, $(\tilde{u}_{n}^{B})$, $(\tilde{w}_{n}^{1})$ 
are approximate units for $A,B$ and $J_{G_{1}}$ respectively;\\
2) $J_{p}$ admits an approximate unit $(v^{p}_{n})=\big(\begin{smallmatrix}\tilde{v}_{n}^{p} & 0 \\ 0 & \tilde{w}_{n}^{2}\end{smallmatrix}\big)$ 
where $(\tilde{v}^{p}_{n})$ is an approximate unit for the algebra generated by 
$[G,A_{p}]$, $[G,G_{1}\otimes 1]$, $1-G^{2}$ and $\kK(E_{B})\otimes 1$ and  
$(\tilde{w}^{2}_{n})$ is an approximate unit for $J_{G_{2}}$;\\
3) $K_{p}$ is an ideal in $\tilde{B}_{p}$;\\
4) $A_{p}K_{p}=K_{p}A_{p}=K_{p}$, that is $K_{p}$ is $A_{p}$-essential;\\
5) $A_{p}J_{p}$, $J_{p}A_{p}\subset K_{p}\subset J_{p}$;\\
6) $[G,A_{p}]\subset K_{p}$.
\end{lemma}
\begin{proof}  
To prove $1)$, we show that both $A_{0}$ and $A_{1}$ 
admit approximate units $(u_{n}^{0})$, $(u_{n}^{1})$ that are quasicentral for 
$\tilde{G}$ and $\tilde{G}_{1}\otimes 1$.  
Since both $A$ and $J_{G_{1}}$ are essential on $E_{B}$, 
by Lemma \ref{lem:connectionquasi} there exist approximate 
units $(\tilde{u}_{n}^{A})$, $(\tilde{w}_{n}^{1})$ for $A$ and 
$J_{G_{1}}$, respectively,  that are quasicentral for $G$. 
Since $[G_{1},A]$, $[G_{1},J_{G_{1}}]\subset \kK(E_{B})$, the approximate units $(\tilde{u}_{n}^{A})$, 
$(\tilde{w}^{1}_{n})$ can be chosen $G_{1}$-quasicentral as well by Lemma~\ref{strictconvex}. 
For the same reason $B$ admits a $G_{2}$ quasicentral approximate unit 
$(\tilde{u}_{n}^{B})$. By Lemma \ref{linkingA} setting
\[
(u_{n}^{0}):=\begin{pmatrix} \tilde{u}_{n}^{A} & 0 \\ 0 & \tilde{u}_{n}^{B}\end{pmatrix}, 
\quad (u_{n}^{1}):=\begin{pmatrix} \tilde{w}_{n}^{1} & 0 \\ 0 & \tilde{u}_{n}^{B}\end{pmatrix},
\]
yields $\tilde{G}$, $G_{1}\otimes 1$-quasicentral approximate units for $A_{0}$ and $A_{1}$.

To prove $2)$ we first show that any $\tilde{G}$-quasicentral approximate unit $(v^{p}_{n})$ for $\tilde{J}_{p}$ is an approximate unit for $J_{p}$. Since $\tilde{J}_{p}$ is essential (it contains $\mathcal{L}(E_{B})$ ), existence of such $(v^{p}_{n})$ is guaranteed by Lemma \ref{lem:connectionquasi}. It is clear that $(v^{p}_{n})$ is an approximate unit for $\tilde{J}_{p}$ and $\tilde{J}_{p}A_{p}$. For $A_{p}\tilde{J}_{p}$ it suffices to show that
\[v_{n}^{p}a[\tilde{G},b]=v_{n}^{p}[\tilde{G},ab]-v_{n}^{p}[\tilde{G},a]b\to 0,\]
\[v_{n}^{p}a[\tilde{G},\tilde{G}_{1}]=[\tilde{G}, v_{n}^{p}a\tilde{G}_{1}]-[\tilde{G},v^{p}_{n}]a\tilde{G}_{1}-v^{p}_{n}[\tilde{G},a]\tilde{G}_{1}\to 0,\]
\[v_{n}^{p}a(1-\tilde{G}^{2})=v_{n}^{p}(1-\tilde{G}^{2})a+[\tilde{G},v^{p}_{n}][\tilde{G},a]-\tilde{G}v_{n}[\tilde{G},a]-v^{p}_{n}[\tilde{G},a]\tilde{G}\to 0,\]
which all follow from $\tilde{G}$-quasicentrality of $(v^{p}_{n})$.

 Now we proceed to the statement of 2). Consider the subalgebra of 
$\End^{*}_{C}(E\hotimes F)$ generated by 
$[G, A_{p}]$, $[G, G_{1}\otimes 1]$, $1-G^{2}$ 
and $\kK(E_{B})\otimes 1$, and choose a $G$-quasicentral approximate unit 
$(\tilde{v}_{n}^{p})$ for this algebra, as well as a $G_{2}$-quasicentral  approximate unit 
$(\tilde{w}^{2}_{n})$ for the algebra $J_{G_{2}}$. 
Then 
$(v^{p}_{n}):=\diag(\tilde{v}_{n}^{p},\tilde{w}^{2}_{n})$ is clearly an 
approximate unit for the diagonal entries of $J_{p}$. For the off-diagonal entries, 
it suffices to show that for all $e\in E_{B}$, $(v^{p}_{n})$ is a two-sided 
approximate unit for the operators
\begin{equation}
\label{conncomm}
\left[\begin{pmatrix} G &0 \\ 0 & G_{2}\end{pmatrix}{}_, 
\begin{pmatrix} 0 & |e\rangle \\ 0 & 0 \end{pmatrix}\right]=
\sum_{i\in\Z}\begin{pmatrix} 0 & |x_{i}\rangle [G_{2},\langle x_{i}, e\rangle ] \\ 0& 0\end{pmatrix}.
\end{equation}
Since $e\in E_{B}$, the series 
$\sum_{i\in\Z} [G_{2},\langle x_{i},e \rangle]^{*}[G_{2}\langle x_{i},e\rangle]$ 
is norm convergent in $\kK(F_C)\subset J_{G_2}$. Therefore $(v^{p}_{n})$ is a right 
approximate unit for the operator in Equation \eqref{conncomm}. 
On the other hand, since the series $\sum_{i\in\Z} |x_{i}\rangle [G_{2},\langle x_{i}, e\rangle ] $ 
is norm convergent and $(\tilde{v}_{n}^{p})$ is an approximate for 
$\kK(E_{B})\otimes 1$, it follows that $(v_{n}^{p})$ is a left approximate unit for 
operators of the form \eqref{conncomm}.

For $3)$ it is sufficient to observe that $A_{p}^2$ is dense in $A_p$ and $\tilde{J}_{p}^2$ is dense in $\tilde{J}_p$. 
Thus the subalgebra generated by $A_{p}\tilde{J}_{p}$, $\tilde{J}_{p}A_{p}$ is 
indeed a two-sided ideal in $\tilde{B}_{p}$.

For $4)$, to show that $A_{p}K_{p}=K_{p}A_{p}=K_{p}$ it suffices to show that 
$\tilde{J}_{p}A_{p}\subset A_{p}K_{p}$ and $A_{p}\tilde{J}_{p}\subset K_{p}A_{p}$. 
To show that $\tilde{J}_{p}A_{p}\subset A_{p}K_{p}$, it suffices to show that for all 
$a,b\in A_{p}$ we have 
\[
u_{n}^{p}[\tilde{G},a]b\to [\tilde{G},a]b, \quad u_{n}^{p}[\tilde{G},G_{1}\otimes 1]a
\to [\tilde{G},G_{1}\otimes 1]a,\quad u_{n}^{p}(1-\tilde{G}^{2})a\to (1-\tilde{G}^{2})a.
\] 
These limits all follow directly from $\tilde{G},G_{1}\otimes 1$ quasicentrality of $(u_{n}^{p})$. 

Property $5)$ is immediate from the definition of $K_{p}$ and $J_{p}$. 

Property $6)$ follows from the convergence 
$u_{n}^{p}[G,a]=[G,u_{n}^{p}a]-[G,u_{n}^{p}]a\to [G,a]$ for all 
$a\in A_{p}$, since $u_{n}^{p}[G,a]\in A_{p}J_{p}\subset K_{p}$.
\end{proof}

\begin{prop}
\label{singleliplift} 
Let $A,B,C$ be separable $C^{*}$-algebras, and let $(A,E_B,G_{1})$, 
$(B,F_{C},G_{2})$ be essential Kasparov modules.
Then $(B,F_{C},G_{2})$ can be lifted to an unbounded Kasparov module $(\B,F_{C},T)$ such
that $\B$ has an approximate unit $(u_n)$ with $[T,u_n]\to 0$ in norm, and moreover
$E_{B}$ admits a compatible complete projective 
$\B$-submodule $\mathcal{E}_{\B}\subset E_{B}$ in the sense of Definition \ref{compatible}.
\end{prop}
\begin{proof} 
To prove the theorem, we have to lift $G_{2}$ to an unbounded 
representative $T$ with resolvent in $J_{G_{2}}$, such that $B$ 
admits a differentiable subalgebra with an approximate unit $(u_{n}^{B})$ 
such that $[T,u_{n}^{B}]\to 0$ in norm. Moreover, the lift 
$T$ should also satisfy Definition \ref{compatible}. 
This means we have to provide a column finite frame 
$(x_{i})_{i\in\Z}$ for $E_{B}$ 
such that the resulting submodule $\E_{\B}$ is complete, 
as well as satisfying properties 1)-4). These properties 
imply that, for $\nabla$ the Grassmann connection 
associated to the frame $(x_{i})_{i\in\Z}$, the operators 
$G_{1}$ and $(1-G_{1}^{2})^{\frac{1}{2}}$ are elements of 
$\Lip(1\otimes_{\nabla}T)$  and the algebras $A$ and 
$J_{G_{1}}$ admit differentiable subalgebras $\A$ and 
$\mathcal{J}_{G_{1}}$, generated by $(1-G_1^2)$ 
and $\kK(E_B)$, with approximate units $(u^{\A}_{n})$ and $(w_{n}^{\mathcal{J}})$ 
such that $\lim_{n\to\infty} [1\otimes_{\nabla}T,u^{\A}_{n}]
=\lim_{n\to\infty}[1\otimes_{\nabla}T,w_{n}^{\mathcal{J}}]=0$. 
In order to achieve this, we once again employ linking-type algebras.

Choose a stabilisation isometry $v:E_B\to \H_{B^+}$.
Consider the frame $(x_{i})_{i\in\Z}:=(v^{*}e_{i})_{i\in\Z}$ 
and the corresponding frame approximate unit 
$(\chi_{n}):=\sum |x_{i}\rangle\langle x_{i}|$. By 
Lemma \ref{lem:connectionquasi} there exists an approxmate unit 
$(w_{n})\in\mathscr{C}(\chi_{n})$ with $[G,w_{n}]\to 0$ in norm. 
We proceed with the notation introduced in 
Lemma \ref{preparetoapplytechnical}, and apply Theorem \ref{thm: technical} with 
\begin{align*}
\B=\tilde{B}_{0}\oplus\tilde{B}_{1},\quad \J
&=J_{0}\oplus J_{1},\quad F=\tilde{G}
=\begin{pmatrix} G & 0 \\ 0 & G_{2,\epsilon}\end{pmatrix},\quad 
\A=A_{0}\oplus A_{1}, \quad \K:=K_{0}\oplus K_{1}, 
\end{align*}
which by Lemma \ref{preparetoapplytechnical} satisfy 
$\A\J$, $\J\A\subset \K$, $\A\K=\K\A=\K$ and the algebra $\B$ is unital. 
Since the approximate units $(u_{n}^{0})$ for $A_0$ and $(w_{n}^{})$ for $\kK(E_B)$ are 
$\tilde{G}$-quasicentral, 
and $(v_{n}^{p})$ are approximate units for $J_{p}$, we can apply 
Theorem \ref{thm: technical} by setting 
$(u_{n}'):=(u_{n}^{0})\oplus (u_{n}^{1})$, $(v_{n}')=(v_{n}^{0})\oplus (v_{n}^{1})$. 
In doing so we obtain approximate units $(u_{n})$, $(v_{n})$ satisfying 
properties $1)-10)$ from Theorem \ref{thm: technical}. 
In addition to these properties, with $d_n^p=v_{n+1}^p-v_n^p$, we may also assume that:\\
\\
11) $\|d_{n}^{1}[\tilde{G}, \diag(w_{k},0)]\|<\epsilon^{2k}$ for all $n$;\\
12) $\|d_{n}^{1}[\tilde{G}, \diag(w_{k},0)]\|<\epsilon^{2k}$ for $n\geq k$;\\
13) $\|[d^{1}_{n},\diag(0,\langle x_{i}, x_{k}\rangle )_{i\in\Z}]\|<\epsilon^{2n}$ for $n\geq k$;\\
14) $\|d_{n}^{1}[\tilde{G}, \diag (0, \langle x_{i}, x_{k}\rangle)_{i\in\Z}]\|<\epsilon^{2n}$ for $n\geq k$;\\
15) $\left\|\left[d^{p}_{n}, \langle x| \right]\right\|<\epsilon^{2n}$.\\
\\
Properties $11)$ and $12)$ can be achieved because $\kK(E_{B})\otimes 1 \oplus 0\subset J_{1}$, so this only requires a further convexity argument when running the proof of Theorem \ref{thm: technical}. Properties $13)$, $14)$ and $15)$ are a direct appliciation of Theorem \ref{Ped} by viewing the columns $\langle x |$ and $(\langle x_{i},x_{k}\rangle)_{i\in\Z}$ as elements in $\End^{*}_{\mathcal{L}(E_{B})^{+}}(H_{\mathcal{L}(E_{B})^{+}})$.

The $(v_{n}^{p})$ so obtained
from Theorem \ref{thm: technical} define two unbounded multipliers on $E\hotimes F_C\oplus F_C$
\[
h^{-1}_{p}=\sum_{n}c^{n}d_{n}^{p}=\begin{pmatrix}k^{-1}_{p} & 0 \\ 0 &  \ell^{-1}\end{pmatrix},\quad
v^p_n=\begin{pmatrix} \tilde{v}^p_n & 0\\ 0 & w^{2}_{n} \end{pmatrix}, 
\quad d_n^p=v^p_{n+1}-v^p_n,
\]
which we use to lift $\tilde{G}$ in two ways. That we obtain the same $\ell^{-1}$ for $p=0,\,1$
follows from the form of $(v^p_n)$, that is, the approximate unit $(w^{2}_{n})$ for $J_{G_{2}}$ occurs in both lower right corners.

From the specific form of the $(u_{n}^{p})$ and $(v_{n}^{p})$, 
cf. Lemma \ref{preparetoapplytechnical} 2) and 3), it follows that we obtain 
new approximate units $(u_{n}^{A})$ for $A$, $(w^{1}_{n})$ for 
$J_{G_{1}}$ and $(v^{p}_{n})$ for the algebras generated by 
$[G,A_{p}]$, $[G,G_{1}\otimes 1]$, $1-G^{2}$ and $\kK(E_{B})\otimes 1$. 

This allows us to define unbounded lifts $\tilde{T}_{0}:=G k_{0}^{-1}$ and 
$\tilde{T}_{1}:=G k_{1}^{-1}$. One proves, as in Proposition \ref{prop: Kasparovlift},  
that $A$ admits a differentiable subalgebra $\A$ for which $(u_{n}^{A})$ is an 
approximate unit with $\lim [\tilde{T}_{0},u_{n}^{A}]\to 0$ in norm. 
For $J_{G_{1}}$ the same statement holds with respect to 
$\tilde{T}_{1}$. Moreover, $11)$ and $12)$ ensure that 
$[\tilde{T}_{1}, w_{n}]\to 0$ in norm as well, again with the same proof as 
Theorem \ref{prop: Kasparovlift}. Furthermore properties $13)$ and $14)$ 
guarantee that the columns $[T,\langle x_{i},x_{k}\rangle]_{i\in\Z}$ 
are elements of $\H_{\End^{*}_{C}(F)}$. That is, the frame 
$(x_{i})$ is column finite for $T$. 

It must be noted that 
our method does not allow us to obtain a uniform 
bound on the norms of these columns, and thus 
we are not able to produce a projection operator in $\End^{*}_{\B}(\H_{\B})$. Later, we will see
 that we do obtain a complete projective module.

We now compare the connection operator $1\otimes_{\nabla} T$ of the frame $(x_i)_{i\in\Z}$
to the operators $\tilde{T}_{p}$, $p=0,1$, which we have used to lift the bounded connection $G$.
Condition 15) guarantees that $[h^{-1}_{p},\langle x| ]$ 
is a bounded operator. Since 
\begin{equation}
\label{perturbcolumn}
\left[h^{-1}_{p},\langle x| \right]
=\begin{pmatrix} 0 & 0\\ \langle x_{i}|k_{p}^{-1}-\ell^{-1}\langle x_{i} | &0 \end{pmatrix}_{i\in\Z},
\end{equation}
it follows that $v\,\im k_{p}\subset \im \ell$ and 
$v^{*}\,\im \ell\subset \im k_{p}$. 
We wish to show that the difference
$\tilde{T}_{p}-1\otimes_{\nabla}T$ 
is bounded. To this end we  compute
\begin{equation}\label{perturbation}
v^{*}G_{2,\epsilon}\ell^{-1}v- Gk_{p}^{-1}
=|x\rangle G_{2,\epsilon}\ell^{-1}\langle x| -| x\rangle G_{2,\epsilon}\langle x | k_{p}^{-1} \\
=|x\rangle G_{2,\epsilon}(\ell^{-1}\langle x| -\langle x| k_{p}^{-1}),
\end{equation}

which is bounded by construction. Note that this 
implies the self-adjointness of $1\otimes_{\nabla} T$, 
and also that $\tilde{T}_{0}$ and $\tilde{T}_{1}$ have the same domain. 

Since $[\tilde{T}_{1},u_{n}]\to 0$ in norm and $u_{n}\to 1$ strictly on 
$E\hotimes_{B} F$, it follows that $[1\otimes_{\nabla} T, u_{n}]\to 0$ strictly, 
and so is a bounded sequence. 
Hence 
$
p[T_{\epsilon}, vu_{n}v^{*}] p=v[v^{*}T_{\epsilon}v, u_{n}]v^{*}\to 0
$
on a dense subspace of $\H_{B^+}\hotimes F$, and by boundedness of the sequence,
strictly on all of $\H_{B^{+}}\hotimes F$. Hence the frame $(x_{i})$ 
defines a complete projective submodule $\E_{\B}\subset E_{B}$, which also (and independently) 
proves the self-adjointness of $1\otimes_\nabla T$. 

Lastly, we must show that there are approximate units 
$(u_{n}^{\A})\in\mathscr{C}(u_{n}^{A})$ for $\A$ and 
$(w_{n})^{\mathcal{J}}\in\mathscr{C}(w^{1}_{n})$ for $J_{G_{1}}$
that satisfy
\begin{equation}
\label{lastrequirement}
\lim [1\otimes_{\nabla}T, u_{n}^{\A}]
=\lim [1\otimes_{\nabla} T, w_{n}^{\mathcal{J}}]\to 0,
\end{equation}
in norm. Observe that we can obtain the strict convergences of 
Equation \eqref{lastrequirement}, for both $(u_{n}^{A})$ and 
$(w^{1}_{n})$ converge strictly to $1$ on $E\hotimes_{B}F$ and the lifts 
$T_{0}$ and $T_{1}$ are bounded perturbations of $1\otimes_{\nabla}T$.

The column in Equation \eqref{perturbcolumn} is an element of 
$\End^{*}_{\tilde{B}_{p}}(\H_{\tilde{B}_{p}}, J_{p})$, because for each $i$ we have $\left[h^{-1}_{p},\langle x_{i}| \right]\in J_{p}$. Thus by Theorem \ref{Ped2} 
there exist $[h^{-1}_{p}, \langle x| ]$-quasicentral approximate 
units in the convex hull $\mathscr{C}(u_{n}^{p})$, which at the 
same time remain $\tilde{G}$-quasicentral. The resulting 
approximate units will again be of the form indicated in Lemma \ref{preparetoapplytechnical}, 1).
Now compute 
\begin{align*} 
\begin{pmatrix} [Gk_{p}^{-1}-v^{*}G_{2,\epsilon}\ell^{-1} v, u_{n}^{\A}] & 0 \\ 0& 0\end{pmatrix}&=\left[|x\rangle \begin{pmatrix} 0 & 0 \\ 0 & G_{2,\epsilon}\end{pmatrix}[h_{p}^{-1},\langle x| ] _{},\begin{pmatrix} u_{n}^{\A} & 0 \\ 0 & u_{n}^{B}\end{pmatrix}\right] \\
&= \left[| x\rangle \tilde{G} [h_{p}^{-1},\langle x | ],\begin{pmatrix} u_{n}^{\A} 
& 0 \\ 0 & u_{n}^{B} \end{pmatrix}\right] \to 0,
\end{align*}
and the same computation works for $(w_{n}^{\J})$.
\end{proof}
\begin{thm} 
\label{thm:best-lift}
Let $A,B,C$ be separable $C^{*}$-algebras, $x\in KK(A,B)$ and 
$y\in KK(B,C)$. There exists an unbounded Kasparov module $(\B,F_{C},T)$ 
representing $y$ and a correspondence $(\A,\E_\B,S,\nabla)$ for 
$(\B,F_{C},T)$ representing $x$. Consequently
$(\A,E\hotimes_{B}F_{C}, S\otimes 1 + 1\otimes_{\nabla}T)$ 
represents the Kasparov product $x\otimes_B y$.
\end{thm}
\begin{proof} 
First represent $x$ and $y$ by essential Kasparov modules.
By Proposition \ref{singleliplift}, for any pair of essential Kasparov modules, 
the second module can be lifted such that the first module
admits a compatible complete projective submodule. Now apply Theorem \ref{Liplift}.
\end{proof}
By the same method, and lifting simultaneously with $n+1$ Kasparov modules instead of 2, 
one can prove that for classes 
$x_{n},\dots , x_{0}$ with $x_{j}\in KK(A_{j+1},A_{j})$, one can find an unbounded 
Kasparov module $(\A_1,E_{A_0}, T_{0})$ representing $x_{0}$ and correspondences 
$(\A_{j+1},\E_{\A_j},T_{j},\nabla_{j})$ representing $x_{j}$ such that for each $1\leq j\leq n$,
$(\A_{j+1},\E_{\A_j},T_{j},\nabla_{j})$ is compatible with 
\[
\left (\bigotimes_{i=1}^{j}(\A_{i+1},\E_{\A_i},T_{i},\nabla_{i})\right)\otimes (\A_1,E_{A_0},T_{0}).
\]

\appendix
\section{Weakly anticommuting operators}
\begin{defn}[cf. \cite{KaLe2}]
\label{weaklyanti} 
Let $E_B$ be a graded $C^{*}$-module and $s,t$ odd self-adjoint 
regular operators such that for all $\lambda,\mu\in \R\setminus\{0\}$:\\
1) there is a core $X$ for $t$ such that $(s\pm \lambda i)^{-1}X\subset \Dom t$;\\
2) $t(s\pm \lambda i)^{-1}X\subset\Dom s$;\\
3) $[s,t](s\pm \lambda i)^{-1}$ is bounded on $X.$

Then we say that the pair $(s,t)$ \emph{weakly anticommutes}, 
or that $t$ \emph {anticommutes weakly with} $s$. Note that this relation is not symmetric
in $s$ and $t$, and that the graded commutator is defined on $\im (s+\lambda i)^{-1}X$.
\end{defn}
It was proved in \cite{KaLe2} that the sum of weakly 
anticommuting operators occurring in odd Kasparov products is self-adjoint and regular on 
$\Dom s\cap \Dom t$. Since we are concerned here with the general graded case, a few words are in order.
\begin{lemma} 
If $(s,t)$ is a weakly anticommuting pair then the operators
$(s\pm \lambda i)^{-1}$ preserve the domain of $t$ and 
$[s,t](s\pm \lambda i)^{-1}$ is bounded on $\Dom t$. Consequently
\begin{equation}\label{commdom}s \left( (s-\lambda i)^{-1}\Dom t\right)\subset \Dom t,\quad t\left(\im (s-\lambda i)^{-1}\Dom t\right)\subset \Dom s.\end{equation}
Therefore $[s,t]$ is defined on $(s-\lambda i)^{-1}\Dom t=\im (s-\lambda i)^{-1}(t-\mu i)^{-1}$.
\end{lemma}
\begin{proof} 
For $x\in X$, the commutator $[t,(s\pm \lambda i)^{-1}]$ can  be expanded as
\[
\begin{split} 
[t,(s\pm \lambda i)^{-1}]x& =(t(s \pm \lambda i)^{-1}x+(s\mp \lambda i)^{-1}tx
=(s\mp \lambda i)^{-1}((s\mp \lambda i)t+t(s\pm \lambda i))(s\pm \lambda i)^{-1}x \\
&=(s\mp \lambda i)^{-1}[s,t](s\pm \lambda i)^{-1}x,
\end{split}
\]
and by 2) of Definition \ref{weaklyanti} this operator is bounded. 
Thus by \cite[Proposition 2.1]{FMR} $(s\pm \lambda i)^{-1}$ preserves the domain of $t$. Since $X$ is a core for $t$, for every $x\in \Dom t$ we can choose a sequence $x_{n}\in X$ such that $x_{n}\to x$ and $tx_{n}\to tx$. Then
\[\begin{split} t(s-\lambda i)^{-1}x_{n}&=-(s+\lambda i)^{-1}tx_{n}+ (s+\lambda i)^{-1}[s,t](s-\lambda i)^{-1} x_{n}\\ &\to -(s+\lambda i)^{-1}tx+ (s+\lambda i)^{-1}[s,t](s-\lambda i)^{-1} x\in\Dom s,\end{split}\]
which is a Cauchy sequence, so the limit equals $t(s-\lambda i)^{-1}x\in\Dom s$. The statement that $s \left( (s-\lambda i)^{-1}\Dom t\right)\subset \Dom t$ follows directly from the equality
\[s(s-\lambda i)^{-1}(t-\mu i)^{-1}=(t-\mu i)^{-1}+\lambda i(s-\lambda i)^{-1}(t-\mu i)^{-1}\] This proves \eqref{commdom}.
\end{proof}
\begin{lemma}
\label{epsilonlem} 
For $|\lambda | >1$ there is a constant $C$ such that 
$\|[s,t](s\pm \lambda i)^{-1}\|<C$. Thus, for $|\lambda|$ sufficiently large, 
we may assume $\epsilon:=\|(s\mp \lambda i)^{-1}[s,t](s\pm \lambda i)^{-1}\|<1$.
\end{lemma}
\begin{proof} 
Let $C:=2\|[s,t](s+i)^{-1}\|$ and write
\[
(s+\lambda i)^{-1}=(s+i)^{-1}-(s+i)^{-1}(\lambda -1)i(s+\lambda i)^{-1}.
\]
Using this,  estimate 
\[
\begin{split} 
\|[s,t](s+\lambda i)^{-1}\| &\leq \|[s,t](s+i)^{-1}-[s,t](s+i)^{-1}(\lambda -1)i(s+\lambda i)^{-1}\| \\
& \leq \|[s,t](s+i)^{-1}\|+\|[s,t](s+i)^{-1}\|\|(\lambda -1)i(s+\lambda i)^{-1}\|  \\
&\leq\frac{C}{2}\left(1+\frac{|\lambda -1|}{|\lambda|} \right) 
<C,
\end{split}
\]
since $|\lambda | >1$. Consequently, using that 
$\|(s\pm \lambda i)^{-1}\|\leq\frac{1}{|\lambda |}$, we find that
\[ 
\|(s\mp \lambda i)^{-1}[s,t](s\pm \lambda i)^{-1}\|\leq \frac{C}{|\lambda|}<1,
\]
for $|\lambda|$ sufficiently large.
\end{proof}
\begin{thm}[cf.\cite{KaLe2}] 
\label{thm:bob}
If $(s,t)$ weakly anitcommutes, then $s+t$ is closed, self-adjoint and 
regular on $\Dom s\cap\Dom t$, and $\im (s\pm i)^{-1}(t\pm i)^{-1}$ is a 
core for $s+t$. The same holds for $s-t$.
\end{thm}
\begin{proof}
It was shown in \cite{KaLe2} that the sum $s+t$ of such operators is closed, 
self-adjoint and regular on $\Dom s\cap \Dom t$ by a localisation argument. 
However, to get the statement on the core, we proceed by adapting the 
spectral argument given in \cite{Mes}. Choose $\lambda$ large enough as in Lemma \ref{epsilonlem}. The operators
\[
x:=(t+\mu i)^{-1}-(s-\lambda i)^{-1},\quad y:=x^{*}=(t-\mu i)^{-1}-(s+\lambda i)^{-1},
\]
have dense range. This follows because $xx^{*}$ is strictly positive by esitmating 
\[
\begin{split}
xx^{*}&=(\mu^{2}+t^{2})^{-1}+(\lambda^{2}+s^{2})^{-1}
-(t+\mu i)^{-1}(s+\lambda i)^{-1}-(s-\lambda i)^{-1}(t-\mu i)^{-1}\\
&=(\mu^{2}+t^{2})^{-1}+(\lambda^{2} +s^{2})^{-1}
-(t+\mu i)^{-1}(s+\lambda i)^{-1}([s,t]-2\lambda\mu) (s-\lambda i)^{-1}(t-\mu i)^{-1}\\
&=(\mu^{2}+t^{2})^{-1}+(\lambda^{2} +s^{2})^{-1}
+2\lambda\mu(t+\mu i)^{-1}(\lambda^{2}+s^{2})^{-1}(t-\mu i)^{-1}\\
&\quad\quad\quad\quad\quad\quad\quad 
-(t+\mu i)^{-1}(s+\lambda i)^{-1}[s,t](s-\lambda i)^{-1}(t-\mu i)^{-1}\\
&\geq (1-\epsilon)(\mu^{2}+t^{2})^{-1}+(\lambda^{2}+s^{2})^{-1} 
+2\lambda\mu(t+\mu i)^{-1}(\lambda^{2}+s^{2})^{-1}(t-\mu i)^{-1},
\end{split}
\]
using \ref{commdom} and Lemma \ref{epsilonlem}. 
Since the latter operator is strictly positive, so is $xx^{*}$ by 
\cite[Corollary 10.2]{Lance}. The same holds for $x^{*}x$.
The rest of the proof now follows by using the factorisations
\[
x=(s+t+(\mu-\lambda)i-(s+\lambda i)^{-1}([s,t]-2\lambda\mu))(s-\lambda i)^{-1}(t-\mu i)^{-1},
\]
\[
y=(s+t+(\lambda -\mu)i-(s-\lambda i)^{-1}([s,t]-2\lambda \mu))(s+\lambda i)^{-1}(t+\mu i)^{-1},
\]
as in \cite[Theorem 6.18]{Mes}, and applying \cite[Lemma 6.1.7]{Mes}. 
The statement for $s-t$ is obtained by observing that $(s,-t)$ weakly anticommutes.
\end{proof}

\end{document}